\documentclass[11pt]{article}

\usepackage{etoolbox}
\usepackage{authblk}
 \usepackage{orcidlink}
\usepackage{appendix}
\usepackage{comment}
\usepackage{array}
\usepackage{enumerate,amsmath,amsfonts,amsthm,amssymb,graphicx}
 \usepackage{cite}
\usepackage{tikz}\usetikzlibrary{matrix, calc, arrows, lindenmayersystems, decorations.pathreplacing, shapes.geometric}
\usepackage{hyperref} 
\usepackage{subcaption}
\usepackage[labelformat=simple]{subcaption}

 \numberwithin{equation}{section}

\date{}
\definecolor{myblue}{rgb}{0,0,0.6}     
\hypersetup{pdftitle={A DG method with fractal elements},  pdfauthor={Gomez, Hewett, Moiola},
     bookmarksopen=true, bookmarksopenlevel=4,
     colorlinks=true, linkcolor=myblue,  citecolor=myblue, filecolor=myblue,   urlcolor=myblue,  }
\hypersetup{hypertexnames=false}       
\graphicspath{{./}{figs/}}
\allowdisplaybreaks[4]

\title{A discontinuous Galerkin method with fractal elements}

\author[$*\dagger$]{S.\ G\'omez\,\orcidlink{0000-0001-9156-5135}}
\author[$\S$]{D.\ P.\ Hewett\,\orcidlink{0000-0003-3302-2567}}
\author[$\ddagger\dagger$]{A. Moiola\,\orcidlink{0000-0002-6251-4440}}
\affil[$*$]{\footnotesize Department of Mathematics and Applications, University of Milano-Bicocca, Milan, Italy}
\affil[$\dagger$]{IMATI-CNR ``E. Magenes'', Pavia, Italy}
\affil[$\S$]{\footnotesize Department of Mathematics, University College London, London, United Kingdom}
\affil[$\ddagger$]{\footnotesize  Department of Mathematics ``F. Casorati'',  University of Pavia, Pavia, Italy}

\newcommand{\rd}{\mathrm{d}}
\newcommand{\R}{\mathbb{R}}

\newcommand{\N}{\mathbb{N}}

\newcommand{\cH}{\mathcal{H}}
\newcommand{\cI}{\mathcal{I}}
\newcommand{\cT}{\mathcal{T}}

\newcommand{\bn}{\mathbf{n}}
\newcommand{\bp}{\mathbf{p}}

\newcommand{\bv}{\mathbf{v}}
\newcommand{\bx}{\mathbf{x}}
\newcommand{\by}{\mathbf{y}}

\newcommand{\dive}{\boldsymbol{\nabla}\cdot}
\newcommand{\eps}{\varepsilon}

\newtheorem{thm}{Theorem}[section]
\newtheorem{lem}[thm]{Lemma}

\newtheorem{prop}[thm]{Proposition}
\newtheorem{cor}[thm]{Corollary}
\newtheorem{rem}[thm]{Remark}

\newcommand{\dimH}{{\rm dim_H}}
\newcommand{\deO}{{\partial\Omega}}
\newcommand{\OO}{{(\Omega)}}

\newcommand{\Tr}{{\mathrm{Tr}}}
\newcommand{\cF}{{\mathcal{F}}}

\newcommand{\dn}{{\partial_\bn}}
\newcommand{\deK}{{\partial K}}

\newcommand{\cL}{{\mathcal L}}

\newcommand{\trd}{\blacktriangledown}
\newcommand{\tru}{\blacktriangle}
\newcommand{\loz}{\blacklozenge}
\renewcommand{\dive}{\mathrm{div}}
\newcommand{\DG}{_{\mathrm{DG}}}
\newcommand{\DGp}{_{\mathrm{DG}^+}}

\newcommand{\IP}{\mathbb{P}}

\newcommand{\GG}{(\Gamma)}

\newcommand{\IL}{\mathbb{L}}

\DeclareMathOperator{\diam}{diam}
\DeclareMathOperator{\dist}{dist}

\newcommand{\aDG}{a_{\rm SIP}}
\newcommand*{\jmp}[1]{[\![#1]\!]}                      
\newcommand*{\mvl}[1]{\{\!\!\{#1\}\!\!\}}              
\newcommand{\CDG}{C_{\mathrm{DG}}}
\newcommand{\CDGp}{C_{\mathrm{DG}^+}}

\DeclareMathOperator{\card}{card}
\newcommand{\Order}[1]{\mathcal{O}(#1)}

\newcommand{\gh}{g_h}
\newcommand{\ch}{c_h}

\newcommand{\ph}{p_h}

 \newcommand{\dx}{\,\mathrm{d}\bx}
\newcommand{\ds}{\,\mathrm{d}s}
\newcommand{\dhx}{\,\rd\widehat{\bx}}
\newcommand{\dhS}{\,\rd\widehat{s}}
\newcommand{\dH}{\,\rd \cH}
\newcommand{\dhH}{\,\rd \widehat{\cH}}

  \newcommand{\Th}{\cT}
\newcommand{\hatK}{\widehat{K}}
\newcommand{\hphi}{\widehat{\phi}}
\newcommand{\hK}{h_K}
\newcommand{\hF}{h_F}

\newcommand{\NT}{N_{\cT}}
\newcommand{\Np}{N_p}

 \newcommand{\Vh}{V_h}
\newcommand{\uh}{u_h}
\newcommand{\vh}{v_h}
\newcommand{\wh}{w_h}

  \newcommand{\ADG}{\mathsf{A}_{\mathrm{SIP}}}
\newcommand{\sG}{\mathsf{G}}
\newcommand{\sC}{\mathsf{C}}
\newcommand{\sP}{\mathsf{P}}
\newcommand{\sfb}{\mathsf{b}}
\newcommand{\su}{\mathsf{u}}
\newcommand{\sM}{\mathsf{M}}

\begin{document}
\maketitle
\renewcommand{\thefootnote}{\arabic{footnote}}
\begin{abstract}
\noindent We formulate, analyse, and implement a discontinuous Galerkin finite element 
method (DG-FEM) for the approximation of the solution of an elliptic boundary value problem in a domain with fractal boundary. 
We consider the case of the Poisson equation in the Koch snowflake domain with zero Dirichlet boundary conditions, but our methodology can be generalised to other cases. 
Rather than first approximating the snowflake domain by a polygonal ``prefractal''  and then applying a standard DG-FEM on the prefractal, we define a DG-FEM on the snowflake itself, using a geometry-conforming mesh (a fractal tiling) consisting of fractal elements, each similar to the original snowflake.  
Fluxes across inter-element boundaries, which are fractal curves, are represented in a weak way by integrals over element subdomains. We show how, for local polynomial basis functions, these integrals can be evaluated exactly using the similarity of the elements. 
We prove well-posedness, quasi-optimality, and convergence of the method.
We present numerical results for piecewise linear and piecewise quadratic basis functions, which demonstrate the effectiveness of the method. 
We also apply our method to the related Dirichlet eigenvalue problem. 


\end{abstract}

\section{Introduction}

  The study of PDEs on fractal domains is a growing field, motivated by numerous real-world applications in which physical boundaries and interfaces possess complicated multiscale microstructure. Recent work includes
  studies relating to the Poisson~\cite{achdou2016transmission,
 lancia2002transmission,
capitanelli2016dynamical,
capitanelli2010robin,
  bagnerini2006finite,Bagnerini13},
Helmholtz  \cite{Achdou07,bannister2022acoustic,
chandler2021boundary,
chandler2018well,caetano2025integral,caetano2024hausdorff, CaChHe25,claret2026helmholtz}, 
modified Helmholtz  \cite{claret2024convergence}, 
heat  \cite{lancia2010irregular,
lancia2012numerical,
cefalo2023fractal}, {and}
Stokes {equations}  \cite{cefalo2021approximation,
lancia2020stokes}, {as well as}
  magnetostatic problems 
\cite{creo2020magnetostatic}  and spectral problems for the Laplacian 
\cite{sapoval1991vibrations,lapidus1995fractals,Banjai:2007,grebenkov2013geometrical,
rosler2024computing,
rosler2024res}.  

\begin{figure}[t]
\centering
\begin{subfigure}[t]{0.3\textwidth}
\includegraphics[height=40mm]{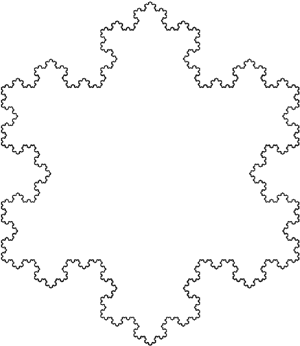}
\centering
\caption{Koch snowflake domain\label{f:koch_snowflake}}
\end{subfigure}
\begin{subfigure}[t]{0.3\textwidth}
\includegraphics[height=40mm]{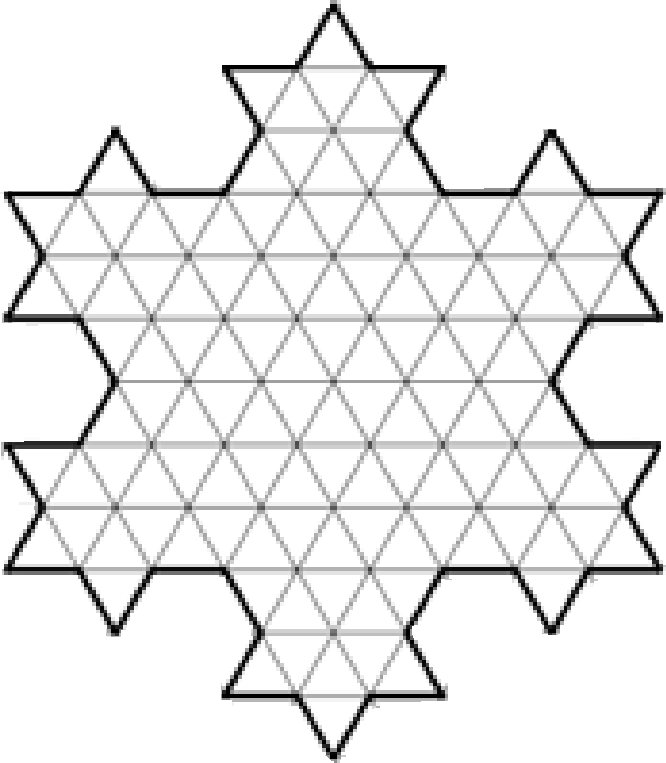}
\centering
\caption{Prefractal mesh\label{f:prefractal-Koch}}
\end{subfigure}
\begin{subfigure}[t]{0.3\textwidth}
\includegraphics[height=40mm]{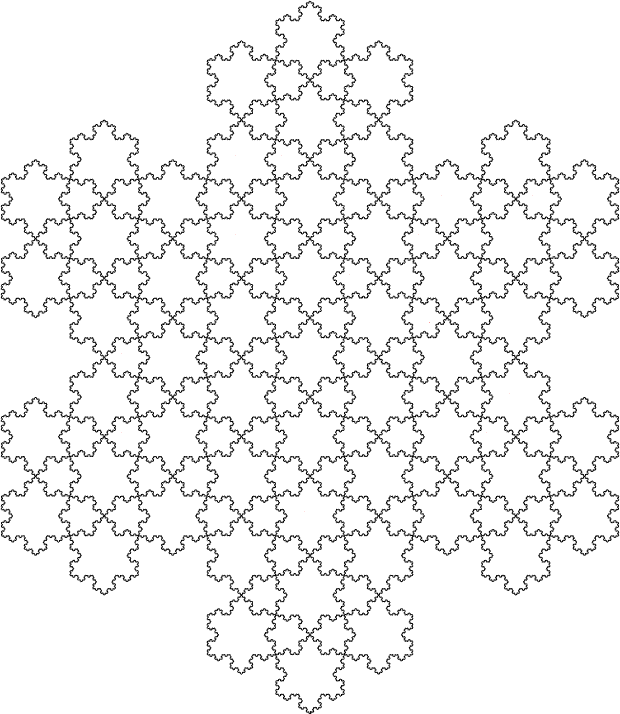}
\centering
\caption{Geometry-conforming mesh\label{f:geometry-conforming-mesh}}
\end{subfigure}
\caption{The Koch snowflake domain and illustrations of two different meshing strategies.
In this paper, we study the ``geometry-conforming'' approach (c), using a mesh with fractal elements.}
\label{f:Disc}
\end{figure}

To obtain a numerical solution to a PDE on a domain with fractal boundary such as the Koch snowflake domain pictured in Figure \ref{f:koch_snowflake}, one cannot directly apply a classical conforming finite element method (FEM) \cite{Ciarlet78,brenner2008mathematical}, because the domain cannot be meshed with a finite number of simplicial {or quadrilateral} elements (even if curvilinear). 
A natural idea (see, e.g., \cite{bagnerini2006finite}) is to first approximate the domain by a polygonal/polyhedral ``prefractal'' domain and mesh this using simplicial elements, as illustrated in Figure \ref{f:prefractal-Koch}.
However, the fractality of the domain boundary implies that obtaining accurate numerical approximations from such an approach requires a prefractal approximation with a very complex boundary. 
  This, in turn, leads to highly complex meshes with a large number of elements due to the constraints of mesh conformity and shape regularity. For an example of the extreme mesh complexity that can arise in fractal geometries with conforming FEM, see \cite{bagnerini2006finite,Bagnerini13}. 

In this paper, we present a discontinuous Galerkin FEM (DG-FEM) for the solution of a PDE on a domain with a fractal boundary, based on a radically different meshing strategy that completely captures the geometry of the domain and avoids any prefractal approximation. 
Our method uses geometry-conforming meshes, inspired by the theory of iterated function systems (IFS) and fractal tilings, consisting of a finite number of elements with fractal boundary, as illustrated in Figure~\ref{f:geometry-conforming-mesh}.  
It is well known  that DG-FEMs offer increased flexibility in mesh design compared to classical conforming FEM, permitting the use of more general element shapes and removing mesh conformity requirements. 
However, to date, the study of DG-FEM has been restricted to Lipschitz domains and Lipschitz mesh elements;  
the state of the art being the work of Cangiani et al.\ in  \cite{cangiani2022hp}, which builds on earlier related work including~\cite{cangiani2017hp}. 
To the best of our knowledge, the present work is the first {to propose and analyse a} DG-FEM for a non-Lipschitz domain using non-Lipschitz elements.

We restrict our attention here to a specific model problem, namely, the Poisson problem 
\begin{align}
\label{e:Poisson}
-\Delta u = f \quad \text{on } \Omega, \qquad \text{with }\,\, u=0 \quad \text{on } \partial\Omega,
\end{align}
 in the case where $\Omega$ is the Koch snowflake domain pictured in Figure~\ref{f:koch_snowflake}.  However, we expect that the ideas presented here can be generalised to the study of a broad class of PDEs on other fractal domains, and such extensions are currently under investigation. 

Our approach is based on the classical symmetric interior penalty (SIP) DG-FEM of~\cite{arnold1982interior}.  Using the ideas presented here one could also consider other DG-FEM formulations, but we leave this for future work. 
 \paragraph{{SIP-DG-FEM on polygonal domains.}} To provide some context, we briefly review the $h$-version SIP method for \eqref{e:Poisson}, in the case where $\Omega\subset\R^2$ is a polygonal domain meshed with a classical conforming triangulation, as in Figure~\ref{f:prefractal-Koch}.

Suppose that $\Omega\subset\R^2$ is a polygon. Let $\cT$ be a triangulation of $\Omega$ comprising a finite collection of disjoint triangular elements of maximum mesh width $h>0$, as in Figure \ref{f:prefractal-Koch}, and let $\mathcal{F}$ denote the set of all (straight-line) element edges. 
Fix $p\in \mathbb{N}$,  and let $V_h$ be the set of functions on $\Omega$ whose restriction to each $K\in \cT$ is a polynomial of total degree at most $p$.
The SIP-DG-FEM for \eqref{e:Poisson} then reads: find $u_h\in V_h$ such that
\begin{align}
\aDG(u_h,v_h) =\cL(v_h) \qquad \forall v_h\in V_h,
\label{Variational}
\end{align}
where the linear functional $\cL(v_h):=\int_\Omega fv_h \dx $, and the bilinear form
\begin{align}
\aDG(u_h,v_h):= 
\sum_{K\in \cT}\int_{K} \nabla u_h\cdot \nabla v_h \dx 
+ \sum_{F\in \cF}\int_{F} \big(\underbrace{-\{\partial_\bn u_h\}  [v_h]}_{\text{consistency}} 
\underbrace{-\{\partial_\bn v_h\} [u_h]}_{\text{symmetry}} 
\underbrace{+\alpha [u_h][v_h] }_{\text{penalty}}\big)\ds, 
\label{aDef}
\end{align}
for some positive function $\alpha$, piecewise-constant with respect to the mesh {$\cT$}.
 In~\eqref{aDef}, if $F = \partial K_+ \cap \partial K_-$ for two neighbouring elements $K_+$ and $K_-$, then {the average~$(\{\cdot\})$ and jump~$([\cdot])$ operators are defined as follows:}
\[ \{v\} := \tfrac{1}{2}(v|_{K_+} + v|_{K_-}), \qquad [v] := v|_{K_+} - v|_{K_-},\]
and, if $F$ is a boundary edge, $\{v\}$ and $[v]$ are simply the trace of~$v$ on~$F$. 
In the expression $\{\partial_\bn v\}[w]$, $\partial_\bn v$ denotes the normal derivative of $v$ on $F$ with $\bn$ being the unit normal vector pointing outward from $K_+$.
In \eqref{aDef}, the first term is an elementwise version of the standard FEM bilinear form, the second term ensures consistency with \eqref{e:BVP101} (as can be checked by integrating by parts elementwise), the third term makes $\aDG(u_h,v_h)$ symmetric with respect to $u_h$ and $v_h$ (giving optimal convergence rates in $L_2(\Omega)$ and permitting easier solution of the linear system), and the final penalty (or stabilization) term can be tuned (via $\alpha$) to make $\aDG(\cdot, \cdot)$ coercive with respect to a suitable norm, which ensures well-posedness of~\eqref{Variational}.

\paragraph{Main challenges and novelty.}
In generalising the SIP-DG-FEM \eqref{Variational}--\eqref{aDef} to a  geometry-conforming mesh of the Koch snowflake, as in Figure~\ref{f:geometry-conforming-mesh}, there are a number of challenges to overcome:
 \begin{itemize} \item
 \textbf{Boundary integrals:}\\
Given that the mesh elements $K \in \cT$ now have fractal boundaries,  the edges $F\in \mathcal{F}$ are fractal curves. So what should replace the surface measure $s$ in integrals like $\int_F [u_h][v_h] \ds$?

\item
 \textbf{Normal derivatives:}\\
Given that a unit normal does not exist on $F$, how should one define the terms involving $\partial_\bn$?

\item
 \textbf{Numerical integration:}\\
How can the resulting integrals be computed numerically?

\item
 \textbf{Penalty term:}\\
How should one choose the penalty function $\alpha$ to ensure the well-posedness of the method?   
    
\end{itemize}

To deal with these challenges, we exploit results obtained recently by two of the authors and their collaborators, relating to trace operators, approximation theory, and numerical quadrature on fractal sets \cite{caetano2025integral,
caetano2019density, caetano2024hausdorff, CaChHe25, nondisjoint, disjoint, Hewett25}. 
Regarding boundary integrals, we replace the surface measure $s$ by the $d$-dimensional Hausdorff measure $\mathcal{H}^d|_F$, where $d$ is the fractal (Hausdorff) dimension of $F$, which equals $\log{4}/\log{3}$ for the Koch snowflake.
For normal derivatives, we replace terms involving $\partial_\bn$ by weak versions, comprising the sum of integrals over certain open subsets of $K_\pm$ and integrals over certain line segments (see Figure \ref{fig:wedgesB}(a)). 
For numerical integration, we apply and extend the classical results of 
\cite{barnsley1985iterated, vrscay1989iterated,
strichartz2000evaluating}
on the evaluation of moments of self-similar measures, and the more recent related work in \cite{disjoint,nondisjoint}, to derive exact integration formulas for polynomials, and approximate quadrature rules for non-polynomial functions.
 Finally, for the penalty term, informed by trace theorems for fractal sets, combined with the scaling properties of Hausdorff measures, we choose $\alpha$ on each mesh face $F$  to be equal to a global constant $\eta>0$ (required to be sufficiently large to ensure discrete coercivity) 
 multiplied by $h_F^{-d}$, where $h_F$ is the diameter of $F$.

Having defined the SIP-DG-FEM on our fractal domain, we prove its well-posedness and quasi-optimality (\S\ref{s:Coercivity}--\S\ref{s:QO}). We also provide an $h$-convergence analysis of the method, combining known results about the regularity of the PDE solution $u$ of \eqref{e:Poisson} with polynomial approximation estimates for functions with such regularity. 
For the former, we  
report results from \cite{NystromThesis} and \cite{capitanelli2015weighted}, which give regularity estimates for $u$ in Sobolev norms on $\Omega$ weighted by a power of the distance to the boundary of $\Omega$.
The sharpness of these results is not clear, but one can view them as a fractal analogue of classical elliptic solution regularity results in curvilinear polygons (e.g., \ \cite{BGhpCurv}), where the solution singularities are concentrated on the domain corners as opposed to the full domain boundary. 
Our polynomial approximation estimates are presented in Proposition \ref{prop:ApproxDG}, and our a priori error estimates for the SIP-DG-FEM solution are presented in  Theorem~\ref{thm:Rates}.
         
To recover optimal convergence rates, we expect that the fractal mesh should be appropriately refined towards $\partial\Omega$ in order to capture the singular behaviour of the derivatives of $u$. 
While we postpone a systematic study of this issue to future work, we present numerical evidence suggesting that refining the mesh towards $\partial\Omega$ can indeed significantly improve solution accuracy.

The geometry-conforming approach of the current paper builds on related work presented recently in \cite{caetano2024hausdorff,
caetano2025integral,
CaChHe25,
bannister2026scattering}  
for the solution of boundary- and volume integral equation formulations of scattering problems for the Helmholtz equation in fractal domains, and the related quadrature rules developed in \cite{disjoint,nondisjoint,joly2024high}. 
Furthermore, the geometry-conforming meshes we consider are constructed using the self-similar structure of the snowflake, specifically the fact that $\overline{\Omega}$ is the attractor of an  IFS of contracting similarities (as defined in, e.g., \cite{hutchinson1981fractals,CaChHe25}).
As such, they are closely related to the concept of a ``fractal tiling'', as studied in, e.g., \cite{bandt1991self,grochenig1994self,strichartz1999geometry,Keesling:99,BaVi:14}.
   However, we believe that this is the first application of these techniques in the context of DG-FEM.

The structure of the paper is as follows.
\S\ref{s:BVP} introduces the Poisson--Dirichlet boundary value problems on the Koch snowflake, the relevant function spaces, and the regularity results from \cite{NystromThesis,capitanelli2015weighted}.
\S\ref{s:Meshes} uses the IFS decomposition of the Koch snowflake $\Omega$  to define different families of locally quasi-uniform meshes (including globally quasi-uniform and boundary-refined meshes), whose elements are similar copies of $\Omega$ itself.
\S\ref{s:Wedge} describes the duality integral between Neumann and Dirichlet traces on fractal faces.
\S\ref{s:DG} defines the SIP-DG-FEM and shows its discrete coercivity, from which well-posedness follows.
\S\ref{s:DGErr} carries out the SIP-DG-FEM error analysis: consistency, continuity, quasi-optimality, and $h$-convergence rates.
\S\ref{s:Quad} explains in detail how the SIP-DG-FEM linear system is assembled.
\S\ref{s:Numer} shows some numerical results for smooth and singular boundary value problems, and for the approximation of eigenproblems.
\S\ref{s:Future} lists several possible extensions.
Appendix~\ref{app:A} contains some technical results needed in the DG analysis and Appendix~\ref{app:Integration} describes how to compute the integrals entering the system matrix. 
Table~\ref{tab:Notation} collects the symbols used in the paper.

\section{The boundary value problem}
\label{s:BVP}

In this section, we define the boundary value problem (BVP) we consider, and collect known results about the regularity of its solution.

\subsection{The Koch curve \texorpdfstring{$\Gamma$}{Gamma} and the Koch snowflake \texorpdfstring{$\Omega$}{Omega}}
\label{ss:Koch}
Throughout the paper, we let $\Gamma\subset \R^2$ denote the standard Koch curve, defined as the limit of the sequence of prefractal piecewise-linear curves in Figure \ref{fig:KochCurvePrefractals}, starting from the line segment $[0,1]\times \{0\}$, so that $\diam(\Gamma)=1$. 
We also let $\Omega\subset\R^2$ denote the standard Koch snowflake, defined as the limit of the sequence of prefractals shown in Figure \ref{fig:SnowflakePrefractals}, starting from the equilateral triangle with vertices $(0,1)$, $(-\frac{\sqrt3}2,-\frac12)$, $(\frac{\sqrt3}2,-\frac12)$, so that $\diam(\Omega)=2$ and the centre of $\Omega$ is at the origin. 
As is well known, $\Gamma$ is a fractal curve with zero two-dimensional Lebesgue measure and Hausdorff dimension $d:=\dimH(\Gamma)=\frac{\log4}{\log3}$. 
The snowflake $\Omega$ is a connected open set with two-dimensional Lebesgue measure (area) 
$|\Omega| = \frac{6\sqrt{3}}{5}$, 
and its boundary $\deO$ is fractal, with $\dimH(\deO)=d$, since $\deO$ is the union of six {rotated} copies of $\Gamma$ or, alternatively, the union of three {rotated} copies of $\Gamma$, each scaled by a factor $\sqrt{3}$, see Figure~\ref{fig:SnowflakeBoundaryColor}.
 
\begin{figure}[htb]
\centering
\includegraphics[width=\textwidth]{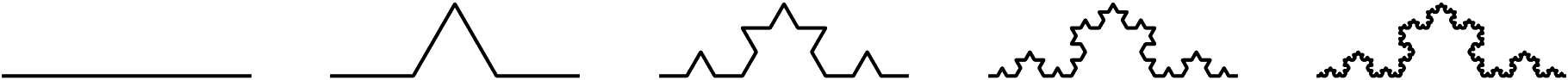}
\caption{Sequence of prefractals approximating the Koch curve $\Gamma$.}
\label{fig:KochCurvePrefractals}
\end{figure}
\begin{figure}[htb]
\centering
\includegraphics[width=\textwidth]{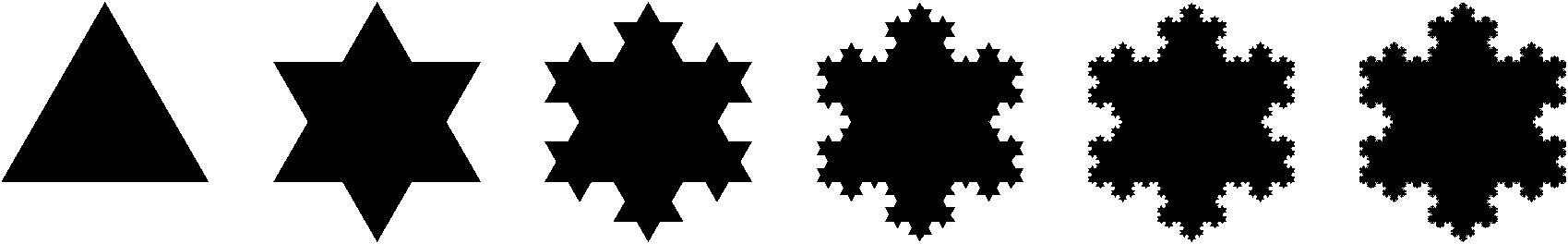}
\caption{Sequence of prefractals approximating the Koch snowflake $\Omega$.}
\label{fig:SnowflakePrefractals}
\end{figure}
\begin{figure}[htb]
\centering
\includegraphics[width=.6\textwidth]{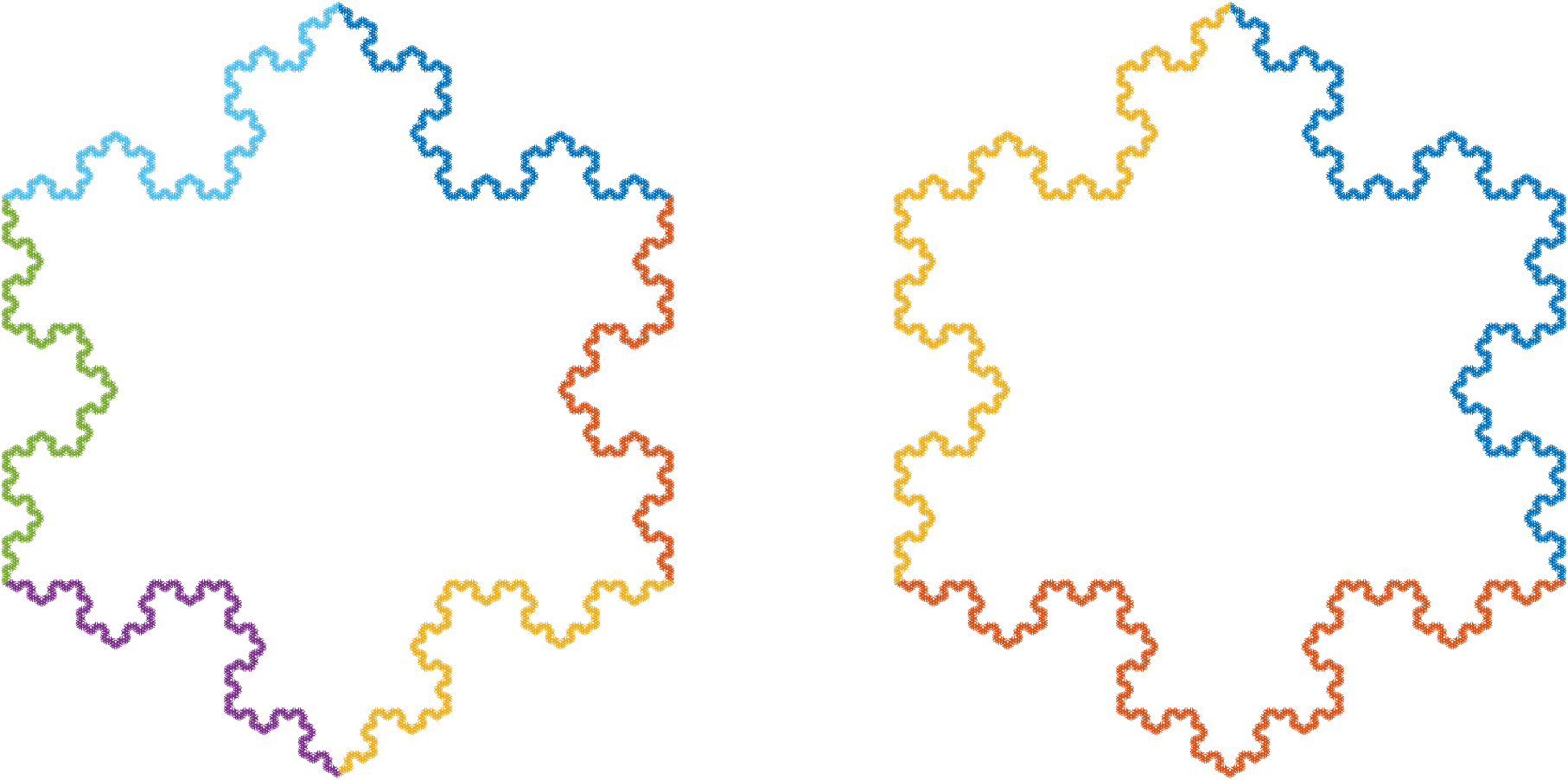}
\caption{The boundary of $\Omega$ is the union of either  {six} copies of $\Gamma$, or  {three} copies of $\Gamma$ scaled by a factor of $\sqrt{3}$.}
\label{fig:SnowflakeBoundaryColor}
\end{figure}

\subsection{The boundary value problem}\label{ss:BVP}
We consider the Dirichlet BVP: 
given $f \in L_2\OO$, find $u\in H^1_0\OO$ such that
\begin{equation}\label{e:BVP101}
-\Delta u=f \quad \text{in }\Omega.
 \end{equation}
Here, $H^1_0\OO$ is the standard Sobolev space defined as the closure of $C_0^\infty\OO$ in $H^1\OO:= \{u\in L_2\OO: \nabla u\in L_2\OO\}$, equipped with the usual norm.\footnote{Since $\Omega$ is an $H^1$ extension domain \cite[p.~73]{Jones}, $H^1(\Omega)$ coincides (with equivalent norms) with the space of restrictions to $\Omega$ of $H^1(\R^2)$ functions, equipped with the usual quotient norm \cite[Thm 3.18]{Mclean}.}
By the Lax--Milgram Theorem, this BVP is well-posed, with the solution being the unique $u\in H^1_0\OO$ such that
$\int_\Omega \nabla u\cdot \nabla v \dx  = \int_\Omega f v \dx $ for all $v\in H^1_0\OO$.

\subsection{Solution regularity: the weighted space \texorpdfstring{$H^{2,2}_\mu(\Omega)$}{H22mu}}
\label{ss:H22}

It was shown in \cite[Thms 5 and 1.3]{NystromThesis} that the solution $u$ of the BVP \eqref{e:BVP101} is H\"older continuous with $u\in C^{0,1/3}(\overline\Omega)$, and belongs to~$W^1_3\OO:=\{u\in L_3\OO: \nabla u\in L_3\OO\}$ 
 and to the
   weighted Sobolev space
\begin{align}\label{e:H22mu}  H^{2,2}_\mu\OO :=\big\{ v\in H^1\OO:\ \delta^\mu {D^2v\in L_2\OO}\big\},
\qquad D^2 v:=\big(|\partial_{x_1}^2 v|^2+|\partial_{x_1}\partial_{x_2} v|^2+|\partial_{x_2}^2 v|^2\big)^\frac12,
\end{align}
for some $\mu\in(0,1)$, where 
\begin{equation}\label{e:delta}
\delta(\bx):=\dist(\bx, \deO)\in C^{0,1}(\overline\Omega). 
\end{equation}
The space $H^{2,2}_\mu\OO$ is a Hilbert space with the norm
 $\|v\|_{H^{2,2}_\mu\OO}^2:=\|v\|_{H^1\OO}^2 +  \|\delta^\mu D^2 v\|_{L_2\OO}^2 $.
 Specifically, \cite[{Thms}~5 and 1.3]{NystromThesis}  give that $u\in H^{2,2}_\mu\OO $ 
with (for some $C_{\rm stab,\mu}>0$ depending on $\mu$)
\begin{align}
 \label{e:Stab}
 \|u\|_{H^{2,2}_\mu(\Omega)}\leq C_{\rm stab,\mu} \|f\|_{L_2(\Omega)}
 \end{align}
for all $\mu>\frac{2}{3}+\frac{1}{3}\frac{\log2}{\log3}\approx 0.877$, while \cite[{Thm}~4.1]{capitanelli2015weighted} improves this to $\mu>\mu_*\approx 0.828$.
 {Moreover, \cite[{Thm}~2]{NystromThesis} states that, if $f\in C^\infty_0(\Omega)$, $f\ge0$, and $f\ne0$, then $u\notin H^2(\Omega)$.
Our notation $H^{2,2}_\mu(\Omega)$ follows \cite{BGhpCurv}, where, for curvilinear polygons, $\delta$ is replaced by the distance from the set of vertices.

\begin{figure}[t]
\centering
\includegraphics[height=50mm]{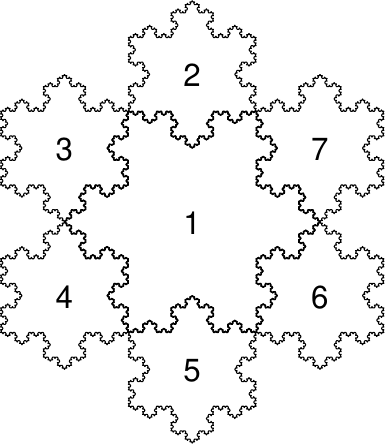}
\caption{The self-similar decomposition of the Koch snowflake into seven smaller snowflakes $s_1(\Omega),\ldots,s_7(\Omega)$ (labelled $1,\ldots,7$ in the figure), where $s_1,\ldots,s_7$ are defined in \eqref{e:smDef}.}
\label{fig:1}
\end{figure}

\section{Fractal meshes}
\label{s:Meshes}

We now consider the construction of geometry-conforming meshes of $\Omega$ comprising fractal elements. 

\subsection{Decomposition of the Koch snowflake}\label{ss:DecompOmega}
To define meshes on $\Omega$, we use the fact that $\overline{\Omega}$ is the attractor of an  IFS of contracting similarities satisfying the open set condition (see, e.g., \cite[\S5.4]{nondisjoint}),
 which implies that $\Omega$ has a decomposition into a finite union of similar copies of itself, as illustrated in Figure \ref{fig:1}. Explicitly,  
\begin{align}
\label{e:Decomp}
\overline{\Omega} = \bigcup_{m=1}^7 \overline{s_m(\Omega)},
\end{align}
where $s_1$ is a similarity with contraction factor $1/\sqrt{3}$ and $s_2,\ldots,s_7$ are similarities with contraction factor $1/3$, defined by 
 \begin{align}\nonumber
s_1(\bx)=\frac{1}{\sqrt{3}}R\bx, 
\qquad R =  \begin{pmatrix}
\frac{\sqrt{3}}{2} & -\frac{1}{2} \\
\frac{1}{2} & \frac{\sqrt{3}}{2}
\end{pmatrix} \qquad \text{(anticlockwise rotation by }\pi/6\text{)},\\
s_m(\bx)=\frac13\bx+\bx_m, \qquad 
\bx_m=\frac23\binom{\cos\alpha_{m}}{\sin\alpha_m},\qquad 
\alpha_m=\frac{(2m-1)\pi}6,\qquad m=2,\ldots,7.
\label{e:smDef}
\end{align}
This decomposition has the property that $s_m\OO \cap s_{m'}\OO=\emptyset$ for $m\neq m'$, and that, for each $m\neq 1$, the intersection $\partial(s_m\OO)\cap \partial(s_1\OO) = s_m(\partial\Omega)\cap s_1(\partial\Omega)$
 is similar to the Koch curve $\Gamma$.

\subsection{Decomposition of the Koch curve}\label{ss:DecompKoch}
We shall later use the fact that the Koch curve $\Gamma$ is also the attractor of an IFS satisfying the open set condition, and has a decomposition into a union of four similar copies of itself, with $\Gamma=\bigcup_{m=1}^4 t_m(\Gamma)$, where $t_1,\ldots,t_4$ are contractions with contraction factor $\frac13$. 
Explicitly, 
\begin{align} t_1(\bx) = \frac13 \bx, \quad t_2(\bx) = \frac13 R^2\bx + \begin{pmatrix}\frac13\\0 \end{pmatrix},
\quad  t_3(\bx) = \frac13 R^{-2}\bx + \begin{pmatrix}\frac12\\\frac1{2\sqrt3} \end{pmatrix},
\quad t_4(\bx) = \frac13 \bx + \begin{pmatrix}\frac23\\0 \end{pmatrix}, \label{e:tmDef}
\end{align}
for $\bx\in \R^2$, where $R$ is as defined above, so $R^2$ is the matrix representing an anticlockwise rotation by the angle $\frac\pi3$.

\subsection{Geometry-conforming meshes and fractal elements}\label{ss:Mesh}
Since each of the sets $s_1\OO,\ldots,s_7\OO$ is similar to $\Omega$, we can apply the above decomposition recursively to generate infinitely many different meshes of $\Omega$, each consisting of fractal elements similar to $\Omega$.
Any such mesh corresponds to a collection
\[\cT=\big\{K=s_{m_1}\circ \cdots \circ s_{m_{\ell(K)}}(\Omega): \ m_i\in\{1,\ldots,7\},\ i=1,\ldots,\ell(K)\big\},\]
\[ \text{such that }\quad\overline{\Omega}=\bigcup_{K\in \cT}\overline{K} \quad \text{and} \quad 
K\cap K'=\emptyset\ \text{ for } K,K'\in \cT, \, K\neq K',\]
where the number of contractions $\ell(K)$ involved in the definition of each element $K\in \cT$ is a positive integer that may differ from element to element.
For each $K\in\cT$, we define $h_K:=\diam {(K)}$.
From the contraction factors in \eqref{e:smDef}, each element has diameter $h_K=2/3^{j/2}$ for some $j\in\N$; if $\deK\cap\deO$ is neither empty nor a single point, then $K=s_{m_1}\circ \cdots \circ s_{m_{\ell}}(\Omega)$ with all $m_i\ne1$, so that $h_K=2/3^\ell$ for some $\ell\in\N$.

We say that two different elements $K_-,K_+\in \cT$ are \textit{adjacent} if $\overline{K_-}\cap\overline{K_+}$ is neither empty nor a single point. 
We restrict our attention to meshes that are \textit{locally quasi-uniform} (LQU), in the sense that, if $K_-,K_+\in \cT$ are adjacent, then $K_-\cup K_+ = S(s_1\OO\cup s_2\OO)$ for some similarity $S:\R^2\to\R^2$, see Figure~\ref{fig:Similarity+Ellipses}(a).
This implies that the ratio between the larger and smaller diameters of any two adjacent elements equals $\sqrt{3}$.

\begin{figure}[t]\centering
      (a)\includegraphics[width=90mm]{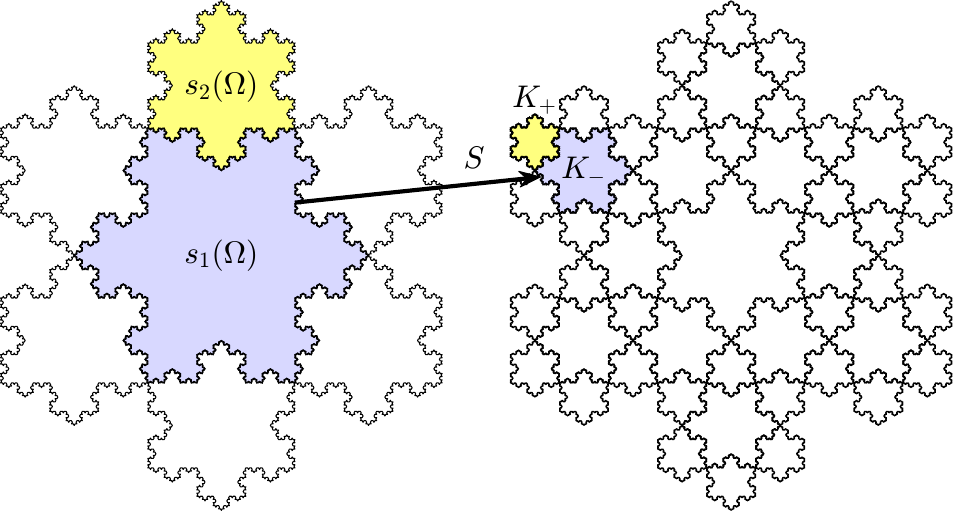} \qquad
(b)\includegraphics[width=50mm]{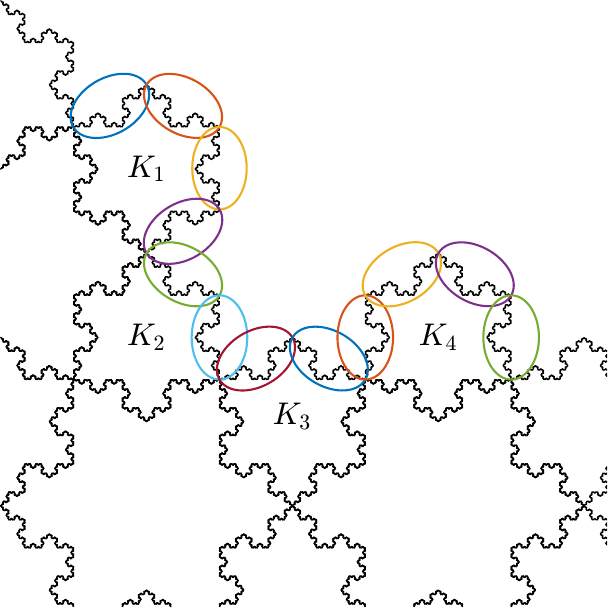}
\caption{(a) Each pair of adjacent elements $K_\pm$ is the image of $s_1\OO$ and $s_2\OO$ under a similarity~$S$.\\
(b) A zoom of the mesh in Figure~\ref{fig:MeshSeq}(b) with boundary faces highlighted: the elements $K_1$ and $K_4$ have four boundary faces each, while the elements $K_2$ and $K_3$ have two. For all boundary faces $F\subset\deK$, $h_F=h_K/2$.}
\label{fig:Similarity+Ellipses}
\end{figure}

\subsection{Faces}\label{ss:Faces}
We define the set of \textit{interior faces}
\[ \cF_I:=\big\{F=\deK_-\cap\deK_+: K_- \text{ and }K_+ \text{ are adjacent}\big\}.\]
We adopt the convention that, given a face $F=\deK_-\cap\deK_+\in\cF_I$, we choose the labels $\pm$ on $K_\pm$ so that $K_-$ denotes the larger and $K_+$ the smaller of the pair of adjacent elements associated with $F$, implying that $K_-=S(s_1(\Omega))$ and $K_+=S(s_2(\Omega))$, where $S$ is the similarity mentioned in \S\ref{ss:Mesh}, and $h_{K_-}/h_{K_+}=\sqrt{3}$.
Moreover, the face 
$F=\deK_-\cap\deK_+ = S(s_1(\partial{\Omega})\cap s_2(\partial \Omega))$ 
is similar to the Koch curve $\Gamma$, and is one third of $\partial K_+$ and one sixth of $\partial K_-$.
We define $h_F:=\diam F$, and note that
\begin{equation}\label{e:hFhK}
h_F=\frac12\,h_{K_-}=\frac{\sqrt3}2\,h_{K_+}. 
\end{equation}

We say that an element $K\in \cT$ is \textit{boundary-adjacent} if $\partial K\cap \partial\Omega$ is neither empty nor a single point. 
One can easily prove that, for every boundary-adjacent element $K$, $\partial K\cap \partial\Omega$ is the union (disjoint up to point intersections) of either 2 or 4 Koch curves of diameter $h_K/2$, i.e., 
 one sixth of $\partial K$. 
(This clearly holds for the original decomposition \eqref{e:Decomp}, and continues to hold when any boundary-adjacent element is refined.)
See Figure \ref{fig:MeshSeq}(b) for a mesh with 30 boundary-adjacent elements, of which 12 are of the first type and 18 are of the second type, and Figure~\ref{fig:Similarity+Ellipses}(b) for a zoom of the same mesh.
For each such $K$, we call these Koch curves \textit{boundary faces}, and we denote by $\cF_B$ the set of all boundary faces {of~$\cT$}.
For each~$F {\in \cF_B}$, we define $h_F:=\diam(F)$, and note that $h_F=h_K/2$.
Finally, $\cF:=\cF_I\cup\cF_B$ is the set of all interior and boundary faces of~$\cT$.

\subsection{Mesh sequences}\label{ss:MeshSeq}
We describe three families of meshes.
  
\paragraph{$\cT_\ell$ meshes.}
One sequence of LQU meshes is given by $\{\cT_\ell\}_{\ell\in\N}$, where, for $\ell\in\N$,
\[ \cT_\ell:=\big\{K=s_{m_1}\circ \cdots \circ s_{m_{\ell}}(\Omega): \;m_i\in\{1,\ldots,7\},\; i=1,\ldots,\ell\big\}.\]
The first three meshes in this family are shown in Figure~\ref{fig:MeshSeq}(a)--(c).
The mesh $\cT_1$ coincides with the decomposition \eqref{e:Decomp}, and, for $\ell\geq 2$, the mesh $\cT_\ell$ is obtained by refining each element in $\cT_{\ell-1}$ using \eqref{e:Decomp}, i.e.,\ replacing each element $K=s_{m_1}\circ \cdots \circ s_{m_{\ell-1}}(\Omega)\in \cT_{\ell-1}$ by the seven smaller elements
 $s_{m_1}\circ \cdots \circ s_{m_{\ell-1}}\circ s_j(\Omega)$ for~$j = 1, \ldots, 7$.
 As a consequence, $\cT_\ell$ is made of $7^\ell$ elements. 
That these $\cT_\ell$ meshes satisfy the LQU condition of \S\ref{ss:Mesh} can be proved by a simple induction argument, noting that the LQU condition holds for $\cT_1$ by inspection, and that it continues to hold when $\cT_{\ell-1}$ is refined to produce $\cT_{\ell}$.

\paragraph{$\cT_\ell'$ meshes.}
Another sequence of LQU  meshes can be obtained by modifying the above procedure so that, at each step,  only the elements with the largest diameter {are refined} (as opposed to all the elements being refined). We shall denote the resulting sequence of meshes by $\{\cT'_\ell\}_{\ell\in\N}$. 
Plots of the meshes  $\cT_\ell'$ for $\ell={2,3,4}$ are presented in Figure \ref{fig:MeshSeq}(d)--(f). 
Again, a simple induction argument proves that these $\cT_\ell'$ meshes satisfy the LQU condition of \S\ref{ss:Mesh}.

                \begin{figure}[htbp]\centering
(a) \raisebox{-\height}{\includegraphics[width=.28\textwidth]{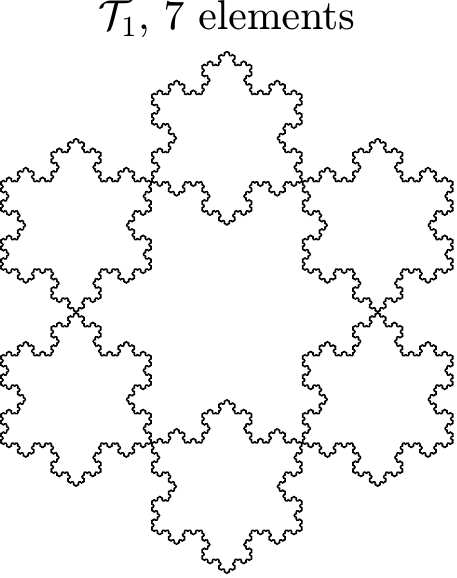}}
 (b)\raisebox{-\height}{\includegraphics[width=.28\textwidth]{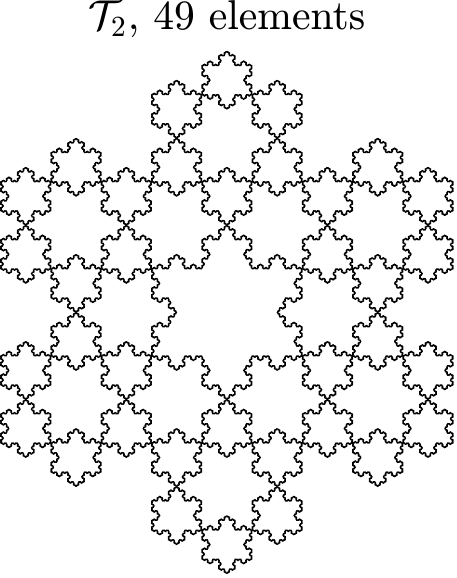}}
(c)\raisebox{-\height}{\includegraphics[width=.28\textwidth]{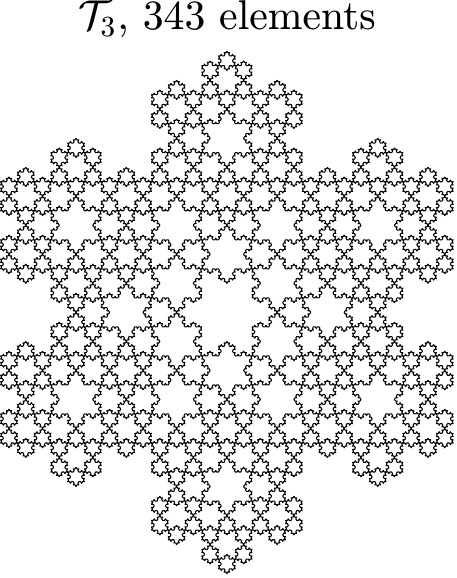}}
\\[1mm]\hrule\smallskip
(d)\raisebox{-\height}{\includegraphics[width=.28\textwidth]{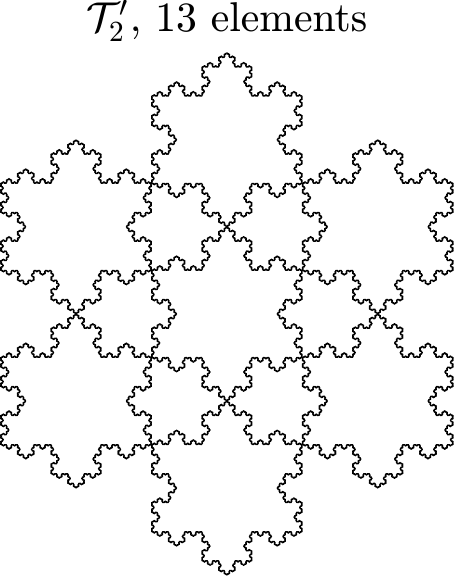}}
(e)\raisebox{-\height}{\includegraphics[width=.28\textwidth]{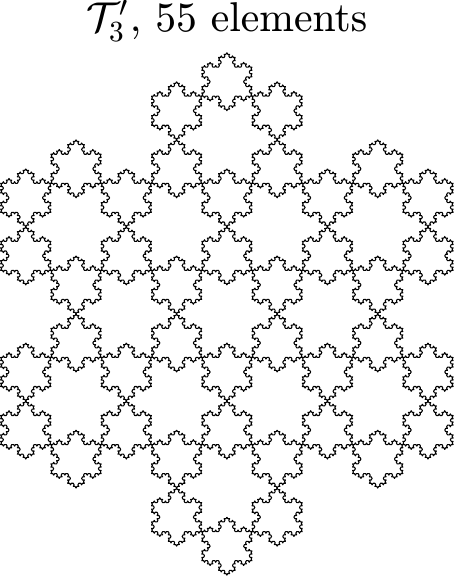}}
(f)\raisebox{-\height}{\includegraphics[width=.28\textwidth]{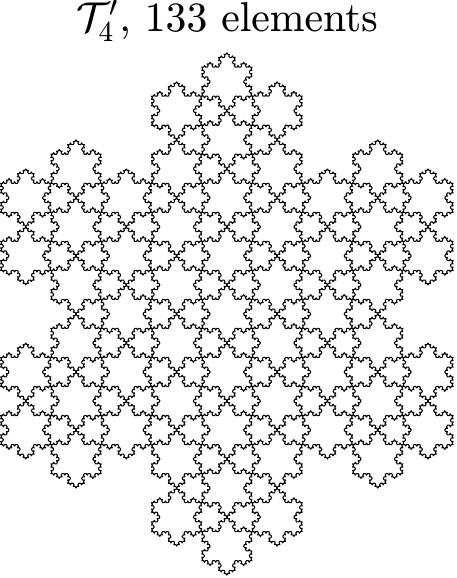}}
\\[1mm]\hrule\smallskip
(g)\raisebox{-\height}{\includegraphics[width=.28\textwidth]{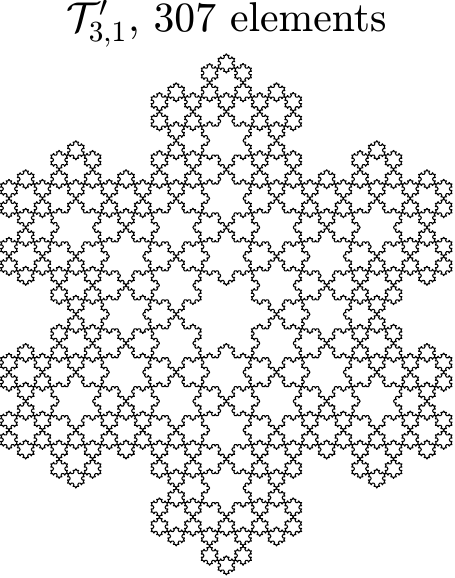}}
(h)\raisebox{-\height}{\includegraphics[width=.28\textwidth]{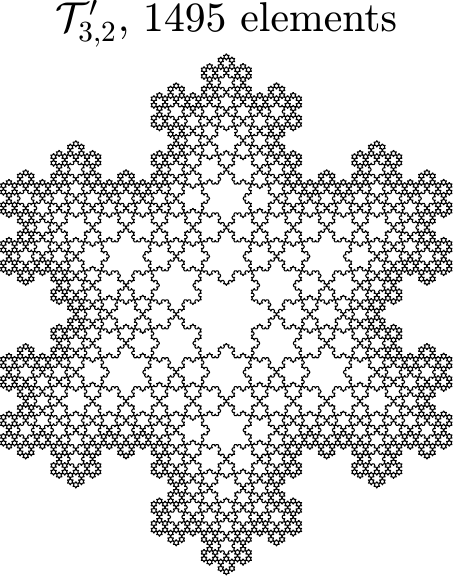}}
(i)\raisebox{-\height}{\includegraphics[width=.28\textwidth]{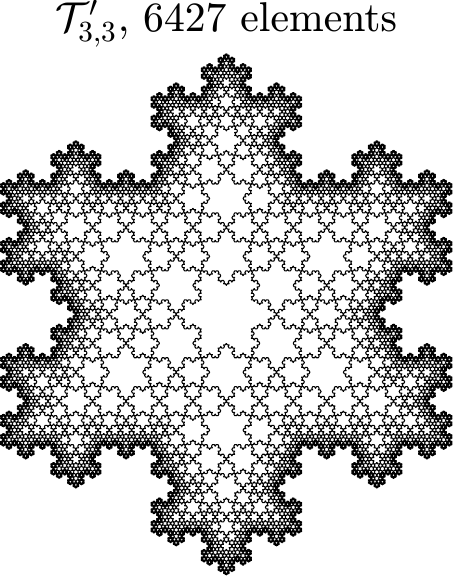}}
\caption{The meshes $\cT_\ell$ for $\ell=1,2,3$ (a--c), $\cT_\ell'$ for $\ell=2,3,4$ (d--f), and $\cT_{3,\ell^*}'$ for $\ell^*=1,2,3$ (g--i); see \S\ref{ss:MeshSeq}.}
\label{fig:MeshSeq}
\end{figure}

Note that the $\cT_\ell$ mesh includes elements of diameters $3^{-\ell}, 3^{-\ell+1/2}, 3^{-\ell+1},\ldots, 3^{-(\ell+1)/2}, 3^{-\ell/2}$ times $\diam(\Omega)$, and offers some local refinement towards $\partial\Omega$. By contrast, the mesh $\cT_\ell'$ contains elements of precisely two distinct diameters, namely $3^{-(\ell+1)/2}$ and $3^{-\ell/2}$ times $\diam(\Omega)$, and hence is globally quasi-uniform, with no refinement towards $\partial\Omega$.

\paragraph{$\cT'_{\ell,\ell^*}$ meshes.}

One can construct meshes that are quasi-uniform away from $\partial\Omega$ but locally-refined near $\partial\Omega$ by taking any one of the $\cT_\ell'$ meshes and systematically refining the elements near the boundary. However, in order to preserve the LQU property it is not sufficient to refine only the elements that touch $\partial\Omega$ --- one also needs to refine some of the elements near to but not touching $\partial\Omega$.
This can be easily seen by consulting Figure \ref{fig:MeshSeq}(d) and noting that refining all the elements of the mesh $\cT_2'$ except the central element produces a mesh which does not satisfy the LQU condition. 

To define our refinement strategy, we introduce some notation relating to the distance of a mesh element to $\partial\Omega$. 
For a mesh element $K$, denote by $\bv_1^K,\ldots,\bv_6^K$ its six ``vertices'', i.e.,\ the points of $\deK$ such that $|\bv_1^K-\bv_4^K|=|\bv_2^K-\bv_5^K|=|\bv_3^K-\bv_6^K|=h_K$, and let $\widehat\delta_K:=\min_{j=1,\ldots,6}\dist(\bv_j^K,\deO)$. 
For example, $\widehat\delta_K=h_K/2$ for the central element in the meshes $\cT_2,\cT_2'$ of Figure~\ref{fig:MeshSeq}(b) and (d).

A family $\{\cT'_{\ell,\ell^*}\}_{\ell,\ell^*\in\N_0}$ of meshes refined towards the boundary of $\Omega$ can be defined as follows.
Let $\cT'_{\ell,0}=\cT'_\ell$ be as above.
Then, for $\ell^*\in\N$, $\cT'_{\ell,\ell^*}$ is constructed from $\cT'_{\ell,\ell^*-1}$ by refining all elements $K$ such that $\widehat\delta_K\le h_K/2$.
The meshes $\cT'_{3,1}$, $\cT'_{3,2}$, and $\cT'_{3,3}$ are shown in Figure~\ref{fig:MeshSeq}(g)--(i).
We are currently unable to prove that such meshes are always LQU, but we observe by inspection that they are so for the range of parameters $\ell,\ell^*$ that we consider in our numerical experiments in \S\ref{s:Numer}, which includes those shown in Figure~\ref{fig:MeshSeq}(g--i).

The diameter of all the boundary-adjacent elements for $\cT_\ell$ is $h_K=2/3^\ell$, for $\cT_\ell'$ is $h_K=2/3^{\lceil\ell/2\rceil}$, and for $\cT'_{\ell,\ell^*}$ is $h_K=2/3^{\lceil\ell/2\rceil+\ell^*}$ (since, in~$\cT'_{\ell, \ell^*}$, the boundary-adjacent elements of $\cT_\ell'$  have been refined $\ell^*$ times).
    Table~\ref{tab:Cardinalities} shows that, for a given diameter of the boundary-adjacent elements (where the solution $u$ of \eqref{e:BVP101} is expected to be singular and a fine discretisation is needed), the $\cT'_{\ell,\ell^*}$ meshes require much fewer elements than the $\cT_\ell$ and $\cT_\ell'$ meshes.

\begin{table}[htb]\centering
\begin{tabular}{|l|ll|ll|ll|ll|ll|}\hline
$h_K$ for boundary-adjacent $K$ 
& \multicolumn{10}{c|}{Number of elements in the mesh}\\\hline
  $ h_K=2/3^4$ & $\cT_4$ &2401 & $\cT_7'$ & 4039  & $\cT'_{2,3}$ & 1567 & $\cT'_{3,2}$ & 1495 & $\cT'_{4,2}$ &1861\\
$h_K=2/3^5$ & $\cT_5$ & 16807 & $\cT_9'$ & 35839   & $\cT'_{2,4}$ & 6499 & $\cT'_{3,3}$ & 6427 & 
$\cT'_{4,3}$ & 6793
\\\hline
\end{tabular}
\caption{The number of elements in different LQU meshes with the same diameter of the boundary-adjacent elements.}
\label{tab:Cardinalities}
\end{table}

\subsection{Piecewise polynomials}\label{ss:PPoly}
For $p\in\N$ and $K\in\cT$, we denote by $\IP^p(K)$ the space of polynomials of degree at most $p$ on the mesh element $K$, and let 
\begin{equation}\label{e:Vh}
V_h:=
\IP^p(\cT ):=
\big\{v\in L_2(\Omega): v|_K\in\IP^p(K)\ \forall K\in \cT\big\}=\prod_{K\in\cT }\IP^p(K).
\end{equation}

\subsection{Hausdorff measure, Lebesgue space, trace operator}\label{ss:Traces}
We denote by $\cH^d$ the $d$-dimensional Hausdorff measure, normalised in such a way that $\cH^d\GG=1$.
Then, for each mesh face $F\in\cF $, $\cH^d(F)=h_F^d$, by the scaling properties of $\cH^d$ and the fact that $F$ is similar to $\Gamma$. 
We write $\IL_2(F)$ for the Lebesgue space of functions on $F$ that are square-integrable with respect to $\cH^d$, normed by 
 \begin{equation*} \|v\|_{\IL_2(F)}^2
:=\int_F |v(\bx)|^2\,\rd\cH^d(\bx).
  \end{equation*}
For any $F\in\cF$ and~$v\in\IL_2(F)$, if $S:\Gamma\to F$ denotes the similarity from the reference Koch curve, then the 
standard scaling property of $\cH^d$ \cite[Scaling Property 3.2]{Fal} gives that 
 \begin{equation}\label{e:ScalingHausdorff}
 \|v\|_{\IL_2(F)}^2
 =\int_F |v(\bx)|^2\,\rd\cH^d(\bx)
 =h_F^d\int_\Gamma \big|v\big(S(\hat\bx)\big)\big|^2\,\rd\cH^d(\hat\bx)
 =h_F^d\|v\circ S\|_{\IL_2(\Gamma)}^2.
 \end{equation}
By \cite{Biegert:09,Hinz}, for all elements $K\in\cT $ and all faces $F\in\cF $ with $F\subset\deK$, the trace of $H^1(K)$ functions on $F$ is well-defined as an element of $\IL_2(F)$, and 
furthermore, using the scaling \eqref{e:ScalingHausdorff} and the fact that  $h_F\le \frac{\sqrt3}2h_K$ by \eqref{e:hFhK}, there is a constant $C_\Tr>0$ independent of $h_K$ such that
\begin{equation}\label{e:TraceIneq}
\frac{1}{h_K^{d/2}}\|w\|_{\IL_2(F)}\le C_{\Tr} \big(\|\nabla w\|_{L_2(K)} + h_K^{-1}\|w\|_{L_2(K)}\big)
\qquad \forall w\in H^1(K).
\end{equation}

 \section{Wedge representation of normal derivatives}
\label{s:Wedge}

For the DG formulation, it is necessary to define, on each face $F\in \cF$, generalisations of the integrals
\begin{align}
\label{e:dnInt}
\int_F\partial_{\bn} w\, v \ds
\end{align}
that appear in the classical SIP-DG-FEM formulation \eqref{aDef} on a triangulation of a polygonal domain. The difficulty here is in properly defining ``normal derivatives'' on the faces $F\in \cF$.
Since a normal vector cannot be defined on the Koch curve, our normal derivatives cannot be defined in the classical sense.
We consider instead normal derivatives defined weakly in terms of domain integrals, following e.g., \ \cite[\S4.2]{lancia2002transmission} and \cite[\S4]{hinz2023boundary}.
These works provide a way to define weak normal derivatives on the boundary of an $H^1$ extension domain, as elements of the dual space of the trace space on the domain boundary.
This approach can be applied, in particular, for our Koch snowflake domain $\Omega$ and, consequently, for every mesh element~$K \in \cT$.
However, in order to correctly generalise integrals of the form \eqref{e:dnInt} we require weak normal derivatives restricted to each face $F\in\cF$.  For this, we introduce ``wedges'' associated with each face, each of which has a boundary consisting of the union of the face $F$ and two straight-line segments sharing a vertex at the barycentre of the element, as illustrated in Figure~\ref{fig:wedgesB}(a).
By analogy with the divergence theorem,  
the normal derivative on $F$ can then be defined weakly as a domain integral over the appropriate wedge, minus integrals involving the (classical) normal derivatives over the straight-line segments.

For the definition of the proposed DG method in \S\ref{s:DG}, it suffices to define normal derivatives only for polynomials. Hence, for now, we restrict our attention to the polynomial case. 
Later, in the analysis of the method (in \S\ref{s:DGErr}), we will extend our definitions to suitable function spaces.

Let $F=\deK_-\cap\deK_+\in\cF_I$. 
Let $S^\pm_1$ and~$S^\pm_2$ be the closed line segments joining the endpoints of $F$ to the barycentre of $K_{\pm}$, as illustrated in Figure~\ref{fig:wedgesB}(a).
Define $\vee:=S^-_1\cup S^-_2$ and 
$\wedge:=S^+_1\cup S^+_2$. 
Then $F\cup \vee$ and~$F\cup \wedge$ form the boundaries of two open subsets of $K_-$ and $K_+$, respectively, which we call \textit{wedges} and denote by $\trd$ and $\tru$.
The set $\overline{\trd\cup\tru}$ is a Lipschitz quadrilateral, the interior and the boundary of which we denote by $\loz$ and $\lozenge$, respectively. 
Whenever we refer to a normal vector $\bn$ on $\lozenge=\partial\loz$, or a subset of it, we assume that $\bn$ is the outward-pointing normal.

\begin{figure}[htb]
\centering
(a)\includegraphics[height=70mm]{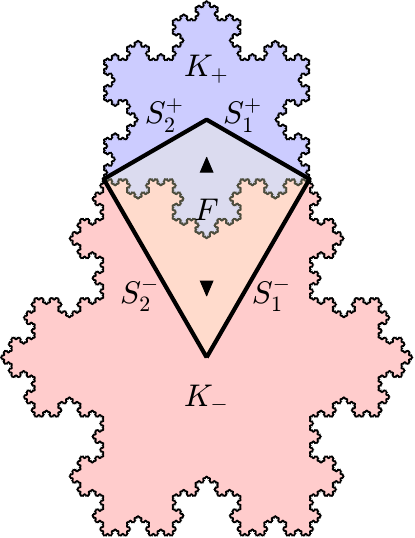}
\qquad\qquad(b)\;
\includegraphics[height=50mm]{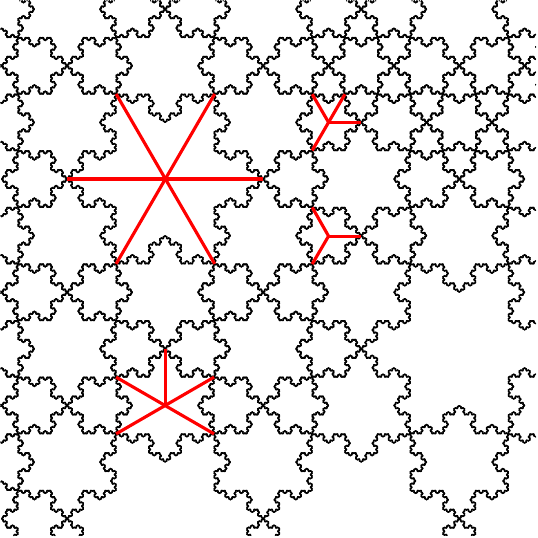}
\caption{(a) Elements $K_\pm$ and wedges $\trd,\tru$ associated with a face $F=\deK_-\cap\deK_+\in \cF_I$.
The quadrilateral $\loz$ is the interior of the set $\overline{\trd\cup\tru}$, and its boundary $\lozenge$ is the thick black line.\\
(b) Zoom of an LQU  mesh showing division of elements into 3, 4, 5, and 6 wedges.
The union of the (3, 4, 5, or 6) red segments contained in $K\in\cT$ is denoted $\ast_K$ in \S\ref{s:Continuity}.
}
\label{fig:wedgesB}
\end{figure}

 The angle between the segments $S_1^-$ and $S_2^-$ is $\pi/3$, so $\trd$ takes one sixth of $K_-$;
conversely, the angle between the segments $S_1^+$ and $S_2^+$ is $2\pi/3$, so $\tru$ takes one third of $K_+$.
Therefore, depending on the relative sizes of its adjacent elements, a given element $K$ can be split in 3, 4, 5, or 6 wedges, as illustrated in Figure~\ref{fig:wedgesB}(b).

Given two polynomials $w,v\in\IP^p(\R^2)$ and a face $F=\deK_-\cap\deK_+\in\cF_I$, we define the bilinear forms
\begin{equation}\label{e:IdownIup}
\cI_\trd(w,v):=\int_\trd\dive[\nabla w\, v]\dx  -\int_\vee\dn w\, v \ds ,\quad\qquad
\cI_\tru(w,v):=\int_\tru\dive[\nabla w\, v]\dx  -\int_\wedge\dn w\, v \ds .
\end{equation}
Clearly, $\cI_\trd(w,v)$ and $\cI_\tru(w,v)$ are well-defined because all integrands are polynomials, and the domains of integration are a bounded open subset of $\R^2$ (either $\trd$ or $\tru$) and the union of two segments (either $\vee$ or $\wedge$). 
We will extend the definition of $\cI_{\trd}(w,v)$ and $\cI_\tru(w,v)$ to sufficiently regular functions $w,v$ in \S\ref{s:DGErr}.

 The motivation for the definition of $\cI_\trd$ and $\cI_\tru$ is the following: if $F$ were a Lipschitz curve, then, 
by the divergence theorem, we would have $\cI_\trd(w,v)=\int_F\partial_{\bn_\trd} w\, v \ds $ and
$\cI_\tru(w,v)=\int_F\partial_{\bn_\tru} w\, v \ds $, with $\bn_\trd$ and $\bn_\tru$ pointing outwards from $\trd$ and $\tru$, respectively.
 Thus $\cI_\trd(w,v)$ and $\cI_\tru(w,v)$ are proxies for the duality products (the integrals of the product) of the Neumann trace of $w$ and the Dirichlet trace of $v$ on the nowhere-differentiable curve $F$.    
By the divergence theorem, the following identity holds:
\begin{equation}\label{e:II0poly}
\cI_\trd(w, v)+\cI_\tru(w,v)
=\int_\loz \dive[\nabla w\, v] \dx  - \int_\lozenge \dn w\, v \ds  =0
\qquad \forall w,v\in 
\IP^p(\R^2),
\end{equation} 
with $\bn$ pointing outwards from $\loz$, which is a proxy for the fact that the Dirichlet and Neumann traces of globally smooth functions coincide when taken from opposite sides of the same face.

For any polynomial $w \in\IP^p(\R^2)$ and any face $F\in\cF $, $\|w\|_{\IL_2(F)}=0$ if and only if $w=0$, since non-trivial polynomials in $\R^2$ cannot vanish $\cH^d$-almost everywhere on a set of Hausdorff dimension $d>1$, see \cite[Prop.~2]{Mityagin20}.
Thus, since $\IP^p(\R^2)$ is finite-dimensional, any norm on $\IP^p(\R^2)$ is equivalent to $\|\cdot\|_{\IL_2(F)}$. 
Similarly, any seminorm on $\IP^p(\R^2)$ whose kernel is the space of constants is equivalent to $\|\nabla(\cdot)\|_{L_2(\trd)}$. 
             This, together with the Cauchy--Schwarz inequality and the norm scaling \eqref{e:ScalingHausdorff},
 gives the existence of a constant $C_p>0$, depending only on the polynomial degree $p$, such that
      \begin{subequations}\label{e:DiscreteBoundOnI}
\begin{align}
&|\cI_\trd(w,v)|\le \|\nabla w\|_{L_2(\trd)} \|\nabla v\|_{L_2(\trd)} + h_F \|\Delta w\|_{L_2(\trd)} h_F^{-1}\|v\|_{L_2(\trd)}
+h_F^{1/2}\|\dn w\|_{L_2(\vee)} h_F^{-1/2}\|v\|_{L_2(\vee)}
\nonumber\\
\le&\Big(\|\nabla w\|_{L_2(\trd)}^2+h_F^2 \|\Delta w\|_{L_2(\trd)}^2 + h_F\|\dn w\|_{L_2(\vee)}^2\Big)^{1/2}
\Big(\|\nabla v\|_{L_2(\trd)}^2 +  h_F^{-2}\|v\|_{L_2(\trd)}^2+ h_F^{-1}\|v\|_{L_2(\vee)}^2\Big)^{1/2}
\nonumber\\
   \le& \,C_p\|\nabla w\|_{L_2(\trd)}\ h_F^{-d/2}\|v\|_{\IL_2(F)}
   \qquad\qquad \forall w,v\in\IP^p(\R^2),  \label{e:DiscreteBoundOnI1}
\end{align}
and, similarly,
\begin{align}
|\cI_\tru(w,v)|
\le C_p\|\nabla w\|_{L_2(\tru)}\ h_F^{-d/2}\|v\|_{\IL_2(F)}
\qquad\qquad \forall w,v\in\IP^p(\R^2). 
\label{e:DiscreteBoundOnI2}
\end{align}
\end{subequations}

Analogously, for all boundary faces $F\in\cF_B$, if $K\in \cT$ is the element such that $F\subset\deK$, we define $S^-_1,S^-_2$ to be the closed segments connecting the endpoints of $F$ to the barycentre of $K$, $\vee:=S^-_1\cup S^-_2$, and the wedge $\trd$ to be the bounded open set with boundary $\vee\cup F$ (see Figure~\ref{fig:BdryTria} below). 
We then define $\cI_\trd(w,v)$ as in \eqref{e:IdownIup}.

The wedges associated  with the mesh faces are a partition of $\Omega$.
In particular, the following identity holds (with all unions disjoint):
\begin{equation}\label{e:WedgePartition}
\overline\Omega=
\overline{\bigg(\bigcup_{F\in\cF_I\cup\cF_B}\trd \bigg) \cup\bigg(\bigcup_{F\in\cF_I}\tru\bigg)}.
 \end{equation}

\section{The SIP-DG method}
 \label{s:DG}

We are now ready to define our SIP-DG-FEM on an LQU mesh $\cT_h$ of the fractal domain $\Omega$. 

Here, and henceforth, for each internal mesh face~$F = \deK_- \cap \deK_+$ and any function~$\vh$ in the piecewise-polynomial space~$\Vh$ defined in~\eqref{e:Vh}, we denote by $v_h^\pm$ the global polynomials in $\IP^p(\R^2)$ such that $v_h|_{K_\pm} =v_h^\pm|_{K_\pm}$.
We also recall our convention that $\trd\subset K_-$, $\tru\subset K_+$ (Figure~\ref{fig:wedgesB}(a)).

Let $\eta$ be a positive penalty parameter.
Our SIP-DG-FEM bilinear and linear forms are defined for $w_h,v_h\in V_h$ by
\begin{align}
\label{e:aDefDisc}
\aDG(w_h,v_h):=&\sum_{K\in\cT }\int_K \nabla w_h\cdot\nabla v_h \dx \\
&+\sum_{F\in\cF_I} \bigg[ -\frac12\Big{(}  \cI_\trd(w_h^-,v_h^--v_h^+)
  -\cI_\tru(w_h^+,v_h^--v_h^+) &&\bigg\}\text{consistency}\notag\\
&\hspace{20mm}
+\cI_\trd(v_h^-,w_h^--w_h^+)-\cI_\tru(v_h^+,w_h^--w_h^+) \Big)  \quad&&\bigg\}\text{symmetry}\notag\\
&
 \qquad \qquad+\frac\eta{h_F^d}\int_F(w_h^--w_h^+)(v_h^--v_h^+)\rd\cH^d \bigg] &&\bigg\}\text{penalty}\notag\\
&+\sum_{F\in\cF_B} \bigg[-\cI_\trd(w_h,v_h)-\cI_\trd(v_h,w_h)+\frac\eta{h_F^d}\int_F w_h\,v_h\,\rd\cH^d\bigg],&&\bigg\}\text{boundary}
\notag\\
\cL(v_h):=&\int_\Omega fv_h \dx .
\notag
\end{align}
Here we recall from \eqref{e:IdownIup} that, for a given internal face $F\in\cF_I$, the term $\cI_\trd(w_h^-,v_h^--v_h^+)$ appearing in \eqref{e:aDefDisc} involves integrals 
$\int_{\trd}\dive[\nabla w_h^-\, (v_h^--v_h^+)]\dx$ and $\int_\vee\dn w_h^-\, (v_h^--v_h^+) \ds$ over the open set $\trd$ and the union of line segments $\vee$, both of which are subsets of $K_-$. 
In accordance with the notation introduced just before \eqref{e:aDefDisc}, in these integrals, $w_h^-$ and $v_h^-$ are just the restrictions of $w_h$ and $v_h$ to $K_-$, while $v_h^+$ is the extension to $K_-$ of the restriction of $v_h$ to the neighbouring element $K_+$. Similar comments apply to the terms $\cI_\tru(w_h^+,v_h^--v_h^+)$, $\cI_\trd(v_h^-,w_h^--w_h^+)$, and $\cI_\tru(v_h^+,w_h^--w_h^+)$.

Our SIP-DG-FEM is then defined by the following discrete variational problem:
\begin{equation}\label{e:DG}
\text{find }u\DG\in V_h \quad \text{s.t.}\quad \aDG(u\DG,v_h)=\cL(v_h) \quad \forall v_h\in V_h.
\end{equation}

Comparing \eqref{e:aDefDisc}--\eqref{e:DG} with the classical SIP-DG-FEM formulation \eqref{aDef}, the ``consistency'' and the ``symmetry'' terms in $\aDG(\cdot,\cdot)$ are proxies for the standard jump and average terms in \eqref{aDef}.
In particular, $\cI_\trd(w_h^-,v_h^-)-\cI_\trd(w_h^-,v_h^+)$ is a proxy for $\int_F\dn w_h^-(v_h^--v_h^+)\ds $, so the ``consistency'' term in $\aDG(\cdot,\cdot)$ is a proxy for $\int_F\mvl{\nabla w_h} \cdot \jmp{v_h} \ds $ in the notation of \cite[p.~19]{cangiani2017hp}, i.e.,\ the product of the average of the gradient of the trial field $w_h$ and the normal jump of the Dirichlet trace of the test field $v_h$. 
The ``symmetry'' term in $\aDG(\cdot,\cdot)$ can be interpreted analogously, with the roles of $w_h$ and $v_h$ swapped.
The ``penalty'' term in $\aDG(\cdot,\cdot)$, which is a proxy for the corresponding term in \eqref{aDef}, does not require the generalised integrals $\cI_\trd,\cI_\tru$ because it involves only Dirichlet traces, which are well-defined by \S\ref{ss:Traces}.

Details of the practical implementation of \eqref{e:DG} will be discussed in \S\ref{s:Quad}.
We first present results relating to the analysis of the method, beginning with discrete coercivity.

\subsection{Discrete coercivity}
\label{s:Coercivity}
For the analysis of the method, we define the following DG norm on $V_h$: 
\begin{align*}
\|v\|\DG^2:=&\sum_{K\in\cT } \|\nabla v\|_{L_2(K)}^2
+\sum_{F\in\cF_I}\frac{1}{h_F^d}\int_F(v^--v^+)^2\rd\cH^d
+\sum_{F\in\cF_B}\frac{1}{h_F^d}\int_F v^2\rd\cH^d \qquad \forall v\in V_h.
\end{align*}
This is a norm on $V_h$ because, if
$\|v\|\DG=0$, then $v$ is constant on each element, with value zero on $\deO$ and vanishing jumps on all internal faces.

We then have the following discrete coercivity and well-posedness result.

\begin{thm}[Coercivity of~$\aDG$ and well-posedness of the SIP-DG-FEM]\label{thm:Coercivity} \ \
Let $\cT_h$ be an LQU mesh of $\Omega$, let $p \in \N$, and let $\Vh$ be defined as in \eqref{e:Vh}.
If the penalty parameter $\eta$ in \eqref{e:aDefDisc} satisfies
\begin{equation}\label{e:Eta}
\eta\ge \frac12+ 2C_p^2,
\end{equation}
with $C_p$ the constant from  \eqref{e:DiscreteBoundOnI}, then the bilinear form $\aDG(\cdot,\cdot)$ satisfies the coercivity inequality
\begin{equation}\label{e:Coercivity}
\aDG(w_h,w_h)\ge \frac12\|w_h\|\DG^2 \qquad \forall w_h\in V_h,
\end{equation}
and the SIP-DG-FEM scheme \eqref{e:DG} is well-posed for all~$f \in L_2(\Omega)$.
Consequently, given any choice of basis for $V_h$, the resulting Galerkin matrix is symmetric and positive definite.
\end{thm}
\begin{proof}
Using the definition \eqref{e:aDefDisc} of the bilinear form $\aDG(\cdot,\cdot)$, the bilinearity of $\cI_\trd$ and $\cI_\tru$, the continuity identity \eqref{e:II0poly}, the inverse inequality \eqref{e:DiscreteBoundOnI}, and the weighted Young inequality ($2ab\leq \eps a^2 + (1/\eps)b^2$ for $a,b\in\R$ and $\eps>0$), it follows that, for all piecewise polynomials $w_h\in V_h$ and arbitrary $\eps>0$,
\begin{align*}
\aDG(w_h,w_h)
   & =\sum_{K\in\cT }\|\nabla w_h\|_{L_2(K)}^2
+\sum_{F\in\cF_I}\frac\eta{h_F^d}\|w_h^--w_h^+\|_{\IL_2(F)}^2
+\sum_{F\in\cF_B} \frac\eta{h_F^d}\|w_h\|_{\IL_2(F)}^2\\
&
\quad\  +\sum_{F\in\cF_I}\Big[ -\cI_\trd(w_h^-,w_h^--w_h^+)+\cI_\tru(w_h^+,w_h^--w_h^+)\Big]
-2\sum_{F\in\cF_B} \cI_\trd(w_h,w_h)
\\
& \overset{\eqref{e:DiscreteBoundOnI}}\ge
\sum_{K\in\cT }\|\nabla w_h\|_{L_2(K)}^2
+\sum_{F\in\cF_I}\frac\eta{h_F^d}\|w_h^--w_h^+\|_{\IL_2(F)}^2
+\sum_{F\in\cF_B} \frac\eta{h_F^d}\|w_h\|_{\IL_2(F)}^2\\
& \quad\
-C_p\sum_{F\in\cF_I} 
\big(\|\nabla w_h\|_{L_2(\trd)}+\|\nabla w_h\|_{L_2(\tru)}\big) 
h_F^{-d/2}\|w_h^+-w_h^-\|_{\IL_2(F)}\\
& \quad\ -2C_p\sum_{F\in\cF_B} \|\nabla w_h\|_{L_2(\trd)} h_F^{-d/2}\|w_h\|_{\IL_2(F)}
\\
& \ge\;\;
\sum_{K\in\cT }(1-{\varepsilon})\|\nabla w_h\|_{L_2(K)}^2
+\sum_{F\in\cF_I}\frac{\eta-\frac12C_p^2{\varepsilon^{-1}}}{h_F^d}\|w_h^--w_h^+\|_{\IL_2(F)}^2 
\\& \quad\ 
+\sum_{F\in\cF_B} \frac{\eta-C_p^2{\varepsilon^{-1}}}{h_F^d}\|w_h\|_{\IL_2(F)}^2,
       \end{align*}
where, in the last step, we used the partition \eqref{e:WedgePartition} of $\Omega$ in wedges.
  Choosing ${\varepsilon} =\frac12$, and assuming 
\eqref{e:Eta}, 
we obtain \eqref{e:Coercivity}.
\end{proof}

\section{SIP-DG-FEM error analysis}
 \label{s:DGErr}

For the error analysis of our method, we follow the framework of \cite[\S1.3 \& \S4.2]{di2011mathematical}. 
 In Theorem~\ref{thm:Coercivity}, we have already proven that the bilinear form~$\aDG(\cdot, \cdot)$ introduced in~\eqref{aDef} is coercive. 
Our aim is now to prove quasi-optimality of the method.

Recalling the weighted Sobolev space in~\eqref{e:H22mu}, given~$\mu \in [0, 1)$, we define\footnote{Note that we could take instead $V^\mu\cap H^1_0(\Omega)$, in analogy to \cite[Ass.~4.4]{di2011mathematical}, with no changes in the subsequent analysis.}
    \begin{align}\label{e:V}
V^\mu   := \big\{w\in H^{2,2}_\mu(\Omega)\ : \   \Delta w\in L_2(\Omega)\big\}.
\end{align}
 For~$0 \le \mu_- < \mu_+ < 1$, the inclusion~$V^{\mu_-}\subset V^{\mu_+}$ holds; moreover, $V^0 = H^2\OO$.
For all open $D$ compactly contained in $\Omega$ and all $w\in V^\mu$, $w|_D\in H^2(D)\subset C^0(D)$ by \eqref{e:H22mu} and the embedding \cite[eq.~(1,4,4,6)]{Grisvard85}, thus $V^\mu\subset C^0(\Omega)$.
Lemma \ref{lem:Embed} implies that $V^\mu\subset C^0(\overline{\Omega})$ for $0\leq \mu<\frac{d}{2}\approx 0.369$; whether this extends to larger $\mu$ is not known to us.

In order to prove quasi-optimality,  
 which we do in Theorem \ref{thm:QO} below, it remains to extend the definition of  $\aDG(\cdot,\cdot)$ to  
$(V^\mu+V_h)\times (V^\mu+V_h)$ (\S\ref{s:aDGDef}), 
show that it is continuous with respect to suitable norms (\S\ref{s:Continuity}), 
and prove the consistency of the discrete variational problem~\eqref{e:DG}, namely that
\begin{align}
\label{e:Cons}
\aDG(u,v_h)=\cL(v_h) \qquad \forall v_h\in V_h,
\end{align}
where $u$ is the solution of \eqref{e:BVP101} (\S\ref{s:Consistency}).
 
\subsection{Defining the bilinear form}
 \label{s:aDGDef}

For $w\in V^\mu$ and $\mu\in[0,1)$, Lemma~\ref{lem:dnu2} in Appendix \ref{app:A} ensures that, for any $F\in\cF$, the trace of $w$ on $\vee$ lies in $L_q(\vee)$ for all $q\in [1,\infty]$, and the normal derivative trace of $w$ lies in $L_q(\vee)$ for all
\begin{equation}\label{e:qmu}
q\in [1,1/\mu) \quad \text{if} \quad 1/2\leq \mu<1, \qquad \text{and} \qquad 
q\in [1,2] \quad \text{if} \quad 0\leq \mu< 1/2.
\end{equation}
The same is true with $\vee$ replaced by $\wedge$ when $F\in \cF_I$.
Hence $\cI_\trd(w,v)$ (for $F\in \cF$) and $\cI_\tru(w,v)$ (for $F\in \cF_I$) are well-defined by \eqref{e:IdownIup} whenever $w,v\in V^\mu+V_h$. 

To define $\aDG(\cdot,\cdot)$ on $(V^\mu+V_h)\times (V^\mu+V_h)$, we use a simple property of $V^\mu+V_h$.
\begin{lem}\label{lem:Jumps}
Let $\mu\in[0,1)$, $w=\widetilde w+w_h\in V^\mu+V_h$ and let $K_\pm$  be two elements sharing a face $F\in\cF_I$.
The jump $w_h^--w_h^+$ of the polynomial part of $w$ is independent of the choice of decomposition $w=\widetilde w+w_h$.
\end{lem}
\begin{proof}
Let $w=\widetilde w_1+w_{1,h}=\widetilde w_2+w_{2,h}$ be two decompositions of $w$.
Then $\widetilde w_1-\widetilde w_2=w_{2,h}-w_{1,h}\in V^\mu\cap V_h$. 
Since $V^\mu\subset C^0(\Omega)$, $\widetilde w_1-\widetilde w_2$ has a uniquely-defined trace on $F$ (except, possibly, at the isolated points $F\cap\deO$).
Thus $w_{2,h}^--w_{1,h}^-=w_{2,h}^+-w_{1,h}^+$, i.e.,\ $w_{1,h}^--w_{1,h}^+=w_{2,h}^--w_{2,h}^+$ on $F$.
By our observation after \eqref{e:II0poly} that if the trace of a polynomial $(w_{1,h}^--w_{1,h}^+)-(w_{2,h}^--w_{2,h}^+)\in\IP^p(\R^2)$ on a set of Hausdorff dimension $d$ vanishes, then the polynomial itself vanishes. This
implies that the jumps of the two discrete functions coincide as polynomials in $\IP^p(\R^2)$. 
 \end{proof}

Given $w=\widetilde{w}+w_h\in V^\mu+V_h$ and $v=\widetilde{v}+v_h\in V^\mu+V_h$, we define 
\begin{align}
\label{e:aDefGen}
\aDG(w,v):=&\sum_{K\in\cT }\int_K \nabla w\cdot\nabla v \dx  \\\notag
&+\sum_{F\in\cF_I}{\bigg[} -\frac12 {\Big(}  \cI_\trd(w,v_h^--v_h^+)-\cI_\tru(w,v_h^--v_h^+) &&\bigg\}\text{consistency}\\\notag
&\hspace{20mm}
+\cI_\trd(v,w_h^--w_h^+)-\cI_\tru(v,w_h^--w_h^+) {\Big)} \quad&&\bigg\}\text{symmetry}\\\notag
&
 \qquad \qquad+\frac\eta{h_F^d}\int_F(w_h^--w_h^+)(v_h^--v_h^+)\rd\cH^d {\bigg]} &&\bigg\}\text{penalty}\\\notag
&+\sum_{F\in\cF_B} \bigg[-\cI_\trd(w,v)-\cI_\trd(v,w)+\frac\eta{h_F^d}\int_F w\,v\,\rd\cH^d\bigg].&&\bigg\}\text{boundary}
 \end{align}
Note that the definition of $\aDG(\cdot,\cdot)$ does not depend on the choice of decompositions $w=\widetilde{w}+w_h$ and $v=\widetilde{v}+v_h$, thanks to Lemma~\ref{lem:Jumps}. 
Note also that 
 the term $\cI_\trd(w,v_h^--v_h^+)$ appearing in \eqref{e:aDefGen} involves integrals over $\trd$ and $\vee$, both subsets of $K_-$.  
Here, $w$ and~$v_h^-$ are the standard restrictions of~$w$ and~$\vh$ to~$K_-$, whereas~$\vh^+$ is defined on~$K_-$ by extending the restriction~$v_h|_{K^+}$  across the face~$F$ as a polynomial.
We emphasize that only the polynomial part $v_h\in V_h$ of $v=\widetilde{v}+v_h\in V^\mu+V_h$ is being extended across $F$; the continuous part $\widetilde v\in V^\mu$ does not need to be extended.  
Similar comments apply to the terms $\cI_\tru(w,v_h^--v_h^+)$, $\cI_\trd(v,w_h^--w_h^+)$ and $\cI_\tru(v,w_h^--w_h^+)$.

\subsection{Consistency}
\label{s:Consistency}

\begin{lem}[Consistency]\label{lem:Consistency}
Let $\mu\in[0,1)$, and suppose that the solution $u\in H^1_0(\Omega)$ of the BVP \eqref{e:BVP101} satisfies $u\in V^\mu$. 
 Then \eqref{e:Cons} holds.
\end{lem}
\begin{proof}
We first note that, if $w\in V^\mu$, we can take $w_h=0$ in \eqref{e:aDefGen} so that the symmetry and penalty terms in \eqref{e:aDefGen} vanish. 

 If, additionally, $w\in H^1_0(\Omega)$ and $v=v_h\in V_h$, then the second and third boundary terms in \eqref{e:aDefGen} also vanish. 
In more detail, the third boundary term $\int_F w v_h \,\rd \cH^d$ vanishes for $w\in H^1_0(\Omega)$ because $H^1_0(\Omega)$ is the kernel of the trace operator onto~$\partial\Omega$ \cite[Thm~3.5]{farkas2001sobolev} (see also \cite[\S4]{Hinz} for more general related results). 
As for the second boundary term $\cI_\trd(v_h,w)$, we first fix a boundary face $F\in \cF_B$. 
Let $v_h=v_p$ on $\trd$, where $v_p\in \mathbb{P}^p(\R^2)$.
Let $\tru\subset\R^2\setminus \Omega$ be the region
shown in Figure~\ref{fig:BdryTria}, which is  constructed analogously as in Figure \ref{fig:wedgesB}(a) for internal faces. 
The fact that $\tru\subset\R^2\setminus \Omega$ can be justified by an induction on the mesh refinement, having verified it by hand on the mesh $\cT_1$. 
Let $\wedge$, $\cI_\tru$, $\loz$, and $\lozenge$ be defined as for internal faces. 
We then note that $w\in H^1_0(\Omega)$ implies that the extension $w_0$ by zero on $\R^2\setminus\Omega$ belongs to $H^1(\R^2)$, so that, as in \eqref{e:II0poly}, we get
\begin{equation*} 
\cI_\trd(v_h,w) 
=\cI_\trd(v_p,w)
= \cI_\trd(v_p,w)
+ \cI_\tru(v_p,0)
=\int_\loz \dive[\nabla v_p\, w_0]- \int_\lozenge \dn v_p\, w_0=0.
\end{equation*} 

\begin{figure}[htb]\centering
   \includegraphics[width=70mm]{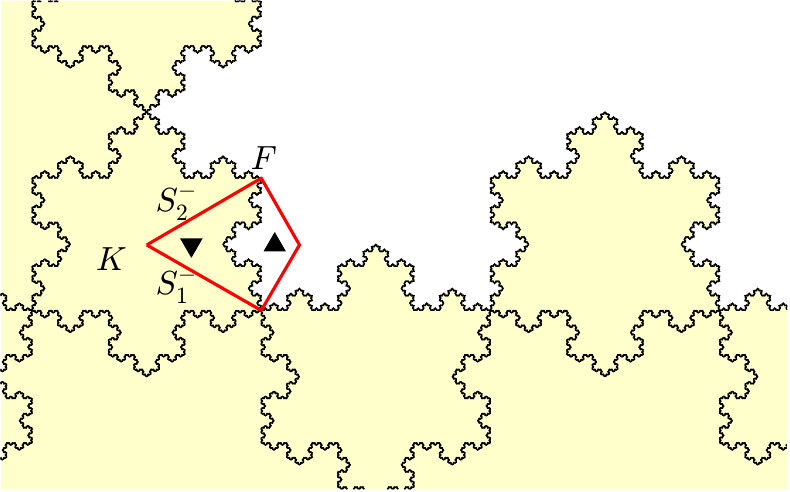}
\hspace{-30mm}\raisebox{40mm}{$\R^2\setminus\Omega$}\hspace{30mm}
\caption{For each boundary face $F$, the wedge $\tru$ is contained in the domain complement $\R^2\setminus\Omega$.}
\label{fig:BdryTria}
\end{figure}

To summarise, so far we have shown that, if $w\in V^\mu\cap H^1_0(\Omega)$ and $v_h\in V_h$, then 
\begin{align*}
\aDG (w,v_h)
=&\sum_{K\in\cT }\int_K \nabla w\cdot\nabla v_h\dx 
-\frac12\sum_{F\in\cF_I}\Big[
 \cI_\trd(w,v_h^--v_h^+)-\cI_\tru(w,v_h^--v_h^+)\Big]
-\sum_{F\in\cF_B}\cI_\trd(w,v_h).
\end{align*}
For such $w$ and $v_h$, noting that~\eqref{e:II0poly} holds also for $w\in V^\mu$ and $v\in \IP^p(\R^2)$ (and not just for $w,v\in \IP^p(\R^2)$), for each $F\in \cF_I$, we have that 
$\cI_\trd(w,v_h^+) = -\cI_\tru(w,v_h^+)$ and $\cI_\tru(w,v_h^-) = -\cI_\trd(w,v_h^-)$, so that
\begin{align*}
\aDG (w,v_h)
=&\sum_{K\in\cT }\int_K \nabla w\cdot\nabla v_h\dx 
-\sum_{F\in\cF_I}\Big[
 \cI_\trd(w,v_h^-)+\cI_\tru(w,v_h^+)\Big]
-\sum_{F\in\cF_B}\cI_\trd(w,v_h)\\
=&\sum_{K\in\cT }\int_K \nabla w\cdot\nabla v_h\dx  \\
 & -\sum_{F\in\cF_I}\Big[
\int_\trd \dive[\nabla w\, v_h^-] \dx  
+\int_\tru \dive[\nabla w\, v_h^+] \dx   
- \int_\vee \partial_\bn w\, v_h^- \ds 
- \int_\wedge \partial_\bn w\, v_h^+ \ds 
\Big]\\
&  -\sum_{F\in\cF_B}  
\Big[\int_\trd \dive[\nabla w\, v_h]  \dx 
- \int_\vee \partial_\bn w\, v_h \ds \Big].
\end{align*}
We then expand each of the divergence terms using the product rule $\dive[\nabla w\, v]=\Delta w\, v + \nabla w \cdot \nabla v$. 
As explained earlier in \S\ref{s:Wedge}, each element $K$ is the disjoint union of 3, 4, 5, or 6 wedges (of type either $\trd$ or $\tru$). 
Moreover, the $\trd$ and $\tru$ wedges appearing in the sum over faces above form a disjoint partition of $\Omega$ by \eqref{e:WedgePartition}.
Also, each of the straight line segments forming $\vee$ and $\wedge$ appear exactly twice, but with opposing normal directions, so that the corresponding integrals sum to zero. Hence, rewriting the sum above as a sum over elements gives 
\begin{align}
\label{e:Cons1}
\aDG(w,v_h)= \sum_{K\in\cT }\int_{K}-\Delta w\, v_h \dx, 
\end{align}
for $w\in V^\mu\cap H^1_0(\Omega)$ and $v_h\in V_h$. 

Finally, taking $w=u$ in \eqref{e:Cons1}, 
 where $u$ is the solution of the BVP \eqref{e:BVP101}, we 
 obtain that
\begin{align*}
\aDG(u,v_h)= \int_\Omega fv_h\dx  = \cL(v_h),
\end{align*}
which concludes the proof of the consistency \eqref{e:Cons} of the DG variational problem \eqref{e:DG}.
\end{proof}

\subsection{Continuity}
\label{s:Continuity}
  We now prove continuity bounds for~$\aDG(\cdot, \cdot)$ (in Lemma \ref{lem:Cont} below), after defining suitable norms.

Let $\mu\in[0,1)$ and let $q>1$ satisfy condition \eqref{e:qmu}. 
    For $w=\widetilde{w}+w_h\in V^\mu+V_h$, with~$V^{\mu}$ defined in~\eqref{e:V}, we introduce the DG and the ($q$-dependent) DG$^+$ norms by
 \begin{align}\nonumber
\|w\|\DG^2&:=\sum_{K\in \cT }\|\nabla w\|_{L_2(K)}^2 
+ \sum_{F\in\cF_I} \frac{1}{h_F^d}\|w_h^--w_h^+\|_{\IL_2(F)}^2 
+ \sum_{F\in\cF_B} \frac{1}{h_F^d}\|w\|_{\IL_2(F)}^2,\\
\|w\|\DGp^2&:=\|w\|\DG^2 + 
\sum_{K\in \cT } \Big[h_K^2 \|\Delta w\|_{L_2(K)}^2 + h_K^{2-2/q}\|\partial_\bn w\|_{L_q(\ast_K)}^2\Big],
\label{e:Norms}
\end{align}
where, for each $K\in \cT $, $\ast_K:=\bigcup_{F\in\cF_I\cup \cF_B} \lozenge\cap K$ is the union of all the line segments (defined in \S\ref{s:Wedge}) separating $K$ into wedges (so that $\ast_K$  is a union of between 3 and 6 line segments, recall Figure~\ref{fig:wedgesB}(b)).
Thanks to Lemma~\ref{lem:dnu2}, the assumptions on $q$ imply that $\|\partial_\bn w\|_{L_q(\ast_K)}$ is finite for each $K\in\cT$, and hence that the norm $\|w\|\DGp$ is finite.
Also, we recall that, by Lemma~\ref{lem:Jumps},  the jump $w_h^--w_h^+$ is independent of the choice of decomposition of $w=\widetilde{w}+w_h$ for $\widetilde{w}\in V^\mu$ and $w_h\in V_h$.  The norms in \eqref{e:Norms} are natural generalisations of those used in the classical analysis of SIP-DG-FEM schemes for Lipschitz polygonal/polyhedral $\Omega$; see, e.g.,\ \cite[eq.~(4.17), (4.18), (4.22)]{di2011mathematical}.
In particular, the contribution  
$h_K^2 \|\Delta w\|_{L_2(K)}^2 + h_K^{2-2/q}\|\partial_\bn w\|_{L_q(\ast_K)}^2$ to $\|w\|\DGp^2-\|w\|\DG^2$ from each $K\in \cT$ is a proxy for $h_K\|\partial_\bn w\|^2_{L_2(\partial K)}$ in the Lipschitz case. 

\begin{lem}[Continuity]\label{lem:Cont}
Let $\mu\in[0,1)$ and let $q>1$ satisfy condition \eqref{e:qmu}.
Then there exist $\CDG>0$ and $\CDGp>0$, depending only on $p$, $q$ and $\eta$, such that
\begin{align}\label{e:Cont}
|\aDG(w,v_h)|&\leq \CDG \|w\|\DGp \|v_h\|\DG \qquad \forall w\in V^\mu+V_h, \; v_h\in V_h, \\
\label{e:ContPlus}
|\aDG(w,v)|\,&\,\leq \CDGp \|w\|\DGp \|v\|\DGp \qquad \forall w\in V^\mu+V_h, \; v\in V^\mu+V_h.
\end{align}
 \end{lem}
\begin{proof}
Let $F\in\cF$. 
Suppose that $w\in H^1(\trd)$ with $\Delta w\in L_2(\trd)$ and $\partial_\bn w\in L_q(\vee)$, 
 and $v\in H^1(\trd)$. 
Let $Q\subset \trd$ be any Lipschitz open set such that $\vee\subset\partial Q$, for instance the quadrilateral $Q$ shown in Figure \ref{fig:Boomerangs}(a). Then $v\in H^1(Q)$, and by standard trace results on Lipschitz domains (e.g.,~\cite[eq.~(1,4,4,5) and Thm~1.5.1.2]{Grisvard85}) it follows that, for every $1\leq q'<\infty$, $v\in L_{q'}(\vee)$ and there exists a constant $C_{\vee,q'}$  depending only on $q'$ such that
\begin{align}\label{e:TraceVee}
 h_F^{-1/q'} \|v\|_{L_{q'}(\vee)} \leq 
C_{\vee,q'} \left(\frac{1}{h_F}\|v\|_{L_2(\trd)} + \|\nabla v\|_{L_2(\trd)}\right)\qquad \forall v\in H^1(\trd).
\end{align}
Let $q' = \frac{q}{q-1} \in [2, \infty)$ be the conjugate exponent of $q$. 
Then, using the triangle and the Cauchy--Schwarz inequalities,  we can bound 
 \begin{align*}
|\cI_\trd(w,v)|
   &\leq \Big|
\int_\trd \Delta w\, v \dx  \Big| + \Big|
\int_\trd \nabla w\cdot \nabla v \dx  \Big| + \Big| \int_\vee \dn w\, v \ds \Big|
\\
&\le h_F \|\Delta w\|_{L_2(\trd)} \frac{1}{h_F}\|v\|_{L_2(\trd)} + 
\|\nabla w\|_{L_2(\trd)} \|\nabla v\|_{L_2(\trd)} + h_F^{1-1/q}\|\partial_\bn w\|_{L_q(\vee)} h_F^{1/q-1}\|v\|_{L_{q'}(\vee)}
\\
&\le N_1^\trd(w)N_2^\trd(v),
\end{align*}
where
\begin{align*}
N_1^\trd(w)^2:=&h_F^2 \|\Delta w\|_{L_2(\trd)}^2 + \|\nabla w\|_{L_2(\trd)}^2 + h_F^{2-2/q}\|\partial_\bn w\|_{L_q(\vee)}^2 ,\\
N_2^\trd(v)^2:=&\frac{1}{h_F^2}\|v\|_{L_2(\trd)}^2 + \|\nabla v\|_{L_2(\trd)}^2+ h_F^{2/q-2} \| v\|_{L_{q'}(\vee)}^2.
\end{align*}
For $F\in\cF_I$ a similar result holds with $\trd$ and $\vee$ replaced by $\tru$ and $\wedge$.

Applying the above estimates we can bound, for $w=\widetilde{w}+w_h\in V^\mu+V_h$ and $v=\widetilde{v}+v_h\in V^\mu+V_h$,
\begin{align}
|\aDG(w,v)|
\leq&
\sum_{K\in\cT }\|\nabla w\|_{L_2(K)} \|\nabla v\|_{L_2(K)}  \notag\\\notag
&+\sum_{F\in\cF_I}{\bigg[} \frac12 {\Big(}
N_1^\trd(w)N_2^\trd(v_h^--v_h^+)+N_1^\tru(w)N_2^\tru(v_h^--v_h^+) \\\notag
&\hspace{20mm}
+N_1^\trd(v)N_2^\trd(w_h^--w_h^+)+N_1^\tru(v)N_2^\tru(w_h^--w_h^+){\Big)}
\\\notag
&
\qquad \qquad+ \frac{\eta}{h_F^{d}}\|w_h^--w_h^+\|_{\IL_2(F)} \|v_h^--v_h^+\|_{\IL_2(F)}{\bigg]}\\\notag
&+\sum_{F\in\cF_B} \Big[N_1^\trd(w)N_2^\trd(v)+N_1^\trd(v)N_2^\trd(w)+\frac\eta{h_F^d}\|w\|_{\IL_2(F)}\|v\|_{\IL_2(F)}\Big]
\\
\leq & 
\, X(w)X(v),
\label{e:aX}
\end{align}
where 
\begin{align*}
X(w)^2:=&\sum_{K\in\cT}\|\nabla w\|_{L_2(K)}^2
\\
+&\sum_{F\in\cF_I}\!\bigg[\frac12\Big(N_1^\trd(w)^2+N_1^\tru(w)^2+N_2^\trd(w_h^--w_h^+)^2+N_2^\tru(w_h^--w_h^+)^2\Big)
+\frac{\eta}{h_F^d}\|w_h^--w_h^+\|_{\IL_2(F)}^2\bigg]\\
+&\sum_{F\in\cF_B}
\bigg[N_1^\trd(w)^2+N_2^\trd(w)^2+\frac\eta{h_F^d}\|w\|_{\IL_2(F)}^2\bigg].
\end{align*}
  
Combining the trace inequality \eqref{e:TraceVee} (in the case $q' = \frac{q}{q-1}$) with the Poincar\'e inequality provided by Lemma \ref{lem:Poincare}, 
there exists a constant $C_q>0$, depending only on $q$, such that
\begin{align*}
N_2^\trd(v) \leq C_q\bigg(\|\nabla v\|_{L_2(\trd)} + \frac{1}{h_F^{d/2}}\|v\|_{\IL_2(F)} \bigg)
\qquad \forall v\in H^1(\trd),
\end{align*}
and similarly for $\tru$ and $\wedge$ in place of $\trd$ and $\vee$ when $F\in\cF_I$.

When one of the arguments of~$\cI_{\trd}(\cdot, \cdot)$ is a polynomial in $\IP^p(\trd)$, we will make use of the following inverse inequalities, which follow from the finite-dimensionality of $\IP^p(\trd)$: 
there exist constants $C_{p,q,1},C_{p,q,2}>0$, depending only on $p$ and $q$, such that
\begin{align}\label{e:InvIneqN}
N_1^\trd({v_p})\leq C_{p,q,1} \|\nabla {v_p}\|_{L_2(\trd)}
\quad {\text{and}} \quad 
N_2^\trd({v_p})\leq C_{p,q,2} \frac1{h_F^{d/2}}\|{v_p}\|_{\IL_2(F)} 
\qquad \forall v_p\in\IP^p(\trd),
\end{align}
 and similarly for $\tru$ and $\wedge$ in place of $\trd$ and $\vee$ when $F\in\cF_I$.

From the above it follows that, for $w=\widetilde{w}+w_h\in V^\mu + V_h$, 
\begin{align*}
X(w)^2 
 \le&\sum_{K\in\cT}\|\nabla w\|_{L_2(K)}^2\\
+&\sum_{F\in\cF_I}\bigg[\frac12\Big(h_F^2 \|\Delta w\|_{L_2(\trd)}^2 + \|\nabla w\|_{L_2(\trd)}^2 + h_F^{2-2/q}\|\partial_\bn w\|_{L_q(\vee)}^2
\\
& \qquad\qquad\quad + h_F^2 \|\Delta w\|_{L_2(\tru)}^2 + \|\nabla w\|_{L_2(\tru)}^2 + h_F^{2-2/q}\|\partial_\bn w\|_{L_q(\wedge)}^2\Big)\\
& \qquad\quad +(C_{p,q,2}^2+\eta)\frac1{h_F^d}\|w_h^--w_h^+\|_{\IL_2(F)}^2\bigg]
\\
+&\sum_{F\in\cF_B}\!\Big[h_F^2\|\Delta w\|_{L_2(\trd)}^2
+(1+2C_q^2)\|\nabla w\|_{L_2(\trd)}^2 + h_F^{2-2/q}\|\partial_\bn w\|_{L_q(\vee)}^2
+(2C_q^2+\eta)\frac1{h_F^d}\|w\|_{\IL_2(F)}^2\Big]
\\
\le&\ C \|w\|\DGp^2,
\end{align*}
for some $C>0$ depending only on $p$ and $q$ (through $C_q$ and $C_{p,q,2}$) and $\eta$. 
Also, for $v_h\in V_h$, 
\begin{align*}
X(v_h)^2 
\le&\sum_{K\in\cT}\|\nabla v_h\|_{L_2(K)}^2\\
&+\sum_{F\in\cF_I}\bigg[\frac12 C_{p,q,1}^2\Big(\|\nabla v_h\|_{L_2(\trd)}^2 + \|\nabla v_h\|_{L_2(\tru)}^2 \Big) 
+(C_{p,q,2}^2+\eta)\frac1{h_F^d}\|v_h^--v_h^+\|_{\IL_2(F)}^2\bigg]
\\
&+\sum_{F\in\cF_B} \Big[C_{p,q,1}^2 \|\nabla v_h\|_{L_2(\trd)}^2 + (C_{p,q,2}^2+\eta)\frac1{h_F^d}\|v_h\|_{\IL_2(F)}^2\Big]\\
\le&\ C \|v_h\|\DG^2,
\end{align*}
for some $C>0$ depending only on $p$ and $q$ (through $C_{p,q,1}$ and $C_{p,q,2}$) and $\eta$. 

The bounds \eqref{e:Cont} and \eqref{e:ContPlus} then follow by combining these estimates with \eqref{e:aX}.
\end{proof}

\subsection{Quasi-optimality}
\label{s:QO}

By \cite[Thm 1.35]{di2011mathematical}, the coercivity~\eqref{e:Coercivity} and continuity~\eqref{e:Cont} of the bilinear form~$\aDG(\cdot,\cdot)$, together with the consistency \eqref{e:Cons}, imply that the Galerkin error is quasi-optimal.
We record this, and the other properties of the SIP-DG-FEM scheme \eqref{e:DG} proved so far, in the following theorem.

\begin{thm}[Quasi-optimality]\label{thm:QO}
Let $\cT$ be an LQU mesh on the Koch snowflake $\Omega$, $p\in\N$, and $V_h$ be the piecewise-polynomial space in \eqref{e:Vh}.
Let the penalty parameter $\eta$ satisfy \eqref{e:Eta}.
Then the bilinear form $\aDG(\cdot,\cdot)$ is coercive \eqref{e:Coercivity}, and the SIP-DG-FEM formulation \eqref{e:DG} has a unique solution $u\DG\in V_h$. 
For any~$\mu \in [0, 1)$, $\aDG(\cdot,\cdot)$ is continuous in the sense of \eqref{e:Cont}. 
Moreover, if the solution~$u$ of the BVP \eqref{e:BVP101} belongs to~$V^{\mu}$, for some~$\mu \in [0, 1)$, 
 then the method is consistent \eqref{e:Cons}, and the following quasi-optimal error bound holds: 
\begin{align}\label{e:QO}
\|u-u\DG\|\DG \leq (1+2\CDG) \inf_{v_h\in V_h}\|u-v_h\|\DGp,
\end{align}
where $\CDG$ is the constant from \eqref{e:Cont}, and where the DG$^+$ norm \eqref{e:Norms} depends on the choice of $q>1$ satisfying condition \eqref{e:qmu}. 
\end{thm}
Quasi-optimality holds in the sense of \cite[Def.~1.54(ii)]{di2011mathematical}: the norms \eqref{e:Norms} on the left- and right-hand
sides of \eqref{e:QO} are different but scale with the same power of the mesh size for smooth ($H^2(\Omega)$) functions and quasi-uniform meshes. 
Quasi-optimality can also be proved with the DG$^+$ norm on both sides of the inequality; this will be done later in \eqref{e:QO-DGp} when we prove $L_2$ error bounds.

 \subsection{Error estimates}
\label{s:Error}

We now study the accuracy of the best-approximation error on the right-hand side of \eqref{e:QO}, in order to obtain convergence estimates for the SIP-DG-FEM scheme \eqref{e:DG}. 
Our aim here is to prove low-order $h$-convergence estimates, independent of the polynomial degree $p\geq 1$ used to define the discrete space $V_h$.
Higher-order error estimates for $p>1$ require a finer regularity analysis of the BVP solution $u$ and are deferred to future work. 
             
For each $K\in\cT$, denote its distance from the domain boundary by 
\begin{equation}\label{e:deltaK}
\delta_K:=\inf\big\{\delta(\bx):\,\bx\in K\big\}=\inf\big\{|\bx-\by|:\, \bx\in K,\, \by\in\deO\big\},
\end{equation}
where $\delta$ was defined in \eqref{e:delta}.
We have $\delta_K\le \widehat \delta_K$, with $\widehat\delta_K$ defined in \S\ref{ss:MeshSeq}.

We begin with a technical lemma, in which we bound all terms in the DG$^+$ norm \eqref{e:Norms} by elementwise weighted Sobolev norms.
\begin{lem}[Bound on the~DG$^+$ norm]\label{lem:DGpvsSobol}
Let $\cT$ be an LQU mesh.
Let $\mu\in[0,1)$ and $q\ge 1$ satisfy condition \eqref{e:qmu}.
For all $w\in V^\mu+V_h$,
\begin{align}\label{e:DGpvsSobol}
\|w\|\DGp^2
&\le(1+24C_\Tr^2)(1+36\widetilde C_q^2)
\bigg[\sum_{K\in \cT }\Big(\frac1{h_K^2}\|w\|_{L_2(K)}^2 + \|\nabla w\|_{L_2(K)}^2 
+h_K^2 \|\Delta w\|_{L_2(K)}^2 \Big)
\\&\hspace{40mm}
+\sum_{K\in\cT,\, \delta_K=0} h_K^{2-2\mu} \|\delta^\mu D^2 w\|_{L_2(K)}
+\sum_{K\in\cT,\, \delta_K>0} h_K^2  \|D^2 w\|_{L_2(K)} \bigg],
\nonumber
\end{align}
where $C_\Tr$ and $\widetilde C_q$ are the constants from \eqref{e:TraceIneq} and Lemma~\ref{lem:dnu2}.
\end{lem}
\begin{proof}
Since the elements of $H^1(\Omega)$ admit Dirichlet traces on any face $F\in\cF_I$ by \S\ref{ss:Traces}, for any $w=\widetilde w+w_h\in V^\mu+V_h$, the polynomial jump term $\|w_h^--w_h^+\|_{\IL_2(F)}=\|(\widetilde w+w_h^-)-(\widetilde w+w_h^+)\|_{\IL_2(F)}$ in the DG norm \eqref{e:Norms} can be bounded by the sum of the $\IL_2(F)$ traces of $w$ from $K_-$ and $K_+$.
Then, using the trace inequality~\eqref{e:TraceIneq}, the DG$^+$ norm \eqref{e:Norms} can be bounded by a sum of terms associated with the mesh elements as follows: for all $w\in V^\mu+V_h$
\begin{align} \label{e:DGpNormBd}
\|w\|\DGp^2
=&\sum_{K\in \cT }\Big(\|\nabla w\|_{L_2(K)}^2 
+ h_K^2\|\Delta w\|_{L_2(K)}^2 + h_K^{2-2/q}\|\partial_\bn w\|_{L_q(\ast_K)}^2\Big)
\\& \nonumber
+ \sum_{F\in\cF_I} \frac{1}{h_F^d}\|w_h^--w_h^+\|_{\IL_2(F)}^2 
+ \sum_{F\in\cF_B} \frac{1}{h_F^d}\|w\|_{\IL_2(F)}^2
\\ \nonumber
 \le & \sum_{K\in \cT }\bigg(\|\nabla w\|_{L_2(K)}^2 
+ h_K^2\|\Delta w\|_{L_2(K)}^2 + h_K^{2-2/q}\|\partial_\bn w\|_{L_q(\ast_K)}^2 + 2\!\!\sum_{F\in\cF,F\subset\deK}\frac{1}{h_F^d}\|w\|_{\IL_2(F)}^2\bigg)
\\ \nonumber
\le &(1+24C_\Tr^2)\sum_{K\in \cT }\bigg( \frac1{h_K^2}\|w\|_{L_2(K)}^2 + \|\nabla w\|_{L_2(K)}^2 
+h_K^2 \|\Delta w\|_{L_2(K)}^2 + h_K^{2-2/q}\|\partial_\bn w\|_{L_q(\ast_K)}^2\bigg),
 \end{align}
where, in the last step, we used that each element has at most six faces.

 Since each $\ast_K$ is the (non-disjoint) union of at most three $\vee$ or $\wedge$ sets (see Figure~\ref{fig:wedgesB}(b)), recalling that $h_K\le 2h_F$ for $F\subset\deK$ (see \eqref{e:hFhK}), and that $\dist(\bx,\deK)\le \min\{\delta(\bx),h_K\}$ for all $\bx\in K$, the trace inequality in Lemma~\ref{lem:dnu2} gives, for all $w\in V^\mu+V_h$
\begin{align*}
h_K^{1-1/q}\|\partial_\bn w\|_{L_q(\ast_K)}
\le \begin{cases} 
3^{1/q}\cdot 2^{1/q-1}\widetilde C_q \Big(\|\nabla w\|_{L_2(K)} +h_K^{1-\mu}  \|\delta^\mu D^2 w\|_{L_2(K)}\Big)
& \text{if }\delta_K=0,\\
3^{1/q}\cdot 2^{1/q-1}\widetilde C_q \Big(\|\nabla w\|_{L_2(K)} +h_K  \| D^2 w\|_{L_2(K)}\Big)
& \text{if }\delta_K>0.
\end{cases}
\end{align*}
Since $q\ge1$,   squaring and summing over the elements leads to
\begin{align*} &\sum_{K\in\cT}h_K^{2-2/q}\|\partial_\bn w\|_{L_q(\ast_K)}^2\\&
\le 36^{1/q}\,\widetilde C_q^2\Big(\sum_{K\in\cT}\|\nabla w\|_{L_2(K)}^2 
+\sum_{K\in\cT,\, \delta_K=0} h_K^{2-2\mu}  \|\delta^\mu D^2 w\|_{L_2(K)}^2
+\sum_{K\in\cT,\, \delta_K>0} h_K^2   \|D^2 w\|_{L_2(K)}^2\Big).
\nonumber
\end{align*}
Substituting this bound in \eqref{e:DGpNormBd} gives the assertion.
\end{proof}

To derive bounds on the best approximation error on the right-hand side of \eqref{e:QO} we will combine Lemma \ref{lem:DGpvsSobol} with the polynomial approximation estimates in Corollary \ref{cor:Approx2}.
The latter is a simple consequence of the following lemma, which we prove by standard arguments based on Poincar\'e-type inequalities, justified by the compactness of certain function space embeddings, which hold (e.g.\ using results in \cite{chua2013compact}) because $\Omega$ is a John domain (in fact, a uniform domain, equivalently, a quasidisk), see, e.g.~\cite{buckleykoskela1995SobolevPoincareJohn,harjulehto2012examples}.

\begin{lem}[Poincar\'e-type inequality]\label{lem:Approx}
Let $\mu\in[0,1)$.
Then $H^{2,2}_\mu(\Omega)$ is compactly embedded in $H^1(\Omega)$, and there exists $C_\mu>0$ such that
\begin{align}\label{e:Approx}
\inf_{P\in\mathbb{P}^1(\Omega)}\big(\|v-P\|_{L_2\OO} + \|\nabla(v-P)\|_{L_2\OO}\big)
\le C_\mu\|\delta^\mu D^2 v\|_{L_2\OO} \qquad\forall v\in H^{2,2}_\mu\OO.
\end{align}
The above result holds not only for the Koch snowflake domain $\Omega$ but for any bounded John domain. 
\end{lem}
\begin{proof}
We first note that, since $\Omega$ is a John domain, the weighted $H^1$ space 
\[
H^{1,1}_\mu(\Omega):= \{u\in L_2(\Omega):  \delta^\mu \nabla u \in L_2(\Omega) \}
\]
is compactly embedded in $L_2(\Omega)$ for any $\mu\in [0,1)$. This holds by \cite[Theorem 2.4(i)]{chua2013compact} with $s=1$, $p=q=2$, $a=0$ and $b=2\mu$. (In the notation of \cite[Theorem 2.4(i)]{chua2013compact}, we note that $L^2_{\rho^a\,dx}\cap E^2_{\rho^b\,dx}\subset L^1_{\rho^a\,dx}\cap E^2_{\rho^b\,dx}$ since $\Omega$ is bounded.) 
As a special case (taking $\mu=0$) we note that $H^1(\Omega)$ is compactly embedded in $L_2(\Omega)$. 

As a consequence, $H^{2,2}_\mu(\Omega)$ is compactly embedded in $H^1(\Omega)$. This follows by a standard argument: Let $(u_n)$ be a bounded sequence in $H^{2,2}_\mu(\Omega)$. Then $(\nabla u_n)$ is a bounded sequence in $H^{1,1}_\mu(\Omega)$, so by the above we can find a subsequence (which we again call $(u_n)$) such that $(\nabla u_n)$ converges in $L_2(\Omega)$. Furthermore, $(u_n)$ is a bounded sequence in $H^1(\Omega)$, so we can find a subsubsequence (which we again call $(u_n)$) such that $(u_n)$ also converges in $L_2(\Omega)$. Combining these two observations and applying a standard argument to show that the limit and the gradient commute, we find that $(u_n)$ converges in $H^1(\Omega)$, as required.
Again, as a special case (taking $\mu=0$) we obtain that $H^2(\Omega)$ is compactly embedded in $H^1(\Omega)$. 

Finally, \eqref{e:Approx} follows by combining the compactness of the above embedding with the general framework for Poincar\'e-type inequalities provided by \cite[Lemma~4.1.3]{ziemer}. Explicitly, in the notation of \cite[Lemma~4.1.3]{ziemer}, we take  $X_0=H^1\OO$, $X=H^{2,2}_\mu\OO$, $\|v\|_1=\|\delta^\mu D^2v\|_{L_2(\Omega)}$, $Y=\IP^1\OO$, and 
$L : H^{2,2}_\mu(\Omega) \to \IP^1(\Omega)$ to be the orthogonal $H^{2,2}_\mu(\Omega)$ projection onto $\IP^1(\Omega)$, noting that $\IP^1(\Omega)=\{f\in H^{2,2}_\mu(\Omega):\|f\|_1=0\}$.
\end{proof}

From Lemma \ref{lem:Approx} we can obtain the following $h$-explicit elementwise estimates.
\begin{cor}[Local approximation estimate]
\label{cor:Approx2}
There exists $\hat{C}_0>0$ such that if $K\in\cT$ and $w\in H^2(K)$ then
\begin{equation}\label{e:H2approx}
\inf_{P_K\in\IP^1(K)} \Big(\frac1{h_K^2}\|w-P_K\|_{L_2(K)}^2+\|\nabla(w-P_K)\|_{L_2(K)}^2\Big)
\le \hat{C}_0\, h_K^2\, \|D^2 w\|_{L_2(K)}^2.
\end{equation}
Let $\mu\in[0,1)$. Then there exists $\hat{C}_\mu>0$ such that if 
$K\in\cT$, $\delta_K=0$ and $w\in H^{2,2}_\mu(\Omega)$ then 
\begin{equation}\label{e:H22approxK}
\inf_{P_K\in\IP^1(K)} \Big(\frac1{h_K^2}\|w-P_K\|_{L_2(K)}^2+\|\nabla(w-P_K)\|_{L_2(K)}^2\Big)
\le \hat{C}_\mu \, h_K^{2-2\mu}\, \|\delta^\mu D^2 w\|_{L_2(K)}^2.
\end{equation}
\end{cor}
\begin{proof}
In the case $\delta_K=0$, note that $\partial K\cap \partial\Omega\subset \partial\Omega$, so that 
$\delta(\bx)=\dist(\bx,\partial \Omega)=\dist(\bx,\partial K\cap \partial\Omega)\geq  \dist(\bx,\partial K)$ for $\bx\in K$, with the result that 
$\{w\in H^1(K): \delta^\mu D^2 w\in L_2(K)\}\subset \{w\in H^1(K): \dist(\cdot,\partial K)^\mu D^2 w\in L_2(K)\}$. 
Recalling that each mesh element is similar to the snowflake $\Omega$, the results \eqref{e:H2approx} and \eqref{e:H22approxK} then follow from applying \eqref{e:Approx} (with $\mu=0$ in the case of \eqref{e:H2approx}) to the composition of $w$ with the similarity mapping $\Omega$ to $K$.
\end{proof}

Combining the above results gives the following best-approximation estimate in the $\|\cdot\|\DGp$ norm. 
Recall that $\delta$ and $\delta_K$ are defined in \eqref{e:delta} and \eqref{e:deltaK}.

\begin{prop}[Approximation estimate in $\|\cdot\|\DGp$ norm]\label{prop:ApproxDG}
Let $p\in\N$, $\cT$ be an LQU mesh as in \S\ref{ss:Mesh}, and $w\in V^\mu$ for some $0\le\mu<1$.
Then
\begin{align}\label{e:ApproxDGp}
&\inf_{w_h\in V_h}\|w-w_h\|\DGp\\
&\le C \bigg[
\sum_{K\in\cT}h_K^2 \|\Delta w\|_{L_2(K)}^2 
+\sum_{K\in\cT,\, \delta_K=0} h_K^{2-2\mu}  \|\delta^\mu D^2 w\|_{L_2(K)}^2
+\sum_{K\in\cT,\, \delta_K>0} h_K^2   \|D^2 w\|_{L_2(K)}^2
\bigg]^{1/2},
\nonumber
\end{align}
where $C$ only depends on $q$ and $\mu$.
\end{prop}
\begin{proof} 
Let $\Pi_\cT w\in \IP^1(\cT)\subset V_h$ be defined elementwise by $(\Pi_\cT w)|_K=P_K$ with $P_K$ as in \eqref{e:H2approx} if $\delta_K>0$, and $P_K$ as in \eqref{e:H22approxK} if $\delta_K=0$.
Then, combining Lemma~\ref{lem:DGpvsSobol}
and Corollary \ref{cor:Approx2}, and using that $\Delta\Pi_\cT w=D^2\Pi_\cT w=0$ in each element, 
we control the best-approximation error as:
\begin{align*}
&\inf_{w_h\in V_h}\|w-w_h\|\DGp^2
\le \|w-\Pi_\cT w\|\DGp^2\\
&\overset{\eqref{e:DGpvsSobol}}\le(1+24C_\Tr^2)(1+36\widetilde C_q^2)
\bigg[\sum_{K\in \cT }\Big( \frac1{h_K^2}\|w-P_K\|_{L_2(K)}^2 + \|\nabla (w-P_K)\|_{L_2(K)}^2
+h_K^2 \|\Delta w\|_{L_2(K)}^2 \Big)
\\
&\hspace{50mm}
+\sum_{K\in\cT,\, \delta_K=0} h_K^{2-2\mu}  \|\delta^\mu D^2 w\|_{L_2(K)}^2
+\sum_{K\in\cT,\, \delta_K>0} h_K^2   \|D^2 w\|_{L_2(K)}^2
\bigg]
\\
&\overset{\eqref{e:H2approx},\eqref{e:H22approxK}}
\le(1+24C_\Tr^2)(1+36\widetilde C_q^2)
(1+\max\{\hat{C}_0,\hat{C}_\mu\}  )\\
&\hspace{10mm}\cdot\bigg[
\sum_{K\in\cT}h_K^2 \|\Delta w\|_{L_2(K)}^2 
+\sum_{K\in\cT,\, \delta_K=0} h_K^{2-2\mu}  \|\delta^\mu D^2 w\|_{L_2(K)}^2
+\sum_{K\in\cT,\, \delta_K>0} h_K^2   \|D^2 w\|_{L_2(K)}^2
\bigg] .
\end{align*}
\end{proof}

The quasi-optimality bound \eqref{e:QO} and the best-approximation estimate \eqref{e:ApproxDGp} together give the following a priori error estimate.
\begin{thm}[A priori error estimate]\label{thm:Rates}
Let $f\in L_2(\Omega)$, $p\in\N$, and $\cT$ be an LQU mesh as in \S\ref{ss:Mesh}.
Let $0\le\mu<1$ be such that the solution $u$ of the BVP \eqref{e:BVP101} belongs to $H^{2,2}_\mu(\Omega)$. 
Let the penalty parameter $\eta$ satisfy \eqref{e:Eta}.
Then the unique solution $u\DG\in V_h$ of the SIP-DG-FEM \eqref{e:DG}  satisfies
\begin{align}\label{e:Rates}
\|u-u\DG\|\DG \le C
    \bigg[
\sum_{K\in\cT}h_K^2 \|\Delta u\|_{L_2(K)}^2 
+\sum_{\substack{K\in\cT,\\ \delta_K=0}} h_K^{2-2\mu}\|\delta^\mu D^2 u\|_{L_2(K)}^2
+\sum_{\substack{K\in\cT,\\ \delta_K>0}} h_K^2\|D^2 u\|_{L_2(K)}^2
\bigg]^{1/2},
\end{align}
where $C$ only depends on $p$, $\mu$ and $\eta$.
\end{thm}

\begin{rem}
\label{rem:RatesComments}
The parameter $q$ entering the definition of the DG$^+$ norm \eqref{e:Norms} is not present in Theorem~\ref{thm:Rates}.
Given $\mu\in[0,1)$, any choice of $q$ satisfying \eqref{e:qmu} (e.g., \ $q=\frac{1+1/\mu}2$ in the case $1/2\leq \mu<1$) determines the $\|\cdot\|\DGp$ norm, the bounding constants in the intermediate steps in the proof of Theorem~\ref{thm:Rates}, and hence the constant $C$.
\end{rem}

The following Corollary of Theorem~\ref{thm:Rates} provides a simpler error bound, relevant for globally quasi-uniform meshes such as the $\cT_\ell'$ family in \S\ref{ss:MeshSeq}. (A more refined analysis could provide estimates in terms of the number of degrees of freedom on meshes refined towards the boundary, but we leave this for future work.) It also provides an $L_2$ error estimate.

\begin{cor}[Simplified bound and \texorpdfstring{$L_2\OO$}{L2(Omega)}-norm error estimate]
\label{c:L2}
Let $0\leq \mu<1$ be such that, for arbitrary $f\in L_2(\Omega)$, the solution $u$ of \eqref{e:BVP101} satisfies \eqref{e:Stab}. (As discussed in \S\ref{ss:H22}, by the results in \cite{NystromThesis} and \cite{capitanelli2015weighted}, this is known to hold for $\mu_*<\mu<1$, where $\mu_*\approx 0.828$.) 

Let $f\in L_2(\Omega)$, $p\in\N$, and $\cT$ be an LQU mesh as in \S\ref{ss:Mesh}.
Let the penalty parameter $\eta$ satisfy \eqref{e:Eta}. 
Let $h=\max_{K\in\cT}h_K$.  
Then there exists $C>0$, 
depending only on $p$, $\mu$ and $\eta$, such that 
the unique solution $u\in H_0^1(\Omega)$ of \eqref{e:BVP101} and the unique solution $u\DG\in V_h$ of the SIP-DG-FEM \eqref{e:DG} satisfy
\begin{align}\label{e:Simple}
\|u-u\DG\|\DG
\le C h^{1-\mu}\|f\|_{L_2\OO}
\end{align}
and 
\begin{align}\label{e:L2}
\|u-u\DG\|_{L_2\OO}
\le C h^{2(1-\mu)} \|f\|_{L_2\OO}.
\end{align}
\end{cor}
\begin{proof}
The simplified bound \eqref{e:Simple} can be derived from \eqref{e:Rates}, with a modified value of $C$,
by noting that the LQU assumption on the mesh implies that $\delta_K\ge \frac{\sqrt7}6h_K$ for each~$K$ with $\delta_K>0$, and applying the stability bound \eqref{e:Stab}.

To derive the $L_2$ bound \eqref{e:L2}, we follow the classical Aubin--Nitsche duality argument as in \cite[\S4.2.4]{di2011mathematical}.
We first introduce the auxiliary problem: find $\zeta\in H_0^1(\Omega)$ such that 
$-\Delta\zeta=u-u\DG$
in $\Omega$.
By the Lax--Milgram Lemma this has a unique solution satisfying 
\begin{equation}\label{e:Dual}
\int_\Omega \nabla \zeta\cdot\nabla v\dx=\int_\Omega (u-u\DG)v\dx \qquad \forall v\in H^1_0\OO.
\end{equation}
Furthermore, by our assumption on $\mu$, and by \eqref{e:Stab}, we have that 
$\zeta\in H^{2,2}_\mu\OO$, with 
$\|\zeta\|_{H^{2,2}_\mu\OO}\le C_{\mathrm{Stab,\mu}}\|u-u\DG\|_{L_2\OO}$. 

Applying the consistency \eqref{e:Cons} (justified in Lemma \ref{lem:Consistency}) with $u$ replaced by $\zeta$ we have that 
$\aDG(\zeta,u\DG) = \int_\Omega (u-u\DG) u\DG  \,\rd \bx$,
and, arguing as in the proof of Lemma \ref{lem:Consistency}, one can check that 
$\aDG(\zeta,u) =  \int_\Omega \nabla\zeta\cdot \nabla u \,\rd \bx$,
because both $\zeta$ and $u$ belong to $H^1_0\OO$ and can be extended by zero outside $\Omega$, so all the boundary terms vanish. 
Combining these two identities with \eqref{e:Dual} gives 
\begin{align*}
\aDG(\zeta,u-u\DG) = \int_\Omega (u-u\DG)^2 \,\rd \bx = \|u-u\DG\|_{L_2(\Omega)}^2.
\end{align*}
Also, applying the consistency \eqref{e:Cons} to $u$, and recalling the definition of $u\DG$, we have the Galerkin orthogonality $\aDG(u-u\DG,v_h)=0$ for all $v_h\in V_h$.

Combining these facts with the symmetry of $\aDG$ and the continuity bound \eqref{e:ContPlus}, we obtain
\begin{align}
\|u-u\DG\|_{L_2\OO}^2
&=\aDG(\zeta,u-u\DG)\notag\\
&=\aDG(u-u\DG,\zeta)\notag\\
&=\aDG(u-u\DG,\zeta-\Pi_\cT\zeta)\notag\\
&\le C\DGp \|u-u\DG\|\DGp \|\zeta-\Pi_\cT\zeta\|\DGp,
\label{e:AN}
\end{align}
with $\Pi_\cT$ as in the proof of Prop.~\ref{prop:ApproxDG}. 

We assume that $q$ has been chosen, depending on $\mu$, as per Remark \ref{rem:RatesComments}. 
Then by Proposition \ref{prop:ApproxDG} and the arguments used to prove \eqref{e:Simple} above, we have that
\begin{align}
\label{e:zetaEst}
\|\zeta-\Pi_\cT\zeta\|\DGp \leq C h^{1-\mu}\|u-u\DG\|_{L_2(\Omega)}
\end{align}
for some $C$ depending only on $\mu$. 
We also claim that
\begin{align}
\label{e:uEst}
\|u-u\DG\|\DGp \leq C h^{1-\mu}\|f\|_{L_2(\Omega)}
\end{align}
for some $C$ depending only on $\mu$, $p$ and $\eta$. To prove \eqref{e:uEst} we note that the first inverse inequality in \eqref{e:InvIneqN} implies that $\|w_h\|\DGp\le C_{p,q,1}\|w_h\|\DG$ for all $w_h\in V_h$.
Thus the discrete coercivity bound \eqref{e:Coercivity} in the $\|\cdot\|\DG$ norm, combined with the continuity \eqref{e:ContPlus}, gives discrete coercivity and quasi-optimality in the stronger $\|\cdot\|\DGp$ norm:
\begin{align}\nonumber
\aDG(w_h,w_h)&\ge \frac12\|w_h\|\DG^2\ge\frac1{2C_{p,q,1}^2}\|w_h\|\DGp^2,
\quad\;\\
\|u-u\DG\|\DGp&\le(1+2C\DGp C_{p,q,1}^2)\inf_{v_h\in V_h}\|v-v_h\|\DGp.
\label{e:QO-DGp}
\end{align}
From this, \eqref{e:uEst} then follows by Proposition \ref{prop:ApproxDG}, in the same way as \eqref{e:zetaEst}.

Finally, \eqref{e:L2} follows by combining \eqref{e:AN}, \eqref{e:zetaEst} and \eqref{e:uEst}, and simplifying. 
\end{proof}

\section{Implementation details}
\label{s:Quad}

 We now consider the practical implementation of the SIP-DG-FEM \eqref{e:aDefDisc}--\eqref{e:DG}. 
As in the classical case, this involves selecting a basis for the numerical approximation space $V_h$, which transforms \eqref{e:DG} into a linear system of equations for the unknown solution coefficients. The entries of the system matrix and the right-hand side vector involve integrals of the basis functions and the source term $f$, which need to be computed before the linear system can be solved. In our case, this includes integrals with respect to two-dimensional Lebesgue measure over mesh elements $K\in \cT$ and wedges $\trd$ and $\tru$, all of which have fractal boundary; integrals with respect to one-dimensional Lebesgue measure over the line segments making up $\vee$ and $\wedge$; and integrals with respect to Hausdorff measure $\cH^d$ (for $d = \log(4)/\log(3)$) over fractal curves $F\in \cF$. We now explain how these integrals are computed in our implementation.

\subsection{Reference domains} \label{ss:RefDom}
The basis for our piecewise-polynomial approximation space $V_h$ will be defined in terms of a set of basis functions on the reference Koch snowflake element~$\hatK = \Omega$, with centre at~$(0, 0)^{\top}$ and~$\diam(\hatK) = 2$.
Each element~$K \in \Th$ is obtained through a similarity~$\psi_K : \hatK \to K$ of the form
\begin{alignat}{3} \label{e:psiK}
\psi_K(\widehat{\bx}) := \frac{\hK}{2} R_K \widehat{\bx} + \bx_K,
\end{alignat}
where~$\bx_K \in \R^2$ is the barycentre of~$K$, and~$R_K\in \R^{2\times 2}$ is a rotation matrix with angle~$\theta = 0$ or~$\theta = \pi/6$. 
The reference element $\widehat{K}$ can be partitioned into six reference wedges $W_1,\ldots,W_6$, separated by six reference line segments $S_1,\ldots,S_6$, and its boundary can be partitioned into six reference Koch curves $F_1,\ldots,F_6$, as illustrated in Figure \ref{Fig2}(a). 

For a face $F\in \cF$, if $K_\pm \in \cT$ are such that $F=\partial K_-\cap\partial K_+$ with $\trd\subset K_-$ and $\tru\subset K_+$
then: 
\begin{itemize}
\item $\widehat{\trd}:=\psi_{K_-}^{-1}(\trd)\subset \widehat{K}$ is one of $W_1,\ldots,W_6$;
\item $\widehat{\vee}:=\psi_{K_-}^{-1}(\vee)$ is one of the six unions
$S_1\cup S_2,
 \ldots,
S_5\cup S_6,S_6\cup S_1$; 
\item $\widehat{F}_-:=\psi_{K_-}^{-1}(F)$ is one of $F_1,\ldots,F_6$;
\item $\widehat{\tru}:=\psi_{K_+}^{-1}(\tru)\subset \widehat{K}$ is one of the six unions $W_1\cup W_2,
 \ldots,
W_5\cup W_6,W_6\cup W_1$;
\item $\widehat{\wedge}:=\psi_{K_+}^{-1}(\wedge)$ is one of the six unions
$S_1\cup S_3,
 \ldots,
S_5\cup S_1,S_6\cup S_2$.
\item $\widehat{F}_+:=\psi_{K_+}^{-1}(F)$ is one of the six unions
$F_1\cup F_2,
 \ldots,
F_5\cup F_6,F_6\cup F_1$.
\end{itemize}
  
We also use the Koch curve $\Gamma$ (introduced in \S\ref{ss:Koch}) as our reference face, noting that for each face $F\in\cF$ there exists a similarity $\xi_{F} : \Gamma \to F$ mapping $\Gamma$ bijectively to $F$, which is of the form
\begin{equation}
\label{e:Face}
\xi_{F}(\widehat{\bx}) = \hF R_{F} \widehat{\bx} + \bx_F,
\end{equation}
for some rotation matrix~$R_{F} \in \R^{2 \times 2}$ and some translation vector~$\bx_F \in \R^2$.

\subsection{Basis for piecewise polynomial space} \label{ss:Basis}
 Let~$\Np := \dim(\mathbb{P}^p(\mathbb{R}^2))$, and let
$  \{\widehat{\phi}_i\}_{i = 1}^{\Np}$
be a basis for the space~$\mathbb{P}^p(\hatK)$. (In our implementation we use a monomial basis.) 
Let $\NT:=\card(\Th)$. 
Let $\{K_m\}_{m = 1}^{\NT}$ be a prescribed ordering of the elements in~$\Th$. For each~$m =1,\ldots \NT$, we consider the basis for the space~$\mathbb{P}^p(K_m)$ given by
 $$
 \big\{\phi_i^m \in \mathbb{P}^p(K_m) \ : \  \phi_i^m(\bx) = \hphi_i \big(\psi_{K_m}^{-1}(\bx)\big), 
\text{ for } i = 1, \ldots, \Np  \big\}.$$
 A basis for the global space~$\Vh$ can then be defined in a natural way, as 
\begin{align}\label{e:Basis}
\big\{\phi_1^1\chi_{K_1},\;\ldots,\; \phi_{\Np}^1\chi_{K_1},\; \ldots \ldots,\;
\phi_1^{\NT}\chi_{K_{\NT}},\;\ldots,\; \phi_{\Np}^{\NT}\chi_{K_{\NT}}\big\},
\end{align}
where, for $K\in\cT$, $\chi_K$ denotes the characteristic function of $K$.
 
\subsection{Linear system} \label{ss:LSE}
Let $N:=\Np\NT$ be the total number of DOFs. 
Then, denoting by~$\su \in \R^{N}$ and~$\sfb \in \R^{N}$ the vector representation of~$\uh$ and the linear functional~$\cL(\cdot)$, respectively, the SIP-DG-FEM~\eqref{e:DG} is equivalent to the linear system
\begin{equation*} \text{find~$\su\in \R^{N}$ such that \ $\ADG \,\su = \sfb$,}
\end{equation*}
where $\ADG\in \R^{N\times N}$ is the Galerkin matrix associated with the bilinear form~$\aDG(\cdot, \cdot)$. Recalling the definition~\eqref{e:aDefDisc}, we can write $\ADG$ as
 \begin{equation}\label{e:ADG}
\ADG = \sG + \sC + \sC^{\top}+ \sP,
\end{equation}
where $\sG$, $\sC$, and~$\sP$ are the matrices associated with the bilinear forms defined, for $\wh, \vh \in \Vh$, by
 \begin{alignat*}{3}
\gh(\wh, \vh) & := \sum_{K \in \Th} \int_K \nabla \wh \cdot \nabla \vh \dx, \\
 \ch(\wh, \vh) & := -\frac12 \sum_{F \in \cF_I} \Big[\cI_\trd(\wh^-,\vh^- - \vh^+)
  -\cI_\tru(\wh^+ , \vh^- - \vh^+)\Big] - \sum_{F \in \cF_B} \cI_{\trd}(\wh, \vh), \\
 \ph(\wh, \vh) & := \sum_{F \in \cF_I} \frac{\eta}{\hF^d} \int_F (\wh^- - \wh^+) (\vh^- - \vh^+) \dH^d + \sum_{F \in \cF_B} \frac{\eta}{\hF^d} \int_F \wh \vh \dH^d.
\end{alignat*}
 
The ordering of the basis in \eqref{e:Basis} implies that $\sG$, $\sC$, and~$\sP$ are $\NT\times\NT$ block matrices, with $\Np\times\Np$ blocks.  
Likewise, the load vector~$\sfb$ is an $\NT\times 1$ block vector, with $\Np\times 1$ blocks. 
 We now consider the block-wise assembly of $\sG$, $\sC$, $\sP$ and $\mathsf{b}$, using integration formulas and quadrature rules on the reference domains.

\subsection{Integration over reference domains} 
 
As we shall explain below, almost all of the integrals over fractal domains arising in our formulation can be 
expressed in terms of a finite set of canonical integrals of polynomial basis functions (monomials, in our implementation) over either the reference element $\widehat{K}=\Omega$, one of the six wedges $W_1,\ldots,W_6$, 
 or the reference face $\Gamma$. These canonical integrals can be evaluated exactly using self-similarity arguments similar to those used in \cite{barnsley1985iterated, vrscay1989iterated,strichartz2000evaluating},   combined with symmetry arguments in the case of the wedges $W_1,\ldots,W_6$. 
The resulting integration formulas, and some details of their derivations, are presented in Appendix \ref{app:Integration} for the cases required to allow us to implement our SIP-DG-FEM for~$p = 1$ and~$p = 2$.
The derivation of analogous formulas permitting the implementation of our method for higher $p$ values is straightforward, but is left for future work. 
Application of our integration formulas to a given polynomial requires  to express the polynomial as a linear combination of the (monomial) basis to which the integration formulas apply. In some instances this requires a small linear solve to find the basis coefficients.

\subsection{Assembly of~\texorpdfstring{$\sG$}{G}}  \label{ss:G}
 
The matrix~$\sG$ has a block-diagonal structure.
Using the fact that~$R_K^{\top} R_K$ is the identity matrix and~$\det((\hK/2) R_K) = \hK^2/4$ for each $K\in\cT$, the $m$th diagonal block of~$\sG$ can be computed as follows (with $i,j\in\{1,\ldots,\Np\}$):
\begin{alignat}{3}
\label{e:G-diag}
( \sG_{m m})_{ij} =  \int_{K_{m}} \nabla \phi_i^{m} \cdot  \nabla \phi_j^{m} \dx  & =  \int_{\hatK} \nabla \hphi_i \cdot \nabla \hphi_j \dhx .\end{alignat}
The right-hand side of \eqref{e:G-diag} involves integrals of polynomials of degree at most~$2(p-1)$ over the reference element $\hatK$, which can be evaluated (for~$p = 1$ and~$p = 2$) using the integration formulas in  Lemma~\ref{lem:SnowQuad}.
Note that the right-hand side of \eqref{e:G-diag} is independent of $m$, and hence only needs to be evaluated once for each pair $(i,j)$. Similar comments apply to later formulas, e.g.\ \eqref{e:C-mm}.

\subsection{Assembly of~\texorpdfstring{$\sC$}{C}} \label{ss:C}
 Each internal face~$F \in \cF_I$ shared by the elements~$K_{m}$ and~$K_n$ in~$\Th$, where we assume that~$\trd \subset K_{m}$ and~$\tru \subset K_n$, makes the following contributions to the blocks $\sC_{mm}$, $\sC_{mn}$, $\sC_{nm}$ and $\sC_{nn}$ of the consistency matrix~$\sC$:
  \begin{alignat*}{3}
(\sC_{mm}^{(F)})_{ij} & = -\frac12 \cI_{\trd}(\phi_j^{m}, \phi_i^{m}) ,  & & \quad 
(\sC_{m n}^{(F)})_{ij} = \frac12 \cI_{\tru} (\phi_j^{n}, \phi_i^{m}), \\
 (\sC_{n m}^{(F)})_{ij} &  = \frac12 \cI_{\trd} (\phi_j^{m}, \phi_i^{n}) , & & \quad
(\sC_{n n}^{(F)})_{ij} = -\frac12 \cI_{\tru} (\phi_j^{n}, \phi_i^{n}).
\end{alignat*}

Recalling \eqref{e:IdownIup}, the block-diagonal contribution~$\sC_{mm}^{(F)}$ can be split into a sum of three terms as
\begin{alignat}{3}
 (\sC_{m m}^{(F)})_{ij} & = -\frac12 \int_{\trd} \dive\big[\phi_i^{m} \nabla \phi_j^{m}\big] \dx  + \frac12 \int_{\vee} \phi_i^{m} \nabla \phi_j^{m} \cdot \bn_{\vee} \ds   \nonumber\\
   & = -\frac12 \int_{\trd} \nabla \phi_i^{m} \cdot \nabla \phi_j^{m} \dx - \frac12 \int_{\trd} \phi_i^{m} \Delta \phi_j^{m} \dx + \frac12 \int_{\vee} \phi_i^{m} \nabla \phi_j^{m} \cdot \bn_{\vee} \ds  \nonumber\\
  \label{e:C-mm}
 & = -\frac12 \int_{\widehat{\trd}} \nabla \hphi_i \cdot \nabla \hphi_j \dhx - \frac12 \int_{\widehat{\trd}} \hphi_i \Delta \hphi_j \dhx
 + \frac12 \int_{\widehat{\vee}} \hphi_i \nabla \hphi_j \cdot \bn_{\widehat{\vee}} \dhS,
  \end{alignat}
where $\widehat{\trd}=\psi_{K_m}^{-1}(\trd)$ is the reference wedge and 
$\widehat{\vee}=\psi_{K_m}^{-1}(\vee)$ 
is the corresponding union of reference line segments.  
           The first two terms in \eqref{e:C-mm} require integration formulas for polynomials of degree at most~$2(p - 1)$ on $\widehat{\trd}$, 
   which are provided in Lemma~\ref{lem:WedgeQuad} for the cases $p=1$ and $p=2$. 
 To simplify the third term in~\eqref{e:C-mm}, we have used the identities~$\nabla \phi_j^m = (2/h_{K_m})R_{K_m} \nabla \hphi_j$, $\bn_{\vee} = R_{K_m} \bn_{\widehat{\vee}}$, and~$\ds = (h_{K_m}/2) \dhS$.
This term involves integrals of polynomials of degree at most $2p-1$ over $\widehat{\vee}$, which can be computed exactly using standard Gauss--Legendre quadrature rules on the line segments making up $\widehat{\vee}$.  
The block-diagonal contribution~$\sC_{nn}^{(F)}$ can be computed analogously, with $\trd$, $\widehat{\trd}$, $\vee$, and $\widehat{\vee}$ replaced by $\tru$, $\widehat{\tru}$, $\wedge$, and $\widehat{\wedge}$. 

The computation of the off-diagonal contributions~$\sC_{nm}^{(F)}$ is more involved, since it requires the extension
to the wedge~$\trd \subset K_{m}$ of the basis functions $\{\phi_i^n\chi_{K_n}\}_{i=1}^{\Np}$, which are supported in $K_n$.  
More precisely, the contribution~$\sC_{nm}^{(F)}$ can be split into three terms as follows: 
\begin{alignat}{3}
\nonumber
(\sC_{nm}^{(F)})_{ij}&  = \frac12 \int_{\trd} \dive \Big[\phi_i^n \nabla \phi_j^{m}\Big] \dx - \frac12 \int_{\vee} \phi_i^{n} \nabla \phi_j^{m} \cdot \bn_{\vee} \ds\\
   \nonumber
 & = \frac12 \int_{\trd} \nabla \phi_i^{n} \cdot \nabla \phi_j^{m} \dx + \frac12 \int_{\trd} \phi_i^{n} \Delta \phi_j^{m} \dx - \frac12 \int_{\vee} \phi_i^{n} \nabla \phi_j^{m} \cdot \bn_{\vee} \ds  \\
   \nonumber
 & = 
  \frac{\sqrt{3}}{2}
 \int_{\widehat{\trd}} {R_{K_n} \nabla \hphi_i \big((\psi_{K_n}^{-1} \circ \psi_{K_{m}} )(\widehat{\bx}) \big)\cdot R_{K_{m}} \nabla \hphi_j}  \dhx 
 \\
 \nonumber
 & \quad + \frac12\int_{\widehat{\trd}} \hphi_i\big((\psi_{K_n}^{-1}\circ \psi_{K_{m}})(\widehat{\bx})\big) \Delta \hphi_j \dhx  \\
 \label{e:C-mn}
 & \quad - 
    \frac12 \int_{\widehat{\vee}} \hphi_i \big((\psi_{K_n}^{-1} \circ \psi_{K_{m}})(\widehat{\bx})\big) \nabla \hphi_j \cdot \bn_{\widehat{\vee}} \dhS,
\end{alignat}
 where we used the fact that the ratio~$h_{K_m}/h_{K_{n}} = \sqrt{3}$ for an LQU mesh. 
  As before, calculation of the first two terms requires formulas for the integration of polynomials of degree at most~$2(p - 1)$ on the reference wedge~$\widehat{\trd}$, which can be carried out using the results in Appendix~\ref{app:Integration}.
However, an important difference compared to 
the calculation of the analogous terms in 
\eqref{e:C-mm} is that in \eqref{e:C-mn} the functions $\hphi_i((\psi_{K_n}^{-1}\circ \psi_{K_{m}})(\widehat{\bx}))$ must first be expressed in the polynomial basis on the reference element before the integration formulas can be applied.
Evaluation of the third term in \eqref{e:C-mn} involves integrals of polynomials of degree at most $2p-1$ over $\widehat{\vee}$, which can be computed exactly using Gauss--Legendre rules, as for the analogous term in \eqref{e:C-mm}. 
    The off-diagonal block contribution~$\sC_{mn}^{(F)}$ can be computed analogously.

Each boundary face~$F \in \cF_B$ with~$F \subset \deK_{m}$ contributes only to the diagonal block~$\sC_{mm}$ as
\begin{equation*}
\sC_{mm}^{(F)} = - \cI_{\trd} (\phi_j^{m}, \phi_i^{m}) ,
\end{equation*}
which can be computed exactly as for the interior faces, using \eqref{e:C-mm}.

\subsection{Assembly of~\texorpdfstring{$\sP$}{P}} \label{ss:P}
 In the penalty matrix~$\sP$, each interior face~$F \in \cF_I$, shared by the elements~$K_{m}$ and~$K_{n}$ in~$\Th$ with~$\trd \subset K_{m}$ and~$\tru \subset K_n$, makes the following contributions to the blocks $\sP_{mm}$, $\sP_{mn}$, $\sP_{nm}$ and $\sP_{nn}$:
\begin{alignat*}{3}
( \sP_{mm}^{(F)} )_{ij} & = \frac{\eta}{\hF^d} \int_F \phi_i^{m} \phi_j^{m} \dH^d , & & \qquad 
(\sP_{m n}^{(F)})_{ij}  = -\frac{\eta}{\hF^d}\int_F \phi_i^{m} \phi_j^{n} \dH^d, \\
(\sP_{n m}^{(F)})_{ij} & = -\frac{\eta}{\hF^d} \int_F \phi_i^n \phi_j^{m} \dH^d , & & \qquad 
(\sP_{nn}^{(F)})_{ij} = \frac{\eta}{\hF^d} \int_F \phi_i^{n} \phi_j^{n} \dH^d .
\end{alignat*}
These contributions can be computed by expressing the integrals as integrals over the reference face $\Gamma$. 
 Special attention must be paid to the evaluation of the basis functions from the two sides of~$F$, which is a ``short face'' for~$K_{m}$ and a ``long face'' for~$K_n$ (see Figures~\ref{fig:SnowflakeBoundaryColor} and~\ref{fig:wedgesB}(a)).

Let~$\widehat{F}^{-}$ and~$\widehat{F}^+$ be the corresponding faces of the reference element~$\hatK$ such that~$\psi_{K_{m}}(\widehat{F}^-) = F$ and~$\psi_{K_n}(\widehat{F}^+) = F$, and let~$\xi_{F} : \Gamma \to F$ be the similarity defined in \eqref{e:Face}. 
     Then the four contributions above can be expressed, 
using the scaling property of $\cH^d$ \cite[Scaling Property 3.2]{Fal}, 
as
\begin{alignat*}{3}
(\sP_{mm}^{(F)})_{ij} & = \eta \int_{\Gamma} \hphi_i\big((\psi_{K_{m}}^{-1} \circ \xi_F)(\widehat{\bx})\big) \hphi_j\big((\psi_{K_{m}}^{-1} \circ \xi_F)(\widehat{\bx})\big) \dhH^d,\\
 (\sP_{mn}^{(F)})_{ij} & = - \eta \int_{\Gamma}  \hphi_i\big((\psi_{K_{m}}^{-1} \circ \xi_F)(\widehat{\bx})\big) \hphi_j\big((\psi_{K_{n}}^{-1} \circ \xi_F)(\widehat{\bx})\big) \dhH^d , \\
 (\sP_{n m}^{(F)})_{ij} & = - \eta \int_{\Gamma}  \hphi_i\big((\psi_{K_{n}}^{-1} \circ \xi_F)(\widehat{\bx})\big) \hphi_j\big((\psi_{K_{m}}^{-1} \circ \xi_F)(\widehat{\bx})\big) \dhH^d , \\
 (\sP_{nn}^{(F)})_{ij} & = \eta \int_{\Gamma} \hphi_i\big((\psi_{K_{n}}^{-1} \circ \xi_F)(\widehat{\bx})\big) \hphi_j\big((\psi_{K_{n}}^{-1} \circ \xi_F)(\widehat{\bx})\big) \dhH^d,
\end{alignat*}
     where the factor~$\hF^{-d}$ canceled out after the change of variables. 
 Since~$(\psi_{K_m}^{-1} \circ \xi_F) : \Gamma \to \widehat{F}^{-}$ and $(\psi_{K_{n}}^{-1} \circ \xi_F) : \Gamma \to \widehat{F}^{+}$, only evaluations of the basis functions~$\{\hphi_i\}$ on the faces~$\widehat{F}^{\pm}$ are required.
 The integrals on the reference faces are computed using the integration formulas in Lemma~\ref{lem:KochQuad},  after expressing the functions $\hphi_i\circ\psi_{K_{n}}^{-1}\circ \xi_F$ etc.\ in the (monomial) basis. 
 
Each boundary face~$F \in \cF_B$ with~$F \subset \deK_{m}$ contributes only to the diagonal block~$\sP_{m m}$, exactly as for $\sP_{mm}^{(F)}$ above in the interior case. 

\subsection{Assembly of \texorpdfstring{$\sfb$}{b}} \label{ss:b}
 Each element~$K_{m}$ contributes to the $m$th block $\sfb_m$ of the load vector~$\sfb$ as follows:
\begin{equation}
\label{e:b}
(\sfb_{m})_i = \int_{K_{m}} f \phi_i^{m} \dx = \Big(\frac{h_{K_{m}}}{2}\Big)^2 \int_{\hatK} f\big(\psi_{K_m}(\widehat{\bx})\big)\, \hphi_i(\widehat\bx)\, \dhx ,
\end{equation}
where the second identity is obtained by mapping to the reference element~$\hatK$. In general, the integrand in \eqref{e:b} is not a polynomial. Hence the integration formulas in Appendix \ref{app:Integration} cannot be applied. Instead, for these integrals, we use the composite barycentre quadrature rule from~\cite[\S3.1]{disjoint} on a suitably fine $\cT_\ell$ mesh of $\widehat{K}$ (in our experiments in \S\ref{s:Numer}, we use $\ell=4$).

\section{Numerical results}
\label{s:Numer}
 
We now present numerical experiments demonstrating that our SIP-DG-FEM works in practice, and assessing its stability and convergence. 
Our implementation of \eqref{e:DG} uses the quadrature described in \S\ref{s:Quad}, polynomial degrees $p=1$ and $p=2$, and the mesh sequences $\cT_\ell'$ and $\cT_{\ell,\ell^*}'$ of \S\ref{ss:MeshSeq}. In all the experiments the penalty parameter in \eqref{e:aDefDisc} is set to~$\eta = 15$.
The SIP-DG-FEM is implemented in Matlab and the code is available at \url{https://github.com/SergioG5/FractalDG}.
  
 \subsection{Convergence for a smooth solution}
\label{s:smooth}

\begin{figure}[t]
\centering
\includegraphics[width = 0.45\textwidth]{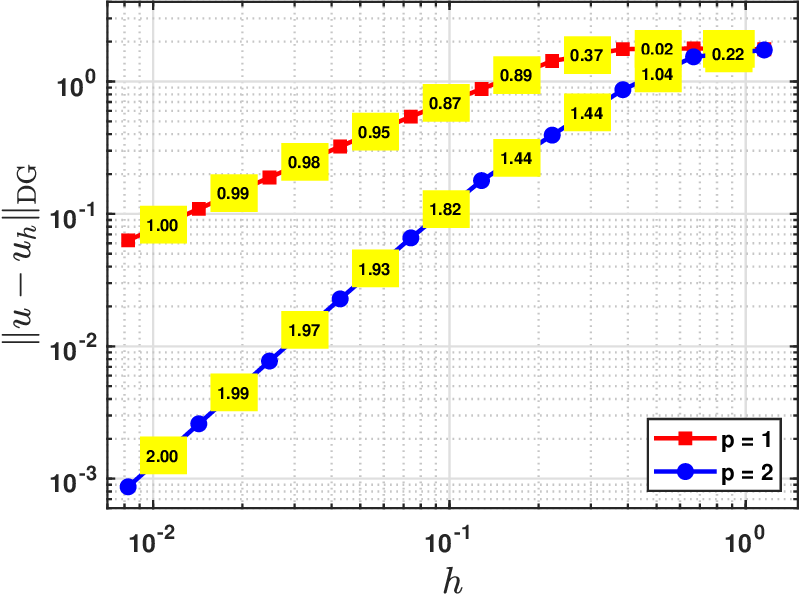}
\hspace{0.1cm}
\includegraphics[width = 0.45\textwidth]{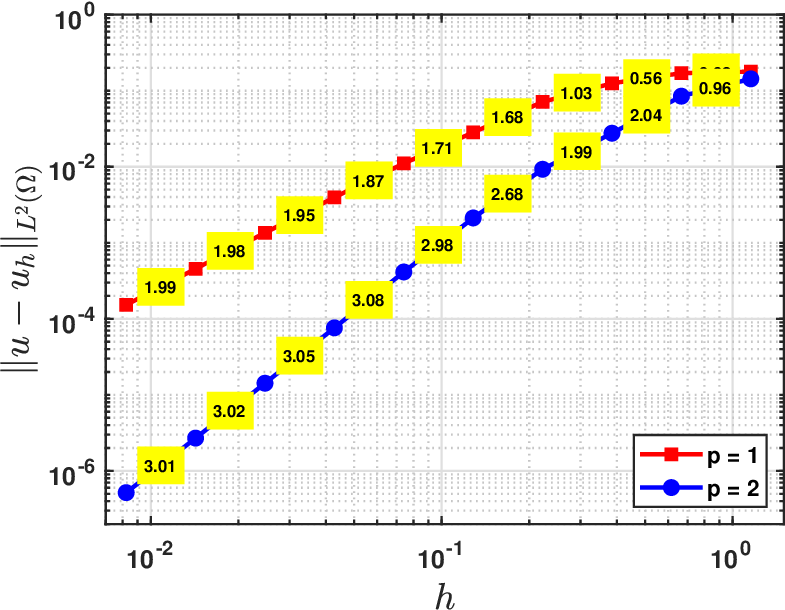}
\caption{$h$-convergence for the problem with smooth (Gaussian) solution (see \S\ref{s:smooth}) in the DG norm (left panel) and the~$L_2(\Omega)$ norm (right panel). The numbers in yellow rectangles are the associated empirical convergence rates. }
\label{fig:convergence}
\end{figure}

In our first experiment, we assess the accuracy of the proposed method in the case of a smooth solution.
We consider a manufactured problem with a source term~$f$ chosen so that the solution to~\eqref{e:BVP101} is given by the Gaussian function~$u(x, y) = \exp(-(x^2 + y^2)/\sigma^2)$, with~$\sigma = 10^{-1}$. 
This solution does not vanish exactly on $\partial\Omega$, but is numerically small there ($|u|\leq 10^{-14}$ on $\partial\Omega$).

In Figure~\ref{fig:convergence}, we show (on~\emph{log-log} axes) the errors $\|u-u_h\|$ in the DG norm~\eqref{e:Coercivity} and in the~$L_2(\Omega)$ norm for the sequence of meshes~$\{\cT_\ell'\}_{\ell = 0}^{9}$ (see Figure~\ref{fig:MeshSeq}, panels (a) and (d)--(f) for $\ell=1,\ldots,4$) and approximations of polynomial degree~$p = 1$ and~$p = 2$.
The results show optimal convergence rates of order~$\Order{h^{p}}$ and~$\Order{h^{p+1}}$ in the DG norm and the~$L_2(\Omega)$ norm, respectively.

 \subsection{Convergence for a singular solution}
\label{s:sing}

\begin{figure}[t]
\centering
 \begin{subfigure}[t]{0.45\textwidth}
\includegraphics[width = \textwidth]{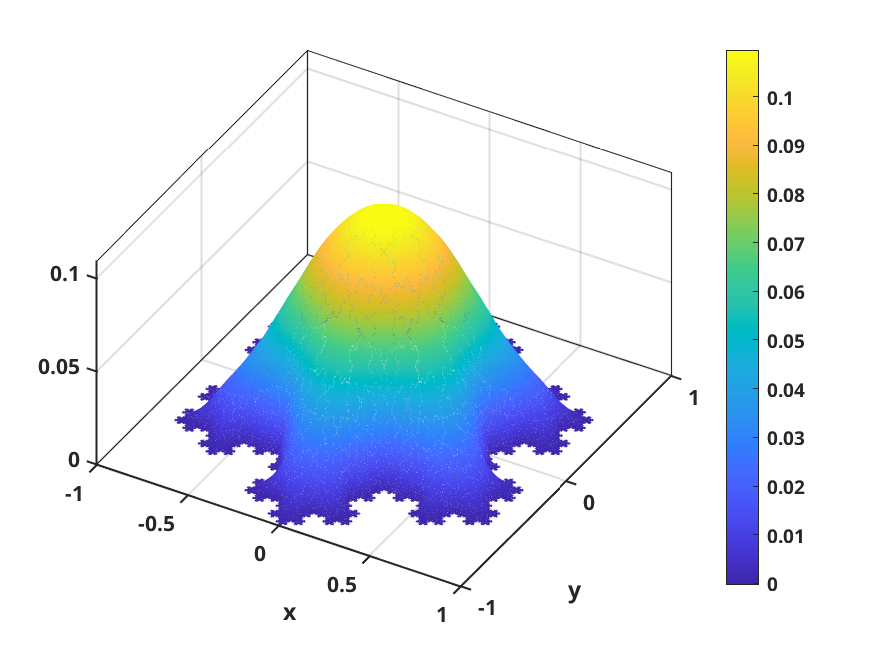}
\caption{{}}
\end{subfigure}
\hspace{0.1cm}
\begin{subfigure}[t]{0.45\textwidth}
\includegraphics[width = \textwidth]{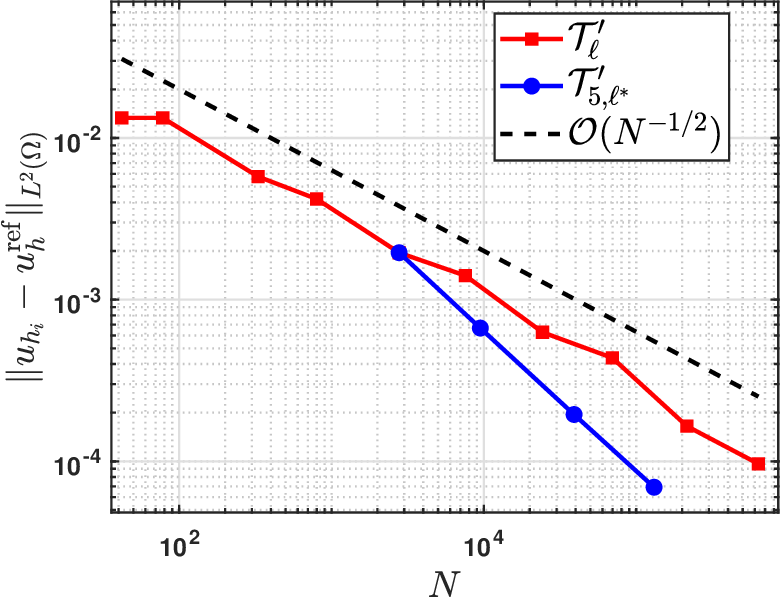}    
\caption{{}}
\end{subfigure}
\caption{(a) Discrete approximation~$\uh$ with  $p =2$ on the mesh~$\cT'_{3,5}$ for the problem with~$f = 1$ (see \S\ref{s:sing}). 
(b) $L_2\OO$ norms of the errors of the solutions computed on the mesh sequences~$\{\cT'_{\ell}\}_{\ell = 1}^{10}$ and~$\{\cT'_{5,\ell^*}\}_{\ell^* = 0}^{3}$, with reference solution computed on $\cT'_{11}$.}
\label{fig:convergence-singular}
\end{figure}

We now consider the BVP~\eqref{e:BVP101} with~$f = 1$, for which the solution $u$ is smooth inside $\Omega$ but singular on $\partial\Omega$ (in the sense that $u\not\in H^2(\Omega)$).
In Figure~\ref{fig:convergence-singular}(a), we display the discrete approximation~$\uh$ obtained for $p = 2$ on the mesh~$\cT'_{3,5}$.

Since the continuous solution~$u$ is not known in closed form, we cannot compute $\|u-u_h\|$ in this case.
Instead, in Figure~\ref{fig:convergence-singular}(b), we plot (on \emph{log-log} axes against the number of degrees of freedom~$N$) the mesh-independent norm $\|u_{h_i}-u_h^{\mathrm{ref}}\|_{L_2\OO}$, for the reference numerical solution $u_h^{\mathrm{ref}}$ computed on the mesh $\cT'_{11}$ and the solutions $u_{h_i}$ computed on the mesh sequences~$\{\cT'_{\ell}\}_{\ell = 1}^{10}$ (red curve) and~$\{\cT'_{5,\ell^*}\}_{\ell^* = 0}^3$ (blue curve), using polynomial approximations of degree~$p = 2$.

For the sequence~$\{\cT'_{\ell}\}_{\ell = 1}^{10}$ of quasi-uniform meshes, a decay of approximate order at least $\mathcal{O}(N^{-1/2})=\mathcal{O}(h)$ is observed. 
Our theoretical result \eqref{e:L2}  predicts decay like $\mathcal{O}(h^{2(1-\mu)})$, for $\mu>\mu_*\approx 0.828$, based on the solution regularity results in \cite{NystromThesis,capitanelli2015weighted}. The numerical results in Figure~\ref{fig:convergence-singular}(b) are consistent with \eqref{e:L2} holding for $\mu\approx 1/2$, suggesting that the regularity results in \cite{NystromThesis,capitanelli2015weighted} may not be sharp.
A faster decay is observed for the sequence~$\{\cT'_{5,\ell^*}\}_{\ell^* = 0}^{3}$ of meshes refined towards the boundary.
We note however that the observed rates might be affected by the error present in the reference solution $u_h^{\mathrm{ref}}$.

\subsection{Conditioning of the Galerkin matrix}
\label{s:cond}
 
\begin{figure}[!t]\centering
\includegraphics[width = 0.5\textwidth]{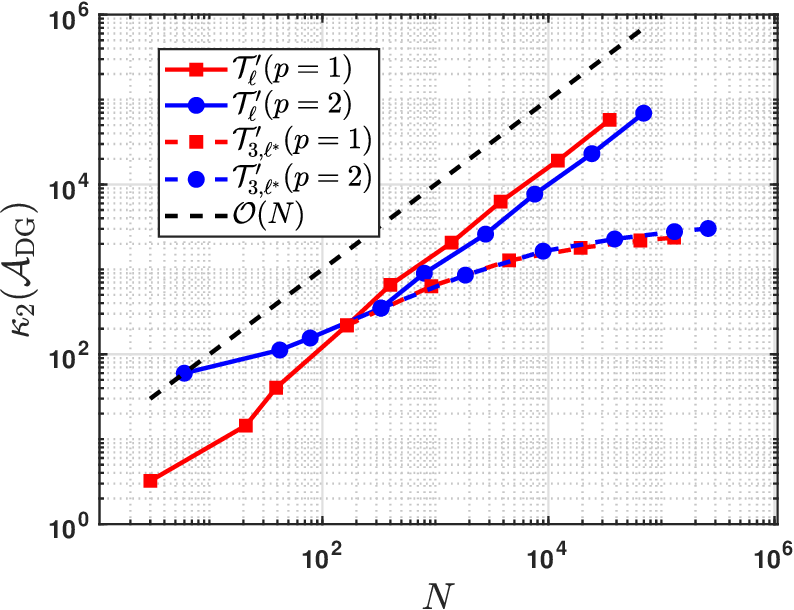}
\caption{Condition number of the Galerkin matrix~$\ADG$ for the sequence of quasi-uniform meshes~$\{\cT_\ell'\}_{\ell = 0}^{8}$ and the sequence of boundary-refined meshes~$\{\cT'_{3, \ell^*}\}_{\ell^* = 0}^5$ (see \S\ref{s:cond}).\label{fig:conditioning}}
\end{figure}

We now study numerically the condition number of the Galerkin matrix.
In Figure~\ref{fig:conditioning}, we show (on \emph{log-log} axes) the~$2$-norm condition number of the Galerkin matrix~$\ADG$ in~\eqref{e:ADG} for the mesh sequences $\{\cT_\ell'\}_{\ell = 0}^8$ and $\{\cT'_{3, \ell^*}\}_{\ell^* = 0}^5$, with approximations of degree~$p = 1$ and~$p = 2$.
For the quasi-uniform sequence, the condition number grows like $\Order{N}$ (which corresponds to~$\Order{h^{-2}}$), as for the SIP-DG-FEM  on standard domains (see~\cite[\S3.4]{Castillo:2002}).
In contrast, for the boundary-refined meshes, a milder growth of the condition number with increasing $N$ is observed.

 \subsection{Approximation of eigenvalues}\label{ss:eigen}

Our final experiment concerns the approximation of the Dirichlet eigenvalue problem: find~$(\lambda, \phi) \in \R^+ \times H_0^1(\Omega)$ such that
\begin{alignat}{3}\label{e:eigen}
-\Delta \phi & = \lambda \phi & & \qquad \text{ in } \Omega.
\end{alignat}
A discrete approximation of the eigenpairs~$(\lambda, \phi)$ can be obtained by solving the generalized eigenvalue problem: find~$(\lambda_h, \phi_h) \in \R^+ \times \Vh$ such that
\begin{equation*}
\ADG \Phi_h = \lambda_h \sM \Phi_h,
\end{equation*}
where~$\Phi_h$ is the vector representation of~$\phi_h$, and~$\sM$ is the mass matrix (representing the~$L_2(\Omega)$ inner product) for the space~$\Vh$.

Previous numerical studies of this problem include \cite{Banjai:2007} (based on conformal mappings of a polygonal prefractal approximation of $\Omega$) and \cite{lapidus1996snowflake,Neuberger_Sibien_Swift:2006} (based on finite difference approximations on a triangular grid of points associated with a prefractal approximation). 
Since 
the numerical results in~\cite{Banjai:2007,Neuberger_Sibien_Swift:2006} are carried out on the domain~$\widetilde{\Omega} = (1/\sqrt{3}) \Omega$, we must scale our discrete eigenvalues as~$\widetilde{\lambda}_h = 3 \lambda_h$ to get results comparable to those in \cite{Banjai:2007,Neuberger_Sibien_Swift:2006}.
 
We carry out computations using both the quasi-uniform mesh~$\cT_{10}'$ and the boundary-refined mesh~$\cT'_{4,3}$. 
In Table~\ref{tab:eigenvalues}, we show the approximation of the first 10 eigenvalues obtained using approximations of degree~$p = 2$, and compare them with the values reported in~\cite[Table 2]{Banjai:2007}. The relative errors obtained for both meshes are comparable, and are less than $1\%$ in all cases considered.
However, the boundary-refined mesh~$\cT'_{4, 3}$ has a substantially smaller number of elements, highlighting the efficiency of this type of mesh for obtaining accurate solutions with singular behaviour near~$\partial\Omega$.

\begin{table}[t]
  \centering
  \caption{Comparison of the first 10 eigenvalues of \eqref{e:eigen} obtained with different meshes.\label{tab:eigenvalues}}
  \begin{tabular}{lr|rc|rc}
  \hline
  & & \multicolumn{2}{c|}{$\card(\cT'_{4, 3}) = 6793$}
    & \multicolumn{2}{c}{$\card(\cT_{10}') = 105469$} \\
  $k$ & $\lambda_{\mathrm{ref}}^{(k)}$ \cite{Banjai:2007}
      & $\widetilde{\lambda}_h^{(k)}$ & Relative error
      & $\widetilde{\lambda}_h^{(k)}$ & Relative error \\
  \hline
  1  &  39.348553 &  39.432 & 2.12e-03 &  39.402 & 1.36e-03 \\
  2  &  97.43691  &  97.714 & 2.84e-03 &  97.573 & 1.40e-03 \\
  3  &  97.43691  &  97.714 & 2.84e-03 &  97.573 & 1.40e-03 \\
  4  & 165.4058   & 165.953 & 3.31e-03 & 165.648 & 1.46e-03 \\
  5  & 165.4058   & 165.953 & 3.31e-03 & 165.648 & 1.46e-03 \\
  6  & 190.370    & 191.414 & 5.48e-03 & 190.651 & 1.48e-03 \\
  7  & 208.608    & 209.380 & 3.70e-03 & 208.939 & 1.59e-03 \\
  8  & 272.406    & 274.060 & 6.07e-03 & 272.832 & 1.56e-03 \\
  9  & 272.406    & 274.060 & 6.07e-03 & 272.832 & 1.56e-03 \\
  10 & 312.3538   & 314.979 & 8.40e-03 & 312.766 & 1.32e-03 \\
  \hline
  \hline
  \end{tabular}
\end{table}

In Figure~\ref{fig:eigenfunctions}, we show some discrete eigenfunctions corresponding to the eigenvalues in  Table~\ref{tab:eigenvalues} (third column) for the mesh~$\cT'_{4, 3}$. These exhibit the same symmetries observed in~\cite[Fig. 6]{Neuberger_Sibien_Swift:2006}.

\begin{figure}[!t]
\centering
\subfloat[$\phi_h^{(1)}$]{
\includegraphics[width = 0.3\textwidth, clip, trim = {70 20 70 20}]{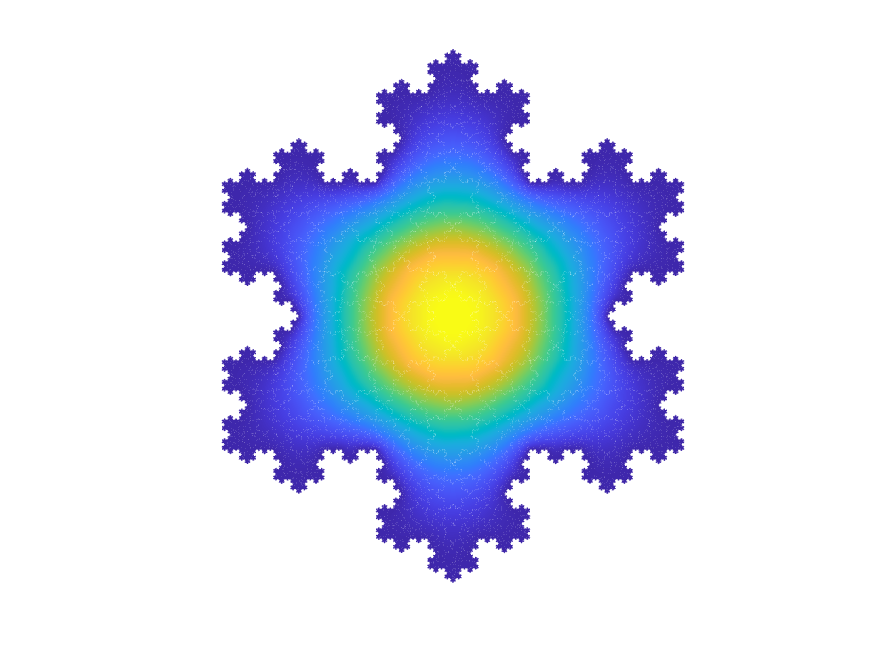}
}
\hspace{0.1cm}
\subfloat[$\phi_h^{(2)}$]{
\includegraphics[width = 0.3\textwidth, clip, trim = {70 20 70 20}]{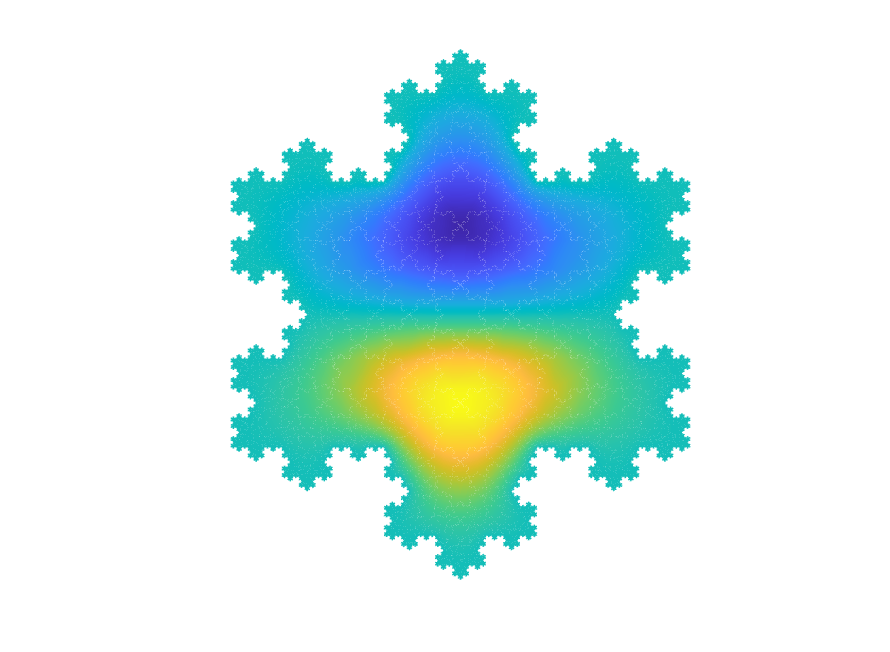}
}
\hspace{0.1cm}
\subfloat[$\phi_h^{(3)}$]{
\includegraphics[width = 0.3\textwidth, clip, trim = {70 20 70 20}]{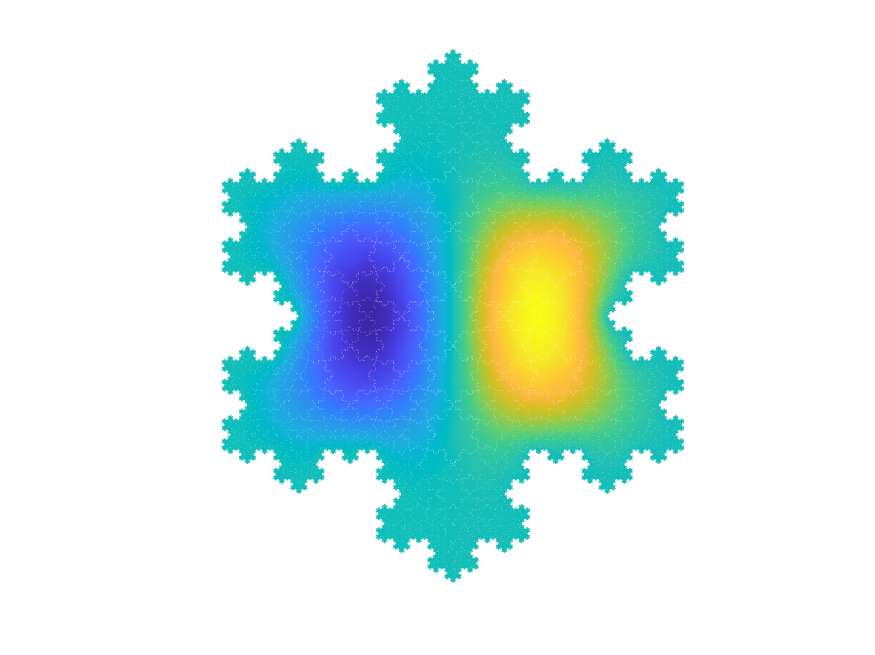}
}\\
\subfloat[$\phi_h^{(4)}$]{
\includegraphics[width = 0.3\textwidth, clip, trim = {70 20 70 20}]{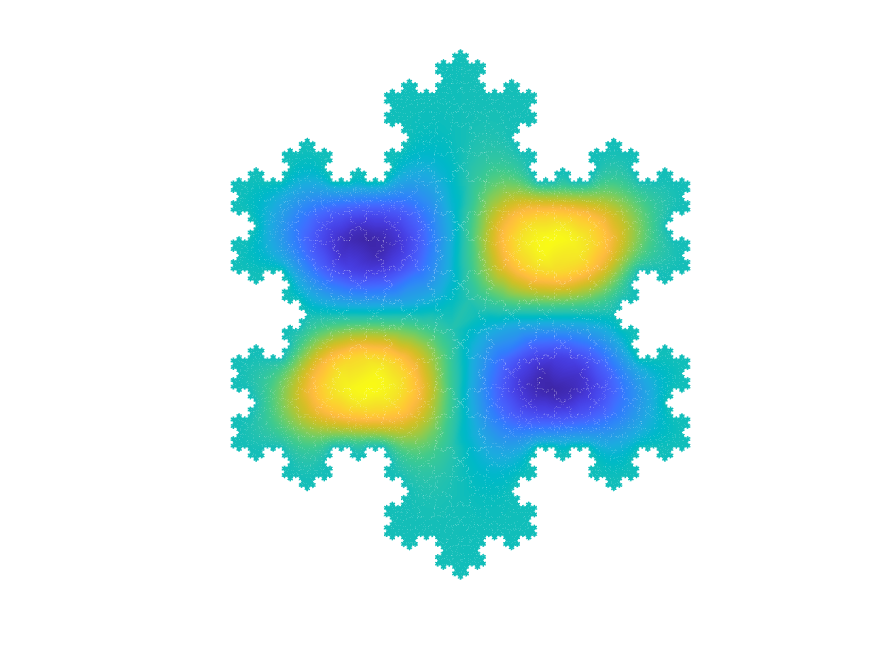}
}
\hspace{0.1cm}
\subfloat[$\phi_h^{(5)}$]{
\includegraphics[width = 0.3\textwidth, clip, trim = {70 20 70 20}]{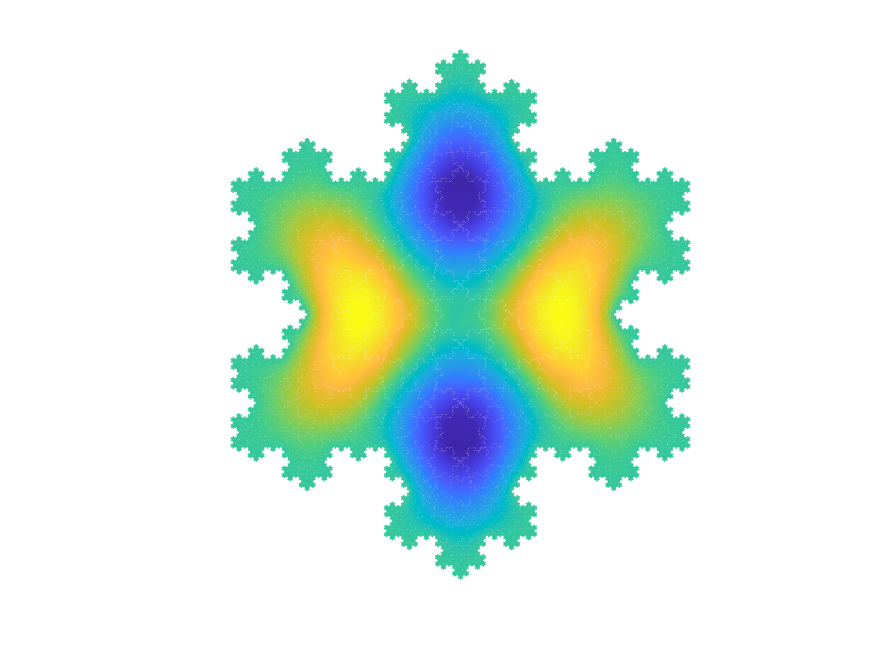}
}
\hspace{0.1cm}
\subfloat[$\phi_h^{(10)}$]{
\includegraphics[width = 0.3\textwidth, clip, trim = {70 20 70 20}]{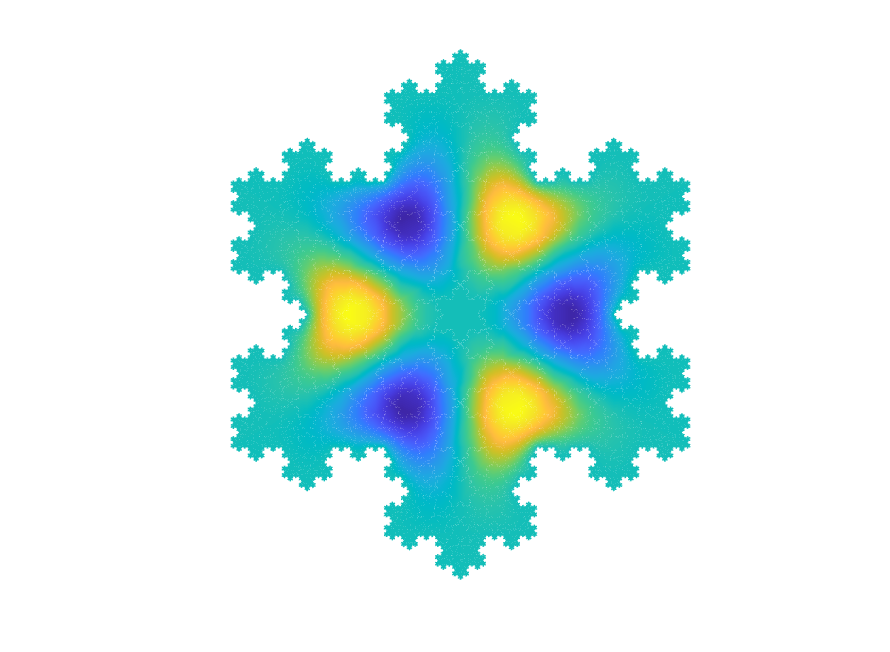}
}
\caption{Discrete eigenfunctions computed with approximations of degree~$p = 2$ on the mesh~$\cT'_{4, 3}$.\label{fig:eigenfunctions}}
\end{figure}

\section{Future work}\label{s:Future}
We have presented just the first steps towards the design and the analysis of a wider class of DG schemes employing elements with fractal boundaries.
The following are some possible extensions.
\begin{itemize}\setlength\itemsep{-.5mm}\setlength\leftmargin{-5mm}
\item A reduction of the exponent $\mu$ in the results of \cite{capitanelli2015weighted,NystromThesis}, would sharpen the convergence estimates \eqref{e:Rates}, \eqref{e:Simple} and \eqref{e:L2} for the SIP-DG-FEM.

\item Higher-order BVP solution regularity analysis and convergence of higher-degree polynomials.

\item Convergence and computational cost analysis for (possibly adaptive) boundary-refined meshes.

\item Design of high-order quadrature rules on snowflakes, wedges, and Koch curves (see \cite{joly2024high}).

\item More general domains and elements with fractal boundaries, such as those in \cite[Figure~1]{CaChHe25} (fudgeflake, Gosper island, twindragon, terdragon, Levy dragon\ldots) and the square snowflake in \cite[\S5.2]{caetano2019density}.

\item Different BVPs: e.g.,\ other elliptic PDEs and boundary conditions, fractal interfaces, rough coefficients, and evolution problems.

\item
Further applications to spectral analysis, e.g.\ for Neumann and Steklov problems.

\item Different DG formulations, such as local-DG.

\item Fast solvers exploiting the regular structure of the mesh, possibly allowing matrix-free matrix--vector products in iterative schemes.
See, e.g.,\ in \cite{bannister2026scattering} an example for the discretisation of an integral equation on a fractal mesh.

\item Homogenisation and upscaling:
can a fractal boundary be replaced by a smooth one with appropriate ``effective'' boundary conditions such that at a given distance the solution is indistinguishable?
How does the roughness of the boundary affect the properties of BVP solutions away from it?

 \end{itemize}

\appendix

\section{Trace regularity and Poincar\'e inequality} \label{app:A}

In this appendix we prove some technical results concerning elements of the weighted space $H^{2,2}_\mu(\Omega)$ defined in \eqref{e:H22mu}, which we use in the analysis of the SIP-DG-FEM. 

We first prove trace estimates in a sector, for functions with singular behaviour at the apex. 

\begin{lem}[Trace estimates in a sector]
\label{lem:wedge}
Let $U=\{(r,\theta):0<r<1,\, 0<\theta<\theta_*\}$, for some $0<\theta_*<2\pi$, and let $0<\mu<1$.
Suppose that $v\in H^1(U)$ with 
\begin{align}
\label{e:Weighted}
\int_U {|D^2 v|^2} |\bx|^{2\mu}\,\rd x<\infty.
\end{align} 
Then $v\in C^{0,\alpha}(\overline{U})$ for all $0<\alpha<1-\mu$ {and $v\in W^2_q(U)$ for all $1\le q<2/(\mu+1)<2$}.

Further, if either $1/2\leq \mu<1$ and $1\leq q<1/\mu\leq 2$, or $0< \mu <1/2$ and $1\leq q \leq 2$, then $\partial_\bn v \in L_q(\partial U)$ and 
  there exists a constant $C_{q,\theta_*,\mu}>0$ depending only on $q$, $\theta_*$, and $\mu$, such that
 \begin{align*}
       \| \partial_\bn v\|_{L_q(\partial U)}
&\le C_{q,\theta_*,\mu}\big(\|\nabla v\|_{L_2(U)}
+ \big\||\bx|^\mu {D^2} v\big\|_{L_2(U)}
\Big).
\end{align*} 
\end{lem}
\begin{proof}
We note first that, by the H\"older inequality,
\begin{align}
\label{e:Holder1}
\left(\int_U |{D^2} v|^q \,\rd \bx \right)^{1/q} 
\leq \left(\int_U |{D^2} v|^2 |\bx|^{2\mu} \,\rd \bx \right)^{1/2} \left(\int_U |\bx|^{-\widetilde{q}\mu} \,\rd \bx \right)^{1/\widetilde{q}},
\end{align}
whenever $1\leq q<2$ and $2\leq \widetilde{q}<\infty$ with $1/\widetilde{q} + 1/2 = 1/q$.
Thus $v\in W^2_q(U)$, and hence $\partial_\bn v\in L_q(\partial U)$ (by the standard trace theorem, e.g., \ \cite[Thm~1.5.1.2]{Grisvard85}), whenever the second integral on the right-hand side of \eqref{e:Holder1} is finite. This holds if and only if $-\widetilde{q}\mu + 1>-1$, i.e.,\ $\widetilde{q}<2/\mu$, or, equivalently, $q<2/(\mu+1)$. By \cite[eq.~(1,4,4,6)]{Grisvard85} it follows that $v\in C^{0,\alpha}(\overline{U})$ for all $0<\alpha<1-\mu$.

To increase the value of $q$ for which $\partial_\bn v \in L_q(\partial U)$ to that claimed in the statement of the lemma (the point being that $\min(1/\mu,2) > 2/(\mu + 1)$ for $0<\mu<1$), we adopt a different approach. 

We start by defining a dyadic decomposition of $U$ by introducing the sets
\[ U_i:=\big\{(r,\theta):\ 2^{-i}<r<2^{-i+1},\, 0<\theta<\theta_*\big\}, \qquad i\in \N.\]
The assumption \eqref{e:Weighted} implies that $v\in H^2(U_i)$ for every $i\in \N$. Hence, if $q\in[1,2]$, then by the standard trace inequality (e.g., \ \cite[Thm~1.5.1.2]{Grisvard85}) applied on $U_1$, combined with norm scaling (noting that $U_i=2^{1-i}U_1$), there exists $C_{{\rm Tr},\theta_*,q}>0$ such that 
 \begin{align*}
\|\partial_\bn v\|^q_{L_q(\partial U_{i})} \leq C_{{\rm Tr},\theta_*,q} 
\Big(2^{i-1}\|\nabla v\|_{L_q(U_i)}^q + 2^{(1-i)(q-1)} \|D^2  v\|_{L_q(U_i)}^q\Big).
\end{align*}
Next, we use the fact that, since $q\leq 2$, $\|v\|_{L_q(U_i)}^q\le |U_i|^{1-q/2}\|v\|_{L_2(U_i)}^q=(2^{-2i}3\theta_*/2)^{1-q/2}\|v\|_{L_2(U_i)}^q$ for $v\in L_2(U_i)$, so that (noting that $\|y\|_q^q\leq 2^{1-q/2}\|y\|_2^q$ for all $y\in\R^2$)
\begin{align*}
\|\partial_\bn v\|^q_{L_q(\partial U_{i})} &\leq C_{{\rm Tr},\theta_*,q}(3\theta_*/2)^{1-q/2} 
\Big(2^{i(q-1)-q/2}\|\nabla v\|_{L_2(U_i)}^q + 2^{-i+q-1} \|D^2  v\|_{L_2(U_i)}^q\Big),
\end{align*}
and then, since $2^i|\bx|\ge 1$ and $2^{i-1}|\bx|\le 1$ 
 for $\bx\in U_i$, we obtain
   \begin{align*}
&\|\partial_\bn v\|^q_{L_q(\partial U_{i})}\\
&\leq  C_{{\rm Tr},\theta_*,q}(3\theta_*/2)^{1-q/2} 
\Big(2^{-i(1-q\mu){+(1/2-\mu)q}}  \||\bx|^{\mu-1}\nabla v\|_{L_2(U_i)}^q + 2^{-i(1-q\mu)+q-1} \||\bx|^{\mu}D^2 v\|_{L_2(U_i)}^q\Big).
\end{align*}
  Then, bounding $\||\bx|^{\mu-1}\nabla v\|_{L_2(U_i)}\leq \||\bx|^{\mu-1}\nabla v\|_{L_2(U)}$ and
$\||\bx|^{\mu}D^2 v\|_{L_2(U_i)}\leq \||\bx|^{\mu}D^2 v\|_{L_2(U)}$ and summing 
over $i$ gives
\begin{align*}
\|\partial_\bn v\|^q_{L_q(\partial U)}\le \sum_{i=1}^\infty \|\partial_\bn v\|^q_{L_q(\partial U_{i})} 
\leq \frac{C_{{\rm Tr},\theta_*,q}(3\theta_*)^{1-q/2}\, {2^q}} 
{2-2^{q\mu}}
\Big(\||\bx|^{\mu-1}\nabla v\|_{L_2(U)}^q +  \||\bx|^{\mu}D^2 v\|_{L_2(U)}^q
\Big),
\end{align*}
where we assumed that $q<1/\mu$, so that $1-q\mu>0$, making the geometric series {$\sum_{i=1}^\infty2^{-i(1-q\mu)}=\frac{2^{q\mu}}{2-2^{q\mu}}$} summable. 
Let $\chi$ be a smooth cut-off function on $[0,\infty)$ with $\chi=1$ on $[0,1/2]$, $\chi=0$ on $[1,\infty)$, and $0\leq \chi\leq 1$. Then \cite[Lemma 7.1.1]{kozlov1997elliptic} gives the existence of a constant $C_{\rm KMR}>0$ such that, for any $g\in L_2(U)$ for which $|\bx|^\mu \nabla g\in L_2(U)$,
\begin{align*}
\label{}
\||\bx|^{\mu-1}g\|_{L_2(U)}
&\leq \||\bx|^{\mu-1}\chi g\|_{L_2(U)} + \||\bx|^{\mu-1}(1-\chi)g\|_{L_2(U)}\\
& \leq C_{\rm KMR} \||\bx|^{\mu}\nabla(\chi g)\|_{L_2(U)} + 2^{1-\mu}\|g\|_{L_2(U)}\\
& \leq C_{\rm KMR} \Big(\||\bx|^{\mu}\nabla g\|_{L_2(U)} + \|g \|_{L_2(U)}\|\nabla\chi\|_{L_\infty(U)} \Big) + 2^{1-\mu}\|g\|_{L_2(U)}.
\end{align*}
Applying this to the components of $\nabla v$ in the above and taking $q$th roots gives the assertion. 
\end{proof}

Next, we apply Lemma \ref{lem:wedge} to prove trace estimates 
for elements of the space $V^\mu$ defined in \eqref{e:V} on an LQU mesh of the snowflake $\Omega$, on the unions of line segments $\vee$ and $\wedge$, for a given face.

\begin{lem}[Trace estimates on $\vee$ and $\wedge$]
\label{lem:dnu2}
Let either $1/2\leq \mu<1$ and $1\leq q<1/\mu\leq 2$, or $0<\mu <1/2$ and $1\leq q \leq 2$.
Let $v\in V^\mu$ (defined in \eqref{e:V}), and let $\cT$ be an LQU mesh of $\Omega$ with faces $\cF$.
Then there is $\widetilde C_q>0$ depending only on $q$ such that
\begin{align*}
h_F^{1-1/q}\| \partial_\bn v\|_{L_q(\vee)}
&\le \widetilde{C}_{q}\Big(
\|\nabla v\|_{L_2(K_-)}+h_F^{1-\mu} \big\|\dist(\bx,\deK_-)^\mu D^2  v\big\|_{L_2(K_-)}\Big)
&& \forall F\in\cF,
 \\
h_F^{1-1/q}\| \partial_\bn v\|_{L_q(\wedge)}
&\le \widetilde{C}_{q}\Big(
\|\nabla v\|_{L_2(K_+)}+h_F^{1-\mu} \big\|\dist(\bx,\deK_+)^\mu  D^2 v\big\|_{L_2(K_+)}\Big)
&& \forall F\in\cF_I.
\end{align*}
Moreover, the Dirichlet trace of $v$ belongs to $L_\infty(\vee)$ for all $F\in\cF$, and to $L_\infty(\wedge)$ for all $F\in\cF_I$.
\end{lem}

\begin{figure}[tb]\centering 
 (a)\includegraphics[width=40mm]{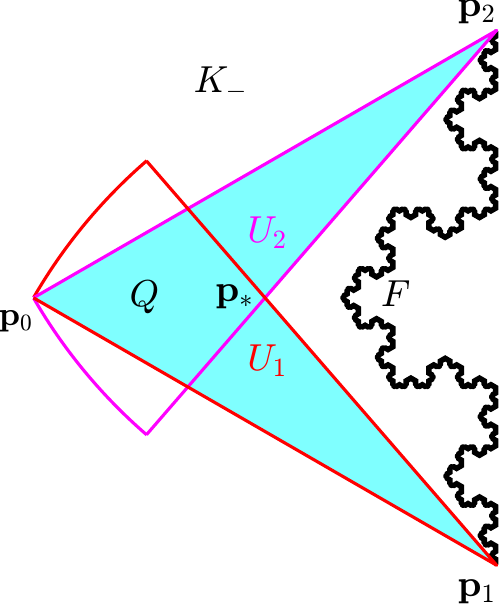}
\qquad\qquad
(b)\includegraphics[width=50mm]{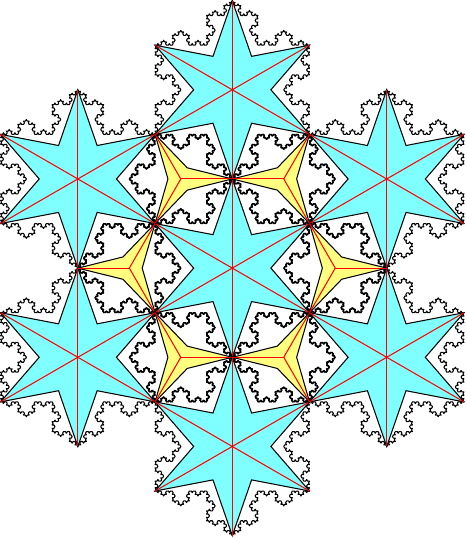}
\caption{(a) The shaded region is the quadrilateral $Q\subset\trd\subset K_-$ associated to the face $F\in\cF$.
Its boundary $\partial Q$ contains $\vee$, which has endpoints $\bp_1,\bp_2$ and apex $\bp_0$.
The sectors $U_1$ and $U_2$ (highlighted in red and magenta) have sides of length $h_F$, are contained in the elements $K_-$ and their boundary covers $\partial Q$.
All points $\bx\in Q$ satisfy the inequality \eqref{e:DistXP}.
\\
(b) The mesh $\cT_2'$ with all the quadrilaterals $Q$ (in light blue) associated to the faces $F\in\cF$ and the quadrilaterals $Q'$ (in yellow) associated to the interior faces $F\in\cF_I$.}
\label{fig:Boomerangs}
\end{figure}

\begin{proof}
We first note that, by Lemma \ref{lem:wedge} and a norm scaling argument, for any $\rho>0$, on the $\rho$-scaled version of the sector $U$ defined in the statement of Lemma \ref{lem:wedge}, i.e.,\ on $\rho U:=\{\rho\bx,\ \bx\in U\}$, we have
\begin{align}\label{e:BoundLqUh}
\rho^{1-1/q}\| \partial_\bn v\|_{L_q(\partial (rU))}
&\le C_{q,\theta_*,\mu}\Big(
\|\nabla v\|_{L_2(rU)}+\rho^{1-\mu} \||\bx|^\mu D^2 v\|_{L_2(rU)}
\Big)
\end{align} 
for all $v\in H^1(\rho U)$ with $|\bx|^\mu D^2v\in L_2(\rho U)$.

For each face $F\in\cF$, and the corresponding wedge $\trd\subset K_-$, denote by $\bp_1$ and $\bp_2$ the endpoints of $\vee$ and by $\bp_0=S_1^-\cap S_2^-$ the apex of $\vee$ (i.e.,\ the barycentre of $K_-$).
The points $\bp_0,\bp_1,\bp_2$ are the vertices of an equilateral triangle with side length $h_F$ and barycentre $(\bp_1+\bp_2+\bp_0)/3\in F$.
Let $\bp_*:=(\bp_1+\bp_2+2\bp_0)/4$.
Then $[\bp_0,\bp_1,\bp_*,\bp_2]$ are the vertices of a non-convex quadrilateral $Q$ contained in $\trd$; see the shaded region in Figure~\ref{fig:Boomerangs}(a).
There exists a constant $0<c<1$, independent of $F$, such that
\begin{align}\label{e:DistXP}
c\,\min\big\{|\bx-\bp_1|,|\bx-\bp_2|\big\}
\le \dist(\bx,F) = \dist(\bx,\deK_-)\le \dist(\bx,\deO)=\delta(\bx)
\qquad \forall \bx\in Q.
\end{align}
With some trigonometry, one can check that $c=\sqrt3/14\approx0.1237$.
Figure~\ref{fig:Boomerangs}(b) shows (in light blue)  the union of all such quadrilaterals for the mesh $\cT_2'$.
   
    For $v\in V^\mu$ and any $F\in\cF$, we apply the
bound~\eqref{e:BoundLqUh} to the two sectors $U_1$ and $U_2$ illustrated in Figure~\ref{fig:Boomerangs}(a), whose intersection $U_1\cap U_2$ equals $Q$. These have their apexes at $\bp_1$ and $\bp_2$, respectively, and are congruent to $\rho U$ for the scaling factor $\rho=|\bp_1-\bp_0|=h_F$ and the aperture $\theta_*=\arctan(2/\sqrt3)-\pi/6$ (the same as the acute angle of $Q$).
Their union $U_1\cup U_2$ covers $Q$ and is contained in $K_-$, while $\partial U_1\cup\partial U_2$ contains $\vee$.
Therefore,
      \begin{align*}
h_F^{1-1/q}\| \partial_\bn v\|_{L_q(\vee)}
&\le h_F^{1-1/q} (\|\partial_\bn v\|_{L_q(\partial U_1)}+\|\partial_\bn v\|_{L_q(\partial U_2)})
\\
&\le C_{q,\theta_*,\mu}\sum_{j=1}^2\Big(\|\nabla v\|_{L_2(U_j)}
+h_F^{1-\mu} \big\||\bx-\bp_j|^\mu D^2 v\big\|_{L_2(U_j)}\Big)
\\
&\le 2C_{q,\theta_*,\mu}\Big(\|\nabla v\|_{L_2(K_-)}
+c^{-\mu}h_F^{1-\mu} \big\|\dist(\bx,\deK_-)^\mu   D^2 v\big\|_{L_2(K_-)}\Big),
\end{align*}
which gives the first assertion of the lemma.

Moreover, Lemma~\ref{lem:wedge} shows that $v\in C^{0,\alpha}(\overline Q)$, and, in particular, the Dirichlet trace of $v$ on $\vee$ belongs to $L_q(\vee)$ for all $q\in[1,\infty]$.  

Similarly, for each internal face $F\in\cF_I$, the wedge $\tru\in K_+$ contains a nonconvex quadrilateral $Q'$ whose boundary
contains $\wedge$, and such that all $\bx\in Q'$ satisfy an inequality analogous to \eqref{e:DistXP}.
Such quadrilaterals for the mesh $\cT_2'$ are shaded in yellow in Figure~\ref{fig:Boomerangs}(b).
So the second estimate in the assertion can be obtained with a similar argument.
\end{proof}

We prove a Poincar\'e inequality in the wedge $\trd$, which we use in treatment of the boundary terms in the proof of the continuity result in Lemma \ref{lem:Cont}.

\begin{lem}[Poincar\'e inequality in a wedge]
 \label{lem:Poincare}
Let $\cT$ be an LQU mesh of $\Omega$ with faces $\cF$.
For any face $F\in\cF$ the domain $\trd$ is an $H^1$ extension domain, i.e.,\ there exists a bounded linear operator $E_\trd:H^1(\trd)\to H^1(\R^2)$ with $(E_\trd v)|_\trd=v$ for all $v\in H^1(\trd)$. 
Hence $H^1(\trd)$ is compactly embedded in $L_2(\trd)$. 
Furthermore, there exists 
 $C_{P}>0$ such that the following Poincar\'e inequality holds:
\[ 
\frac{1}{h_F}\|v\|_{L_2(\trd)}\leq C_P\Big(\|\nabla v\|_{L_2(\trd)} + \frac{1}{h_F^{d/2}}\|v\|_{\IL_2(F)} \Big), \qquad \forall v\in H^1(\trd).
\]
\end{lem}
\begin{proof}
Let $v\in H^1(\trd)$. Let $K\in \cT $ be the element containing $\trd$. Define $\widetilde{v}\in L_2(K)$ in the following way. Reflect $v$ evenly across both straight sides of $\trd$, then reflect evenly to the other half of $K$ (see Figure~\ref{fig:Poincare}). The function $\widetilde{v}$ so defined belongs to $H^1(K)$, as follows by a standard argument, testing the piecewise gradient against a vector field $\phi\in(C^\infty_0(K))^2$, observing that $\phi$ is supported in a Lipschitz polygonal prefractal approximation $\widetilde{K}$ of $K$ (see Figure \ref{fig:SnowflakePrefractals} or, e.g., \cite{caetano2019density}), applying the divergence theorem in each of the intersections of $\widetilde{K}$ with the six wedges making up $K$, and deducing that the piecewise gradient coincides with the weak gradient of $\widetilde{v}$, and that $\|\widetilde{v}\|_{H^1(K)}=\sqrt{6}\|v\|_{H^1(\trd)}$. We then use the fact that $K$ is an $H^1$ extension domain \cite[p.~73]{Jones} to deduce the first assertion of the lemma. The compactness of the embedding of $H^1(\trd)$ in $L_2(\trd)$ follows by \cite[Prop.~7.5]{sayas}. 
Alternatively, this compactness can be proved using \cite{chua2013compact} because $\trd$ is a John domain (cf.\ the proof of Lemma \ref{lem:Approx}).

To prove the Poincar\'e inequality we use the general approach provided by \cite[Lemma 4.1.3]{ziemer}. In the notation of \cite{ziemer}, $X_0=L_2(\trd)$, $X=H^1(\trd)$, $Y$ is the space of constant functions on $\trd$, $\|\cdot\|_0=\|\cdot\|_{L_2(\trd)}$, $\|\cdot\|_1=\|\nabla(\cdot)\|_{L_2(\trd)}$, and $L:X\to Y$ is the projection defined by $Lv:=(\int_F v \,\rd\cH^d)/\cH^d(F)$. Then \cite[Lemma 4.1.3]{ziemer} gives that $\|v-Lv\|_{L_2(\trd)}\leq C \|\nabla v\|_{L_2(\trd)}$ for some constant $C>0$. Hence 
\[
\|v\|_{L_2(\trd)}\leq \|v-Lv\|_{L_2(\trd)} + \|Lv\|_{L_2(\trd)}
\leq C\|\nabla v\|_{L_2(\trd)} + \frac{|\trd|^{1/2}}{\cH^d(F)^{1/2}}\|v\|_{\IL_2(F)},
\]
where we used the fact that $|Lv|\leq \|v\|_{\IL_2(F)}/\cH^d(F)^{1/2}$ by the Cauchy--Schwarz inequality. 
The powers of $h_F$ in the assertion follow by a norm scaling argument.  
   \end{proof}

\begin{figure}[htb]\centering
\includegraphics[width=.8\textwidth]{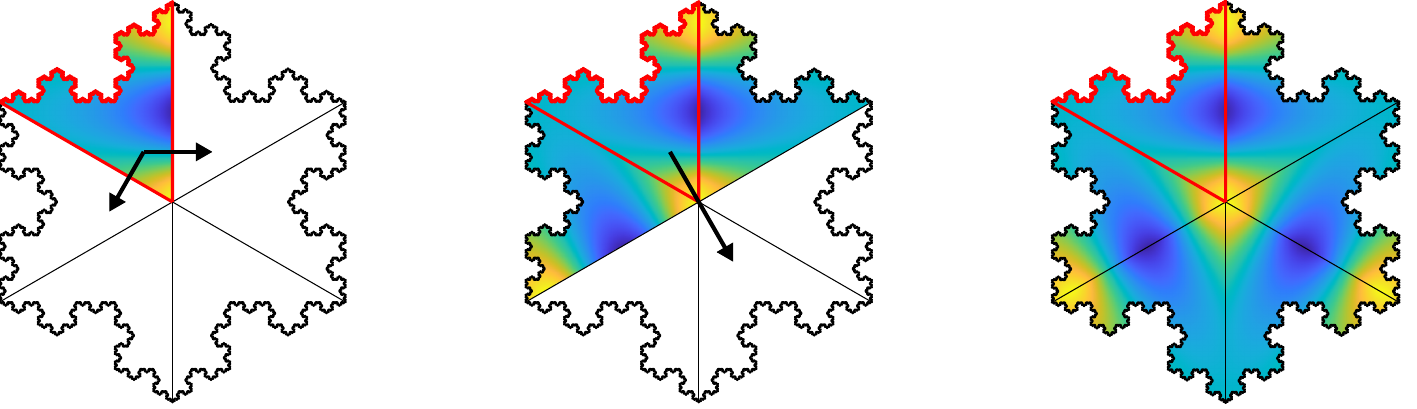}
\caption{Every $H^1(\trd)$ function can be extended to an $H^1(K)$ function by even reflections.}
\label{fig:Poincare}
\end{figure}

The next result, which may be of independent interest, records the fact that for sufficiently small exponent $\mu$, the weighted space in \eqref{e:H22mu} is embedded in $W^2_p\OO$ for a range of $p$, and hence in $C^{0,\alpha}(\overline{\Omega})$ for a range of $\alpha$.

\begin{lem}[Embeddings of the weighted space]\label{lem:Embed}
If $\mu<1-\frac d2\approx0.369$, then the following continuous embeddings hold:
\begin{equation*}
H^{2,2}_\mu(\Omega) \subset W^2_p\OO \subset C^{0,\alpha}(\overline\Omega) \quad 
  \forall p,\alpha \text{ such that }
1\le p <\frac{2(2-d)}{2-d+2\mu},\; 0<\alpha=2-\frac2p<\frac{2-d-2\mu}{2-d}.
\end{equation*}
 \end{lem}
\begin{proof}
Let $w\in H^{2,2}_\mu(\Omega)$.
By the standard H\"older inequality
($\int_\Omega |fg|\dx\le\|f\|_{L_r\OO}\|g\|_{L_{r'}\OO}$ with $\frac1r+\frac1{r'}=1$, here with $f=|D^2 w|^p\delta^{\mu p}$, $g=\delta^{-\mu p}$, $r=\frac2p$) we have
\begin{align}\label{e:HolderE}
\left(\int_\Omega |D^2 w|^p \,\dx \right)^{1/p} 
\leq \left(\int_\Omega |D^2 w|^2 \delta^{2\mu} \,\dx \right)^{1/2} \left(\int_\Omega \delta^{-q\mu} \,\dx \right)^{1/q},
\end{align}
when $1\le p<2$ and $2\le q<\infty$ with $\frac1q + \frac12 =\frac1p$, i.e.\ $q=\frac{2p}{2-p}$. 
Thus $w\in W^2_p(\Omega)$, whenever the second integral on the right-hand side of \eqref{e:HolderE} is finite. 
By \cite[Lem.~2.2]{LapidusPearse2006} (and see the more general related result in \cite[Rem.~2.9]{carvalho2012hausdorff}),
 there is $C_{\mathrm{LP}}>0$ such that 
$|\{\bx\in \Omega, \delta(\bx)<\epsilon\}|\le C_\mathrm{LP}\epsilon^{2-d}$ for small $\epsilon$, with $|\cdot|$ denoting the Lebesgue measure.
Then 
$$
\int_\Omega \delta^{-q\mu}\dx
=\sum_{i=0}^\infty\int_{\{\bx\in \Omega,\ 2^{-i-1}\le\delta(\bx)<2^{-i}\}}\delta^{-q\mu}\dx
\le C_\mathrm{LP}2^{q\mu}\sum_{i=1}^\infty 2^{(q\mu-2+d)i},
$$
and this sum is finite if $q<\frac{2-d}\mu$ (which is possible since $\mu<1-\frac d2$), i.e.\ $1\le p<\frac{2(2-d)}{2-d+2\mu}<2$.
Thus $w\in W^2_p\OO$.

By \cite[p.~73]{Jones}, $\Omega$ is an $(\epsilon,\delta)$ locally uniform domain.
Thus by  \cite{Rogers,Jones} there is a continuous extension operator $E_\Omega:W^2_p\OO\to W^2_p(\R^2)$ with $(E_\Omega w)|_\Omega=w$ for all $w$.
So $W^2_p\OO$ coincides with the space of the restrictions to $\Omega$ of $W^2_p(\R^2)$ functions, denoted $W^2_p(\overline\Omega)$ in \cite[Def.~1.3.2.4]{Grisvard85}, and is embedded in $C^{0,\alpha}(\overline\Omega)$ for $\alpha=2-\frac2p\in(0,\frac{2-d-2\mu}{2-d})$ by standard embeddings such as \cite[eq.~(1,4,4,6)]{Grisvard85}.
\end{proof}

\section{Integration formulas on fractal domains}
\label{app:Integration}

\begin{figure}[t]
 \centering
  \begin{subfigure}[t]{.4\textwidth}
\centering
\includegraphics[height=50mm]{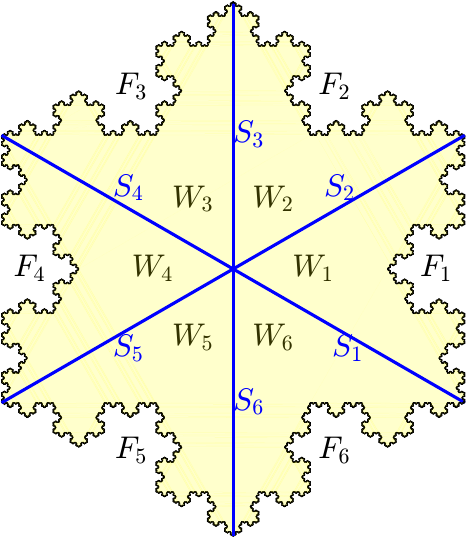} 
\caption{{}}
\end{subfigure}
 \begin{subfigure}[t]{.4\textwidth}
\centering
\includegraphics[height=50mm]{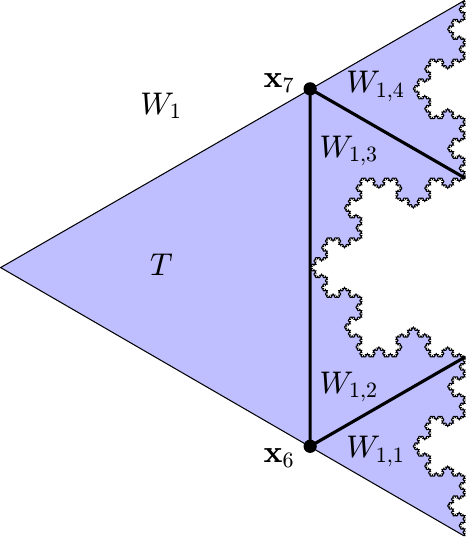}
\caption{{}}
\end{subfigure}
    \caption{(a) Decomposition of the snowflake $\Omega$ into six wedges $W_1,\ldots,W_6$.  (b) Decomposition of the wedge $W_1$ into the equilateral triangle $T$ and four wedges $W_{1,1},\ldots,W_{1,4}$.}
\label{Fig2}
\end{figure}

In this appendix we present exact  formulas for the integration of polynomials over the fractal reference domains used in the implementation of our SIP-DG-FEM, as discussed in \S\ref{s:Quad}. 
We obtain our integration formulas using self-similarity and homogeneity arguments similar to those used in e.g.,\ \cite{barnsley1985iterated, vrscay1989iterated,
strichartz2000evaluating, disjoint,nondisjoint,joly2024high}. 
We present only those formulas required to implement our method for polynomial degree $p=1$ and $p=2$. But higher order formulas permitting implementation of our method for larger $p$ values could be obtained by the same methodology.

We start by considering integrals of polynomials over the reference wedges $W_1,\ldots,W_6$ illustrated in Figure \ref{Fig2}(a). 
We first note that, for each $i$, $W_{i+1}$ is the rotation of $W_i$ by an angle $\pi/3$ about the origin, i.e.,\ 
$R^2:W_i\to W_{i+1}$ is a bijection, where $R$ was defined in \eqref{e:smDef}. Explicitly,
\[ R^2(\bx) =  
\begin{pmatrix}c & -s\\s & c\end{pmatrix}
\bx  , \ \qquad c=\cos\frac\pi3=\frac12,\qquad s=\sin\frac\pi3=\frac{\sqrt3}2.
\]

\begin{lem}\label{lem:WedgeQuad}
Let $W_1,\ldots,W_6\subset\Omega$ be the wedges illustrated in Figure \ref{Fig2}(a). 
For an integrable function $f$ let
 $I_i[f]:=\int_{W_i}f \dx $ for $i=1,\ldots,6$.
Then 
\begin{align}
\label{e:WedgeQuad}
I_1[1]=\frac{\sqrt{3}}{5},\quad I_1[x]=\frac{11}{60},\quad I_1[y]=0,\quad I_1[x^2]=\frac{281\sqrt{3}}{4400},\quad I_1[xy]=0,\quad I_1[y^2]=\frac{39\sqrt{3}}{4400},
\end{align}
and, for $i=2,\ldots 6$,
\begin{align}
I_{i}[1] &= 
I_{i-1}[1],\notag\\ 
I_{i}[x] &= 
cI_{i-1}[x] - sI_{i-1}[y],\notag\\ 
I_{i}[y] &= 
sI_{i-1}[x] + cI_{i-1}[y],\notag\\ 
I_{i}[x^2] &= 
c^2I_{i-1}[x^2] - 2csI_{i-1}[xy] + s^2I_{i-1}[y^2],\notag\\ 
I_{i}[xy] &= 
csI_{i-1}[x^2] +(c^2-s^2)I_{i-1}[xy] -csI_{i-1}[y^2],\notag\\
I_{i}[y^2] &= 
s^2I_{i-1}[x^2] + 2csI_{i-1}[xy] + c^2I_{i-1}[y^2].
\label{e:WedgeQuad2}
\end{align}
\end{lem}
\begin{proof}
The relations \eqref{e:WedgeQuad2} are a simple consequence of the fact that $R^2:W_i\to W_{i+1}$ is a bijection. The point of these relations is that they reduce the calculation of integrals over $W_i$, $i=2,\ldots 6$, to the calculation of integrals \eqref{e:WedgeQuad} over $W_1$, which we now turn our attention to.

 Using the fact that 
$|\Omega| = \frac{6\sqrt{3}}{5}$,
we have that $I_i[1]=\frac{1}{6}|\Omega|=\frac{\sqrt{3}}{5}$,$i=1,\ldots,6$.

To calculate integrals of higher degree polynomials, we use the decomposition of $W_1$ illustrated in Figure \ref{Fig2}(b),  which expresses $W_1$ as a union of an equilateral triangle $T$ of side length $2/3$, and four scaled, rotated and translated copies of $W_1$, which we call $W_{11}, W_{12}, W_{13}, W_{14}$, with area $1/9$ that of $W_1$. 
For a function $f$ we write $I_{1i}[f]:=\int_{W_{1i}}f \dx $ for $i=1,\ldots,4$, and $I_T[f]:=\int_{T}f \dx $.

For the degree one polynomials, we have
\[ I_1[x] = I_T[x] + 2I_{11}[x]+2I_{12}[x] \qquad \text{and} \qquad I_1[y]=0.\]
Furthermore, defining $\bx^*:=\bx_6 = (x^*,y^*)= (1/\sqrt{3},-1/3)$, 
\begin{align*}
I_{11}[x] &= \int_{W_{11}}x \dx  = \frac{1}{9}\int_{W_1}\left(x^* + \frac{x}{3}\right) \dx  = \frac{1}{9}\left(x^* I_1[1] + \frac{1}{3}I_1[x]\right) = \frac{1}{9}\left(\frac{1}{5} + \frac{1}{3}I_1[x]\right),\\
I_{12}[x] &= \int_{W_{12}}x \dx  = \frac{1}{9}\int_{W_2}\left(x^* + \frac{x}{3}\right)\dx  = \frac{1}{9}\left(x^* I_2[1] + \frac{1}{3}I_2[x]\right) = \frac{1}{9}\left(\frac{1}{5} + \frac{1}{6}I_1[x]\right),
\end{align*}
and
\[ I_T[x] = 2\int_0^{1/\sqrt{3}}\int_0^{x/\sqrt{3}}x\, \rd y\rd x = \frac{2}{\sqrt{3}}\int_0^{1/\sqrt{3}} x^2 \,\rd x = \frac{2}{27}.\]
Hence 
\[ I_1[x] = \frac{2}{27} + \frac{4}{45} + \frac{1}{9}I_1[x] = \frac{22}{135} + \frac{1}{9}I_1[x],\]
and solving for $I_1[x]$ gives
\[ I_1[x] = \frac{9}{8}\cdot\frac{22}{135} = \frac{11}{60}.\]

For the degree two polynomials we have
\begin{align*}
&I_1[x^2] = I_T[x^2] + 2I_{11}[x^2]+2I_{12}[x^2],\qquad I_1[y^2] = I_T[y^2] + 2I_{11}[y^2]+2I_{12}[y^2],\qquad \text{and}\quad I_1[xy]=0.
\end{align*}
Furthermore, 
\begin{align*}
\label{}
I_{11}[x^2] = \int_{W_{11}}x^2 \dx  = \frac{1}{9}\int_{W_1} \big(x^* + \frac{x}{3}\big)^2 \dx  &= \frac{1}{9}\left((x^*)^2 I_1[1] + \frac{2}{3}x^* I_1[x] + \frac{1}{9}I_1[x^2]\right)\\
 &= \frac{1}{81}\left(\frac{29}{10\sqrt{3}} + I_1[x^2]\right),\\
I_{12}[x^2] = \int_{W_{12}}x^2 \dx  = \frac{1}{9}\int_{W_2} \big(x^* + \frac{x}{3}\big)^2 \dx 
&= \frac{1}{9}\left((x^*)^2 I_2[1] + \frac{2}{3}x^* I_2[x] + \frac{1}{9}I_2[x^2]\right)\\ 
 &= \frac{1}{108}\left(\frac{47}{15\sqrt{3}} + \frac{1}{3}I_1[x^2]+ I_1[y^2] \right),\\
I_{11}[y^2] = \int_{W_{11}}y^2 \dx  = \frac{1}{9}\int_{W_1}\big(y^* + \frac{y}{3}\big)^2 \dx &= \frac{1}{9}\left((y^*)^2 I_1[1] + \frac{2}{3}y^* I_1[y] + \frac{1}{9}I_1[y^2]\right)\\
 &= \frac{1}{81}\left(\frac{\sqrt{3}}{5} + I_1[y^2]\right),\\
I_{12}[y^2] = \int_{W_{12}}y^2 \dx  = \frac{1}{9}\int_{W_2}\big(y^* + \frac{y}
{3}\big)^2  \dx 
&= \frac{1}{9}\left((y^*)^2 I_2[1] + \frac{2}{3}y^* I_2[y] + \frac{1}{9}I_2[y^2]\right)\\ 
 &= \frac{1}{324}\left(\frac{1}{5\sqrt{3}} + 3I_1[x^2]+ I_1[y^2] \right),
\end{align*}
and
\[ I_T[x^2] = 2\int_0^{1/\sqrt{3}}\int_0^{x/\sqrt{3}}x^2\, \rd y\rd x = \frac{2}{\sqrt{3}}\int_0^{1/\sqrt{3}} x^3 \,\rd x = \frac{1}{18\sqrt{3}},\]
 \[ I_T[y^2] =  2\int_0^{1/\sqrt{3}}\int_0^{x/\sqrt{3}}y^2\, \rd y\rd x  = \frac{2}{9\sqrt{3}}\int_0^{1/\sqrt{3}} x^3 \,\rd x = \frac{1}{162\sqrt{3}}.\]
Hence 
   \begin{align*}
     I_1[x^2] = \frac{1}{27\sqrt{3}} + \frac{5}{162}I_1[x^2] + \frac{1}{54} I_1[y^2]\qquad\text{and}\qquad
I_1[y^2] = \frac{1}{45\sqrt{3}} + \frac{1}{54}  I_1[x^2] + \frac{5}{162}I_1[y^2]
\end{align*}
 satisfy the system
 \[ \left(\begin{array}{cc}
 157 & -3\\
 -3 &   157
 \end{array}\right)
\left(\begin{array}{c}
 I_1[x^2] \\
 I_1[y^2] 
 \end{array}\right) 
=
\left(\begin{array}{c}
 10\sqrt{3} \\
 \frac{6\sqrt{3}}{5} 
 \end{array}\right),
 \]
which has the solution
\[ I_1[x^2] = \frac{281 \sqrt{3}}{4400}, \qquad I_1[y^2] =  \frac{39\sqrt{3}}{4400}. \]
\end{proof}

Integrals of monomials over the full snowflake $\Omega$ can also be computed in closed form.

\begin{lem}
\label{lem:SnowQuad}
Let $\Omega$ be the Koch snowflake. Then 
\begin{align*}
&\int_\Omega 1 \dx =\frac{6\sqrt3}5, \qquad 
\int_\Omega x^2 \dx =
 \int_\Omega y^2 \dx = \frac{12\sqrt3}{55},
\\
&\int_\Omega x^4\dx=\int_\Omega y^4\dx=\frac{27\sqrt3}{319}, \qquad
\int_\Omega x^2y^2\dx=\frac{9\sqrt3}{319},
\end{align*}
and $\int_\Omega x^ay^b \dx=0$ whenever $a,b\in \N_0$ with at least one of $a,b$ odd.
\end{lem}
\begin{proof}
The first formula is just the area of $\Omega$, which has already been reported earlier in the paper. 
The integrals of the degree 2 and degree 4 polynomials can be derived by self-similarity arguments, using the decomposition \eqref{e:Decomp} and the recursion discussed in \cite[Thm~9]{barnsley1985iterated}. (The degree 2 results could alternatively be derived by summing the formulas provided by Lemma \ref{lem:WedgeQuad} for the corresponding integrals over the wedges $W_1,\ldots,W_6$.) 
That $\int_\Omega x^ay^b \dx=0$ whenever $a,b\in \N_0$ with at least one of $a,b$ odd follows by the symmetry of $\Omega$.
\end{proof}
              
Finally, we present formulas for the integration of monomials on the Koch curve $\Gamma$, with respect to the Hausdorff measure $\cH^d$ for $d=\log{4}/\log{3}$. 
 \begin{lem}\label{lem:KochQuad}
Let $\Gamma$ be the Koch curve defined in \S\ref{ss:Koch}, and let $d:=\log{4}/\log{3}$. 
For an $\cH^d$-integrable function $f$, let  $J[f]:=\int_{\Gamma}f(x)\,\rd \cH^d(x)$. 
Then 
\begin{align*}
\label{}
J[1]&=1,\\
J[x]&=\frac{1}{2},\quad
J[y]=\frac{1}{6\sqrt{3}},\\
J[x^2]&=\frac{19}{60},\quad
J[xy]=\frac{1}{12\sqrt{3}},\quad
J[y^2]=\frac{1}{60},\\
J[x^3]&=\frac{9}{40},\quad
J[x^2y]=\frac{13}{280\sqrt{3}},\quad
J[xy^2]=\frac{1}{120},\quad
J[y^3]=\frac{1}{168\sqrt{3}},\\
J[x^4]&=\frac{92983}{542640},\quad
J[x^3y]=\frac{47}{1680\sqrt{3}},\quad
J[x^2y^2]=\frac{47}{10640},\quad
J[xy^3]=\frac{1}{336\sqrt{3}},\quad
J[y^4]=\frac{83}{108528}.
\end{align*}
\end{lem}
\begin{proof}
The formula $J[1]=1$ is just a restatement of our choice of normalisation for $\cH^d$ made in \S\ref{ss:Traces}.
The remaining formulas can be derived by self-similarity arguments, using the recursion discussed in \cite[Thm~9]{barnsley1985iterated}. The key to this is the fact (see, e.g., \cite[\S3.1]{barnsley1985iterated}) that, for measurable $f$, 
     \[J[f]=
 \frac{1}{4}\sum_{i=1}^4 J[f\circ t_i],\]
where $t_1,\ldots,t_4$ are the similarities defined in \eqref{e:tmDef}. 
A Mathematica notebook carrying out the derivations can be found at \url{https://github.com/SergioG5/FractalDG}.
\end{proof}

\section*{Acknowledgements}
The authors are grateful to Maryna Kachanovska, Carlo Marcati and Iain Smears for helpful discussions in relation to the proofs of Lemmas \ref{lem:Approx}, \ref{lem:wedge} and \ref{lem:Poincare} respectively.
We also acknowledge insightful preliminary investigations into mesh refinement algorithms carried out by Yuvan Raja of the Royal Grammar School, Guildford, UK, as part of an Original Research in Science (ORIS) project.

DH was supported by the EPSRC grant EP/V053868/1.  
   DH and AM thank the Isaac Newton Institute for Mathematical Sciences for support and hospitality during the programme ``Mathematical Theory and Applications of Multiple Wave Scattering'', 
 supported by EPSRC grant EP/R014604/1. 
DH and AM were supported by the Swedish Research Council under grant no.~2021-06594 while in residence at Institut Mittag-Leffler in Djursholm, Sweden during the fall semester of 2025. 
AM and SG are members of GNCS-INdAM.
 AM acknowledges support by the PRIN project ``ASTICE'' (202292JW3F), funded by the European Union--NextGenerationEU.  

\begin{table}[htbp]\centering
\begin{tabular}{|l|l|l|}
\hline
$\Gamma,\Omega$& Koch curve and snowflake& \S\ref{ss:Koch}\\
$\dimH(\cdot),d=\dimH(\Gamma)$& Hausdorff dimension& \S\ref{ss:Koch}\\
$H^1_0(\Omega), H^1(\Omega), W^1_p(\Omega)$ & Sobolev spaces & \S\ref{ss:BVP}--\ref{ss:H22}\\
$C^{0,\alpha}(\overline\Omega),H^{2,2}_\mu(\Omega), D^2$ & H\"older and weighted Sobolev space, second derivatives & \S\ref{ss:H22}\\
$\delta,\mu,\mu_*$ & Boundary distance function, integrability exponents & \S\ref{ss:H22}\\
$C_{\mathrm{stab},\mu}$ & Poisson BVP stability constant in weighted norm & \eqref{e:Stab}\\
$s_1,\ldots,s_7, t_1,\ldots, t_4,R$ & Contractions defining $\Omega$ and $\Gamma$, rotation matrix &\S\ref{ss:DecompOmega}--\ref{ss:DecompKoch}\\
$\cT, K, h_K,\ell(K)$ & Mesh, element, element diameter, contraction number &\S\ref{ss:Mesh}\\
$\cF_I,\cF_B,\cF$ &Set of interior, boundary and all faces &\S\ref{ss:Faces}\\
$F, K_\pm, h_F$ & Face, elements adjacent to a face, face diameter &\S\ref{ss:Faces}\\
$\cT_\ell, \cT_\ell',\cT'_{\ell,\ell^*}$ & Families of LQU meshes &\S\ref{ss:MeshSeq}\\
$\bv_1^K,\ldots,\bv_6^K,\widehat\delta_K$ & Element vertices and vertex-boundary distance &\S\ref{ss:MeshSeq}\\
$\IP^p(K), V_h,\IP^p(\cT)
$ & Polynomial and piecewise-polynomial spaces &\S\ref{ss:PPoly}\\
$\cH^d, \IL_2(F), \|\cdot\|_{\IL_2(F)}$ & Hausdorff measure, Lebesgue space and norm  &\S\ref{ss:Traces}\\
$C_\Tr$ & Trace continuity constant &\S\ref{ss:Traces}\\
$\trd,\tru$ & Wedges in $K_\mp$ adjacent to $F$, &\S\ref{s:Wedge}\\
$\loz,\lozenge=\partial\loz$ & Union of two wedges and its boundary &\S\ref{s:Wedge}\\
$S_1^\pm,S_2^\pm,\vee,\wedge$ & Segments bounding the wedges $\trd$, $\tru$, their unions&\S\ref{s:Wedge}\\
$\bn$ & Outward-pointing normal vector on $\lozenge$ &\S\ref{s:Wedge}\\
$\cI_\trd(w,v), \cI_\tru(w,v)$ & Wedge proxies for Dirichlet--Neumann duality on $F$ & \eqref{e:IdownIup}\\
$C_p$ & Norm equivalence constant & \eqref{e:DiscreteBoundOnI}\\
$v_h^\pm$ & Global polynomial such that $v_h^\pm|_{K_\pm}=v_h|_{K_\pm}$ & \S\ref{s:DG}\\
$\aDG(\cdot,\cdot), \cL(\cdot),\eta$ & SIP-DG bilinear and linear forms, penalty parameter &\S\ref{s:DG}, \eqref{e:aDefGen}\\
$u\DG$ & SIP-DG solution &\eqref{e:DG}\\
$\|\cdot\|\DG$, $\|\cdot\|\DGp$ & SIP-DG norms & \S\ref{s:Coercivity}, \eqref{e:Norms}\\
$V^\mu$, $q$ & Energy space, normal trace integrability exponent & \eqref{e:V}, \eqref{e:qmu}\\
$N_1^\trd,N_1^\tru,N_2^\trd,N_2^\tru,X$ & $H^2$-type and $H^1$-type norms on wedges and on $\Omega$ &\S\ref{s:Continuity}\\
$C_{p,q,1},C_{p,q,2},\ast_K$ & Inverse inequality constants and segment star in $K$ &\S\ref{s:Continuity}\\
$\CDG,\CDGp$ & SIP-DG continuity constants & \eqref{e:Cont}-\eqref{e:ContPlus}\\
$\delta_K$ & Element-boundary distance & \eqref{e:deltaK}\\
$C_\mu,\hat{C}_{0}, \hat{C}_\mu$ & Approximation bound constants & (\ref{e:Approx})-(\ref{e:H22approxK})\\
$\Pi_\cT$ & Piecewise-linear projection & Prop.~\ref{prop:ApproxDG}\\
$h,\zeta$ & Mesh size, dual problem solution & Cor.~\ref{c:L2}\\
$\psi_K,\xi_F, R_K,R_F$ & Element and face similarities and rotations & (\ref{e:psiK})-(\ref{e:Face})\\
$\widehat K,W_i,S_i,F_i,\widehat\trd,\widehat\vee,\widehat F_\pm $ & Reference element, wedges, faces and sides & \S\ref{ss:RefDom}\\
$\Np,\NT,\widehat{\phi_i},\phi_i^m,\chi_K$ & Local dimension, element number, basis functions &\S\ref{ss:Basis}\\
$N, \su, \sfb$ & DOF number, solution and right-hand side vectors &\S\ref{ss:LSE}\\
$\ADG, \sG, \sC, \sP,g_h,c_h,p_h$ & Galerkin matrices and bilinear forms & \S\ref{ss:LSE}\\
$C^{(F)}_{mn},P^{(F)}_{mn}$ & Galerkin matrix blocks& \S\ref{ss:C}--\ref{ss:P}\\
$\lambda,\lambda_h,M$ & Continuous and discrete eigenvalues, mass matrix & \S\ref{ss:eigen}\\
$U,\theta_*,C_{q,\theta_*,\mu},C_{\mathrm{KMR}}$ & Sector, opening, bounding constants & Lem.~\ref{lem:wedge}\\
$\widetilde C_q, Q,\bp_0,\bp_1,\bp_2,\bp_*$ & Trace constant for $\vee/\wedge$, quadrilateral, vertices & Lem.~\ref{lem:dnu2}\\
$E_\trd,C_P$ & Wedge extension operator and Poincar\'e constant & Lem.~\ref{lem:Poincare}\\
$E_\Omega,C_{\mathrm{LP}}$ & Snowflake extension operator, tube constant & Lem.~\ref{lem:Embed}\\
$I_i[f],I_{1i}[f],T,W_{1i},J[f]$ & Integrals on wedges and faces, wedge subsets
&App.~\ref{app:Integration}
\\\hline
\end{tabular}
\caption{List of the notation used in the paper.}
\label{tab:Notation}
\end{table}

\addcontentsline{toc}{section}{References}
\bibliography{FractalDG_bib}

@article {hutchinson1981fractals,
    AUTHOR = {Hutchinson, John E.},
     TITLE = {Fractals and self-similarity},
   JOURNAL = {Indiana Univ. Math. J.},
  FJOURNAL = {Indiana University Mathematics Journal},
    VOLUME = {30},
      YEAR = {1981},
    NUMBER = {5},
     PAGES = {713--747},
      ISSN = {0022-2518},
   MRCLASS = {49F20 (00A69 28A12 58C27)},
  MRNUMBER = {625600},
MRREVIEWER = {F. J. Almgren, Jr.},
       URL = {https://doi.org/10.1512/iumj.1981.30.30055},
}

@book{Fal,
author = "Falconer, K.",
title="Fractal Geometry: Mathematical Foundations and Applications",
publisher="Wiley, 3rd ed.",
year= "2014"}

@article{capitanelli2015weighted,
  title={Weighted estimates on fractal domains},
  author={Capitanelli, Raffaela and Vivaldi, Maria Agostina},
  journal={Mathematika},
  volume={61},
  number={2},
  pages={370--384},
  year={2015}
}

@misc{CaChHe25,
title={{Properties of IFS attractors with non-empty interiors, related rough domains, and associated function spaces and scattering problems}},
author={Caetano, A. M. and Chandler-Wilde, S. N. and Hewett, D. P.},
note={arXiv:2511.15213}
}

@book{cangiani2017hp,
  title={{hp-Version Discontinuous Galerkin Methods on Polygonal and Polyhedral Meshes}},
  author={Cangiani, Andrea and Dong, Z and Georgoulis, Emmanuil H and Houston, Paul},
  year={2017},
  publisher={Springer}
}

@article{arnold1982interior,
  title={An interior penalty finite element method with discontinuous elements},
  author={Arnold, Douglas N},
  journal={SIAM J. Numer. Anal.},
  volume={19},
  number={4},
  pages={742--760},
  year={1982}
}

@Article{caetano2019density,
  author  = {Ant\'onio M. Caetano and David P. Hewett and Andrea Moiola},
  title   = {{Density results for Sobolev, Besov and Triebel-Lizorkin spaces on rough sets}},
  journal = {J. Funct. Anal.},
  volume = {281},
  number ={3},
  pages={109019},
  year    = {2021},
}

@book{di2011mathematical,
  title={{Mathematical Aspects of Discontinuous Galerkin Methods}},
  author={Di Pietro, Daniele Antonio and Ern, Alexandre},
  year={2011},
  publisher={Springer}
}

@article {BGhpCurv,
    AUTHOR = {Babu{\v{s}}ka, I. and Guo, B. Q.},
     TITLE = {The {$h$}-{$p$} version of the finite element method for domains with curved boundaries},
   JOURNAL = {SIAM J. Numer. Anal.},
  FJOURNAL = {SIAM Journal on Numerical Analysis},
    VOLUME = {25},
      YEAR = {1988},
    NUMBER = {4},
     PAGES = {837--861},
      ISSN = {0036-1429},
     CODEN = {SJNAAM},
   MRCLASS = {65N30},
  MRNUMBER = {954788 (89i:65111)},
       DOI = {10.1137/0725048},
       URL = {http://dx.doi.org/10.1137/0725048},
}

@book {Grisvard85,
    AUTHOR = {Grisvard, P.},
     TITLE = {Elliptic Problems in Nonsmooth Domains},
     PUBLISHER = {Pitman},
   ADDRESS = {Boston},
      YEAR = {1985},    
}

@article{Rogers,
  title={Degree-independent {S}obolev extension on locally uniform domains},
  author={Rogers, L. G.},
  journal={J. Funct. Anal.},
  volume={235},
  number={2},
  pages={619--665},
  year={2006},
  publisher={Elsevier}
}

@article{Jones,
  title={Quasiconformal mappings and extendability of functions in {S}obolev spaces},
  author={Jones, P. W.},
  journal={Acta Math.},
  volume={147},
  number={1},
  pages={71--88},
  year={1981},
  publisher={Springer}
}

@article{LapidusPearse2006,
 author = {Lapidus, Michel L. and Pearse, Erin P. J.},
 title = {A tube formula for the {Koch} snowflake curve, with applications to complex dimensions},
 fjournal = {Journal of the London Mathematical Society. Second Series},
 journal = {J. Lond. Math. Soc., II. Ser.},
 issn = {0024-6107},
 volume = {74},
 number = {2},
 pages = {397--414},
 year = {2006},
 language = {English},
 doi = {10.1112/S0024610706022988},
 zbMATH = {5122378},
 Zbl = {1110.26006}
}

@article{Biegert:09,
author={Biegert, M.},
title={On traces of {S}obolev functions on the boundary of extension domains}, 
journal={Proc. Am. Math. Soc.}, 
volume={137}, 
year={2009}, 
pages={4169--4176}
}

@misc{Hinz,
author={Hinz, M. and Chandler-Wilde, S. N. and Hewett, D. P.},
title={Kernels of trace operators via fine continuity},
note={arXiv:2507.04536}
}

@article{farkas2001sobolev,
  title={{Sobolev spaces on non smooth domains and Dirichlet forms related to subordinate reflecting diffusions}},
  author={Farkas, Walter and Jacob, Niels},
  journal={Math. Nachr.},
  volume={224},
  number={1},
  pages={75--104},
  year={2001}
}

@book{kozlov1997elliptic,
  title={Elliptic boundary value problems in domains with point singularities},
  author={Kozlov, VA and Maz'ya, VG and Rossmann, J},
  year={1997},
  publisher={American Mathematical Soc.}
}

@book{ziemer,
title={Weakly Differentiable Functions},
author={Ziemer, WP},
year={1989},
publisher={Springer}
}

@book{sayas,
title={Variational Techniques for Elliptic Partial Differential Equations},
author={Sayas, F-J and Brown, TS and Hassell, ME},
year={2019},
publisher={CRC Press}
}

@book{McLean,
author = "McLean, W.",
title= "Strongly Elliptic Systems and Boundary Integral Equations",
publisher="CUP",
year="2000"}

@phdthesis{NystromThesis,
  author =   {Nystr{\"o}m, Kaj},
  title =    {Smoothness properties of solutions to Dirichlet problems in domains with a fractal boundary},
  school =   {University of Ume\aa{}},
  year =     {1994}
}

@misc{Hewett25,
title={{Discontinuous piecewise polynomial approximation on non-Lipschitz domains}},
author={Hewett, D. P.},
note={arXiv: 2511.22628}
}

@article {Banjai:2007,
    AUTHOR = {Banjai, L.},
     TITLE = {Eigenfrequencies of fractal drums},
   JOURNAL = {J. Comput. Appl. Math.},
  FJOURNAL = {Journal of Computational and Applied Mathematics},
    VOLUME = {198},
      YEAR = {2007},
    NUMBER = {1},
     PAGES = {1--18},
      ISSN = {0377-0427,1879-1778},
   MRCLASS = {65N25 (30C20 35L05 35P05 65N35)},
  MRNUMBER = {2250385},
MRREVIEWER = {Elsayed\ Mohamed E. Zayed},
       DOI = {10.1016/j.cam.2005.11.015},
       URL = {https://doi.org/10.1016/j.cam.2005.11.015},
}

@article {Castillo:2002,
    AUTHOR = {Castillo, P.},
     TITLE = {Performance of discontinuous {G}alerkin methods for elliptic
              {PDE}s},
   JOURNAL = {SIAM J. Sci. Comput.},
  FJOURNAL = {SIAM Journal on Scientific Computing},
    VOLUME = {24},
      YEAR = {2002},
    NUMBER = {2},
     PAGES = {524--547},
      ISSN = {1064-8275,1095-7197},
   MRCLASS = {65N30},
  MRNUMBER = {1951054},
       DOI = {10.1137/S1064827501388339},
       URL = {https://doi.org/10.1137/S1064827501388339},
}

@article {Neuberger_Sibien_Swift:2006,
    AUTHOR = {Neuberger, J. M. and Sieben, N. and Swift, J. W.},
     TITLE = {Computing eigenfunctions on the {K}och snowflake: a new grid
              and symmetry},
   JOURNAL = {J. Comput. Appl. Math.},
  FJOURNAL = {Journal of Computational and Applied Mathematics},
    VOLUME = {191},
      YEAR = {2006},
    NUMBER = {1},
     PAGES = {126--142},
      ISSN = {0377-0427,1879-1778},
   MRCLASS = {31C20 (35J25 35P30 65N25)},
  MRNUMBER = {2217789},
MRREVIEWER = {Srinivasan\ Kesavan},
       DOI = {10.1016/j.cam.2005.03.075},
       URL = {https://doi.org/10.1016/j.cam.2005.03.075},
}

@article{Mityagin20,
 author = {Mityagin, B. S.},
 title = {The zero set of a real analytic function},
 fjournal = {Mathematical Notes},
 journal = {Math. Notes},
 issn = {0001-4346},
 volume = {107},
 number = {3},
 pages = {529--530},
 year = {2020},
 language = {English},
 doi = {10.1134/S0001434620030189},
 zbMATH = {7215381},
 Zbl = {1442.26039}
}

@article{rosler2024computing,
  title={{Computing eigenvalues of the Laplacian on rough domains}},
  author={R{\"o}sler, Frank and Stepanenko, Alexei},
  journal={Math. Comput.},
  volume={93},
  number={345},
  pages={111--161},
  year={2024}
}

@misc{rosler2024res,
  title={Computing scattering resonances of rough obstacles},
  author={R{\"o}sler, Frank and Stepanenko, Alexei},
  note={arXiv:2402.00846}
}

@misc{claret2024convergence,
  title={{Convergence of layer potentials and Riemann-Hilbert problem on extension domains}},
  author={Claret, Gabriel and Rozanova-Pierrat, Anna and Teplyaev, Alexander},
note={arXiv:2411.03767}
}

@article{cefalo2023fractal,
  title={Fractal mixtures for optimal heat draining},
  author={Cefalo, Massimo and Creo, Simone and Lancia, Maria Rosaria and Rodr{\'\i}guez-Cuadrado, Javier},
  journal={Chaos, Solitons \& Fractals},
  volume={173},
  pages={113750},
  year={2023},
  publisher={Elsevier}
}

@incollection{creo2020magnetostatic,
  title={Magnetostatic problems in fractal domains},
  author={Creo, Simone and Lancia, Maria Rosaria and Vernole, Paola and Hinz, Michael and Teplyaev, Alexander},
  booktitle={Analysis, Probability and Mathematical Physics on Fractals},
  pages={477--502},
  year={2020},
  publisher={World Scientific}
}

@inproceedings{cefalo2021approximation,
  title={{Approximation of 3D Stokes flows in fractal domains}},
  author={Cefalo, Massimo and Creo, Simone and Gallo, Mirko and Lancia, Maria Rosaria and Vernole, Paola},
  booktitle={Fractals in Engineering: Theoretical Aspects and Numerical Approximations},
  pages={27--53},
  year={2021},
  organization={Springer}
}

@article{lancia2020stokes,
  title={{The Stokes problem in fractal domains: Asymptotic behaviour of the solutions}},
  author={Lancia, Maria Rosaria and Vernole, Paola},
  journal={Discrete Contin. Dyn. Syst. - Ser. S},
  volume={13},
  number={5},
  year={2020}
}

@article{lancia2010irregular,
  title={Irregular heat flow problems},
  author={Lancia, Maria Rosaria and Vernole, Paola},
  journal={SIAM J. Math. Anal.},
  volume={42},
  number={4},
  pages={1539--1567},
  year={2010},
  publisher={SIAM}
}

@article{lancia2012numerical,
  title={{Numerical approximation of transmission problems across Koch-type highly conductive layers}},
  author={Lancia, Maria Rosaria and Cefalo, Massimo and Dell Acqua, Guido},
  journal={Appl. Math. Comput.},
  volume={218},
  number={9},
  pages={5453--5473},
  year={2012},
  publisher={Elsevier}
}

@article{capitanelli2010robin,
  title={Robin boundary condition on scale irregular fractals},
  author={Capitanelli, Raffaela},
  journal={Commun. Pure Appl. Anal.},
  volume={9},
  number={5},
  pages={1221--1234},
  year={2010}
}

@article{chandler2018well,
  title={{Well-posed PDE and integral equation formulations for scattering by fractal screens}},
  author={Chandler-Wilde, Simon N and Hewett, David P},
  journal={SIAM J. Math. Anal.},
  volume={50},
  number={1},
  pages={677--717},
  year={2018}
}

@article{chandler2021boundary,
  title={Boundary element methods for acoustic scattering by fractal screens},
  author={Chandler-Wilde, Simon N and Hewett, David P and Moiola, Andrea and Besson, Jeanne},
  journal={Numer. Math.},
  volume={147},
  number={4},
  pages={785--837},
  year={2021},
  publisher={Springer}
}

@article{bannister2022acoustic,
  title={Acoustic scattering by impedance screens/ cracks with fractal boundary: well-posedness analysis and boundary element approximation},
  author={Bannister, Joshua and Gibbs, Andrew and Hewett, David P},
  journal={Math. Models Methods Appl. Sci.},
  volume={32},
  number={02},
  pages={291--319},
  year={2022}
}

@article{lancia2002transmission,
  title={A transmission problem with a fractal interface},
  author={Lancia, Maria Rosaria},
  journal={Z. Anal. Anwend.},
  volume={21},
  number={1},
  pages={113--133},
  year={2002}
}

@article{achdou2016transmission,
  title={A transmission problem across a fractal self-similar interface},
  author={Achdou, Yves and Deheuvels, Thibaut},
  journal={SIAM J. Multiscale Model. Simul.},
  volume={14},
  number={2},
  pages={708--736},
  year={2016},
  publisher={SIAM}
}

@article{disjoint,
  title={Numerical quadrature for singular integrals on fractals},
  author={Gibbs, Andrew and Hewett, David and Moiola, Andrea},
  journal={Numer. Alg.},
  volume={92},
  number={4},
  pages={2071--2124},
  year={2023}
}

@article{nondisjoint,
  title={Numerical evaluation of singular integrals on non-disjoint self-similar fractal sets},
  author={Gibbs, Andrew and Hewett, David P and Major, Botond},
  journal={Numer. Alg.},
  volume={97},
  pages={311--343},
  year={2024}
}

@article{caetano2024hausdorff,
  title={{A Hausdorff-measure boundary element method for acoustic scattering by fractal screens}},
  author={Caetano, Ant{\'o}nio M and Chandler-Wilde, Simon N and Gibbs, Andrew and Hewett, David P and Moiola, Andrea},
  journal={Numer. Math.},
  volume={156},
  number={2},
  pages={463--532},
  year={2024},
  publisher={Springer}
}

@article{joly2024high,
 author = {Joly, Patrick and Kachanovska, Maryna and Moitier, Zo{\"{\i}}s},
 title = {High-order numerical integration on self-affine sets},
 fjournal = {SIAM Journal on Scientific Computing},
 journal = {SIAM J. Sci. Comput.},
 issn = {1064-8275},
 volume = {48},
 number = {4},
 pages = {a1870--a1897},
 year = {2026},
 language = {English},
 doi = {10.1137/24M1697141},
 zbMATH = {8230766}
}

@article{caetano2025integral,
    AUTHOR = {Caetano, Ant\'onio M. and Chandler-Wilde, Simon N. and Claeys,
              Xavier and Gibbs, Andrew and Hewett, David P. and Moiola,
              Andrea},
     TITLE = {Integral equation methods for acoustic scattering by fractals},
   JOURNAL = {Proc. A},
  FJOURNAL = {Proceedings A},
    VOLUME = {481},
      YEAR = {2025},
    NUMBER = {2306},
     PAGES = {Paper No. 20230650},
      ISSN = {1364-5021,1471-2946},
   MRCLASS = {28A80 (65R20 78A45)},
  MRNUMBER = {4877961},
       DOI = {10.1098/rspa.2023.0650},
       URL = {https://doi.org/10.1098/rspa.2023.0650},
}

@article{bandt1991self,
  title={{Self-similar sets 5. Integer matrices and fractal tilings of $\mathbb{R}^n$}},
  author={Bandt, Christoph},
  journal={Proc. Amer. Math. Soc.},
  pages={549--562},
  year={1991}
}

@article{grochenig1994self,
  title={Self-similar lattice tilings},
  author={Grochenig, Karlheinz and Haas, Andrew},
  journal={J. Fourier Anal. Appl.},
  volume={1},
  pages={131--170},
  year={1994}
}

@article{BaVi:14,
title = {Fractal tilings from iterated function systems},
author = {Barnsley, M. and Vince, A.},
year = {2014},
journal = {Discrete Comput. Geom.},
volume = {51},
pages = {729--752}
}

@article{strichartz1999geometry,
  title={{Geometry of self-affine tiles I}},
  author={Strichartz, Robert S and Wang, Yang},
  journal={Indiana U. Math. J.},
  pages={1--23},
  year={1999}
}

@article{Keesling:99,
journal={Topol. Appl.},
volume={94},
year={1999},
pages={195--205},
title={The boundaries of self-similar tiles in $\mathbb{R}^n$},
author={Keesling, J.}
}

@book{brenner2008mathematical,
  title={{The mathematical theory of finite element methods}},
  author={Brenner, Susanne C and Scott, L Ridgway},
  year={2008},
  publisher={Springer}
}

@book{Ciarlet78,
title={The Finite Element Method for Elliptic Problems},
publisher={North-Holland},
year={1978},
author={Ciarlet, P. G.}
}

@article{cangiani2022hp,
  title={{$hp$-version discontinuous Galerkin methods on essentially arbitrarily-shaped elements}},
  author={Cangiani, Andrea and Dong, Zhaonan and Georgoulis, Emmanuil},
  journal={Math. Comp.},
  volume={91},
  number={333},
  pages={1--35},
  year={2022}
}

@article{bagnerini2006finite,
  title={Finite elements for a prefractal transmission problem},
  author={Bagnerini, Patrizia and Buffa, Annalisa and Vacca, Elisa},
  journal={C. R. Math.},
  volume={342},
  number={3},
  pages={211--214},
  year={2006},
  publisher={Elsevier}
}

@article{Bagnerini13,
  author={Bagnerini, P. and Buffa, A. and Vacca, E.},
  title={Mesh generation and numerical analysis of a {G}alerkin method for highly
conductive prefractal layers},
  journal={Appl. Numer. Math.},
  volume={65},
  pages={63--78},
  year={2013},
  }

@article{Achdou07,
  author={Achdou, Y. and Sabot, C. and Tchou, N.},
  title={{Transparent boundary conditions for the Helmholtz equation
in some ramified domains with a fractal boundary}},
  journal={J. Comput. Phys.},
  volume={220},
  pages={712--739},
  year={2007},
  }

@misc{bannister2026scattering,
  title={{Scattering by fractal inhomogeneities via geometry-conforming Galerkin methods for the Lippmann-Schwinger equation}},
  author={Bannister, J and Hewett, D P and Gibbs, A},
  note={arXiv: 2602.05005}
}

@article{capitanelli2016dynamical,
  title={Dynamical quasi-filling fractal layers},
  author={Capitanelli, Raffaela and Vivaldi, Maria Agostina},
  journal={SIAM J. Math. Anal.},
  volume={48},
  number={6},
  pages={3931--3961},
  year={2016}
}

@misc{claret2026helmholtz,
      title={{Helmholtz transmission problem and intrinsic impedance scattering problem on extension domains}}, 
      author={Gabriel Claret},
      note={arXiv:2601.14866} 
      }

@article{strichartz2000evaluating,
  title={Evaluating integrals using self-similarity},
  author={Strichartz, Robert S},
  journal={American Math. Monthly},
  volume={107},
  number={4},
  pages={316--326},
  year={2000}
}

@article{hinz2023boundary,
  title={{Boundary value problems on non-Lipschitz uniform domains: stability, compactness and the existence of optimal shapes}},
  author={Hinz, Michael and Rozanova-Pierrat, Anna and Teplyaev, Alexander},
  journal={Asymptotic Anal.},
  volume={134},
  number={1-2},
  pages={25--61},
  year={2023}
}

@incollection{vrscay1989iterated,
  title={Iterated function systems and the inverse problem of fractal construction using moments},
  author={Vrscay, Edward R and Roehrig, Christopher J},
  booktitle={Computers and Mathematics},
  pages={250--259},
  year={1989},
  publisher={Springer}
}

@article{barnsley1985iterated,
  title={Iterated function systems and the global construction of fractals},
  author={Barnsley, Michael F and Demko, Stephen},
  journal={Proc. Roy. Soc. Lond. A.},
  volume={399},
  number={1817},
  pages={243--275},
  year={1985}
}

@article{sapoval1991vibrations,
  title={Vibrations of fractal drums},
  author={Sapoval, Bernard and Gobron, Th and Margolina, A},
  journal={Phys. Rev. Lett.},
  volume={67},
  number={21},
  pages={2974},
  year={1991}
}

@article{lapidus1996snowflake,
  title={Snowflake harmonics and computer graphics: numerical computation of spectra on fractal drums},
  author={Lapidus, Michel L and Neuberger, John W and Renka, Robert J and Griffith, Cheryl A},
  journal={Int. J. Bifurcation Chaos},
  volume={6},
  number={07},
  pages={1185--1210},
  year={1996}
}

@article{lapidus1995fractals,
  title={Fractals and vibrations: can you hear the shape of a fractal drum?},
  author={Lapidus, Michel L},
  journal={Fractals},
  volume={3},
  number={04},
  pages={725--736},
  year={1995}
}

@article{carvalho2012hausdorff,
  title={{On the Hausdorff dimension of continuous functions belonging to H\"older and Besov spaces on fractal d-sets}},
  author={Carvalho, Abel and Caetano, Ant{\'o}nio},
  journal={J. Fourier Anal. Appl.},
  volume={18},
  number={2},
  pages={386--409},
  year={2012}
}

@article{grebenkov2013geometrical,
  title={{Geometrical structure of Laplacian eigenfunctions}},
  author={Grebenkov, Denis S and Nguyen, B-T},
  journal={SIAM Rev.},
  volume={55},
  number={4},
  pages={601--667},
  year={2013}
}

@article{chua2013compact,
  title={{A compact embedding theorem for generalized Sobolev spaces}},
  author={Chua, Seng-Kee and Rodney, Scott and Wheeden, Richard},
  journal={Pacific J. Math.},
  volume={265},
  number={1},
  pages={17--57},
  year={2013}
  }

@article{harjulehto2012examples,
  title={Examples of fractals satisfying the quasihyperbolic boundary condition},
  author={Harjulehto, Petteri and Kl{\'e}n, Riku},
  journal={arXiv preprint arXiv:1205.1873},
  year={2012}
}

@article{buckleykoskela1995SobolevPoincareJohn,
  title={{Sobolev-Poincar\'e implies John}},
  author={Buckley, S. and Koskela, P.},
  journal={Math. Res. Lett.},
  volume={2},
  pages={577--593},
  year={1995}
}
\bibliographystyle{siam}
\end{document}